\documentclass[10pt]{article}
\usepackage{}
\usepackage{amsfonts}
\usepackage{latexsym,bm}
\usepackage{amssymb}
\usepackage{amsmath}
\usepackage{amsthm}
\usepackage[all]{xy}
\newtheorem{thm}{Theorem}[section]

\newtheorem{rmk}[thm]{Remark}
\newtheorem{lem}[thm]{Lemma}
\newtheorem{prop}[thm]{Proposition}
\newtheorem{cor}[thm]{Corollary}
\newtheorem{e.g.}[thm]{Example}
\newtheorem{defi}[thm]{Definition}
\begin{document}

\title{\Large Pseudo-effective and nef cones on spherical varieties}
\author{\large Qifeng LI}
\date{}
\maketitle

\begin{abstract}
We show that nef cycle classes on smooth complete spherical varieties are effective, and the products of nef cycle classes are also nef. Let $X$ be a smooth projective spherical variety such that its effective cycle classes of codimension $k$ are nef, where $1\leq k\leq \text{dim}(X)-1$. We study the properties of $X$. And we show that if $X$ is a toric variety, then $X$ is isomorphic to the product of some projective spaces; if $X$ is toroidal, then $X$ is isomorphic to a rational homogeneous space; if $X$ is horospherical, $\text{dim}(X)\geq 3$ and $k=2$, then effective divisors on $X$ are nef; if $X$ is horospherical and effective divisors on $X$ are nef, then there is a morphism from $X$ to a rational homogeneous space such that each fiber is isomorphic to the product of some horospherical varieties of Picard number one.
\end{abstract}

\tableofcontents

\section{Introduction}

The positivity of divisors and curves occupy an important position in algebraic geometry. And recently, there are some work on the positivity of subvarieties or more generally, on the positivity of cycles, for example \cite{Pet09}\cite{Fulg11}\cite{DELV11}\cite{Ot13} and \cite{Le13}. In the paper \cite{DELV11}, they defined pseudo-effective and nef cones and studied the properties of these cones on some Abelian varieties. The aim of this paper is to study the pseudo-effective and nef cones on spherical varieties.

We work on the complex number field $\mathbb{C}$. Let $X$ be a smooth complete variety of dimension $n$. Let $A^{*}(X)=\bigoplus\limits_{k=0}^{n}A^{k}(X)$ be the Chow ring of $X$. Let $N^{k}(X)_{\mathbb{R}}$ be the finite dimensional real vector space of numerical equivalence classes of codimension $k$ algebraic cycles on $X$ with real coefficients. Denote by $\text{Eff}^{\, k}(X)\subseteq N^{k}(X)_{\mathbb{R}}$ the cone generated by effective cycles, $\text{Psef}^{\, k}(X)$ the closure of $\text{Eff}^{\, k}(X)$ in $N^{k}(X)_{\mathbb{R}}$. Let $\text{Nef}^{\, k}(X)=\{\eta\in N^{k}(X)_{\mathbb{R}}\mid \eta\cdot \delta \geq 0, \delta\in \text{Psef}^{\, n-k}(X) \}$, $N_{k}(X)_{\mathbb{R}}=N^{n-k}(X)_{\mathbb{R}}$, $\text{Eff}_{k}(X)=\text{Eff}^{\, n-k}(X)$, $\text{Nef}_{k}(X)=\text{Nef}^{\, n-k}(X)$ and $\text{Psef}_{k}(X)=\text{Psef}^{\, n-k}(X)$.  We call $\text{Eff}^{\, k}(X), \text{Psef}^{\, k}(X), \text{Nef}^{\, k}(X)$ effective cones, pseudo-effective cones and nef cones respectively, and call their elements effective cycle classes, pseudo-effective cycle classes and nef cycle classes respectively.

Let $G$ be a connected reductive algebraic group, $B$ be a Borel subgroup, and $R_{u}(B)$ be the unipotent radical of $B$. A normal $G$-variety $X$ is said to be $G$-spherical if there is an open $B$-orbit on $X$.

\begin{thm} (Theorem \ref{Nef<=Eff=Psef}) \label{introduction nef<=Eff=Psef}
Let $X$ be a smooth complete $G$-spherical variety of dimension $n$. Then for any integer $k$, $\text{Nef}^{\, k}(X)\subseteq\text{Psef}^{\, k}(X)=\text{Eff}^{\, k}(X)$, and these cones are rational polyhedra. If $\eta_{1}\in\text{Nef}^{\, k_{1}}(X)$ and $\eta_{2}\in\text{Nef}^{\, k_{2}}(X)$, then $\eta_{1}\cdot\eta_{2}\in\text{Nef}^{\, k_{1}+k_{2}}(X)$
\end{thm}

This theorem answers Problem 6.8 in \cite{DELV11}. And these phenomena are quite different from those on Abelian varieties. Now a natural question arises. What does $X$ look like if it is a smooth projective $G$-spherical variety such that $\text{Nef}^{\, k}(X)=\text{Psef}^{\, k}(X)$ for some $1\leq k\leq \text{dim}(X)-1$? For this question, we give some answers for special types of spherical varieties. Let $X$ be a $G$-spherical variety. Recall that if there is a point $x$ in the open $G$-orbit on $X$ such that the isotropy group $G_{x}\supseteq R_{u}(B)$, then $X$ is said to be $G$-horospherical. If there is no prime $B$-stable divisor $D$ on $X$ such that $D$ contains some $G$-orbit, but it's not $G$-stable, then $X$ is said to be toroidal. The following are our main results.

\begin{thm} (Theorem \ref{Nef=Psef toric}, Theorem \ref{nef=psef toroidal}, Theorem \ref{nef2=psef2 horospherical} and Corollary \ref{nef1=psef1 horospherical description}) \label{introduction nef=psef all the cases we have discussed}
Let $X$ be a smooth projective $G$-spherical variety of dimension $n$. Then the following hold.

$(i)$ If $X$ is a toric variety, then $\text{Nef}^{\, k}(X)=\text{Psef}^{\, k}(X)$ for some $1\leq k\leq n-1$ if and only if $X$ is isomorphic to the product of some projective spaces.

$(ii)$ If $X$ is toroidal, then $\text{Nef}^{\, k}(X)=\text{Psef}^{\, k}(X)$ for some $1\leq k\leq n-1$ if and only if $X$ is isomorphic to a rational homogeneous space.

$(iii)$ If $X$ is $G$-horospherical, $n\geq 3$ and $\text{Nef}^{\, 2}(X)=\text{Psef}^{\, 2}(X)$, then $\text{Nef}^{\, 1}(X)=\text{Psef}^{\, 1}(X)$.

$(iv)$ If $X$ is $G$-horospherical and $\text{Nef}^{\, 1}(X)=\text{Psef}^{\, 1}(X)$, then there is a $G$-equivariant morphism $\pi: X\rightarrow G/P$ such that each fiber of $\pi$ is isomorphic to the product of some smooth projective $L$-horospherical varieties of Picard number one, where $P$ is a parabolic subgroup of $G$ and $L$ is a Levi factor of $P$.
\end{thm}

This paper is organized as follows. In Section \ref{section spherical varieties}, we review some basic notations and results from the Luna-Vust Theory and the Mori Theory on spherical varieties, which are our main tools. In Section \ref{section nef are effective}, we prove Theorem \ref{introduction nef<=Eff=Psef}. In Section \ref{section classification nef=psef}, we mainly study the smooth projective spherical variety $X$ of dimension $n$ such that $\text{Nef}^{\, k}(X)=\text{Psef}^{\, k}(X)$ for some $1\leq k\leq n-1$. In Subsection \ref{subsection general spherical varieties}, we study the general cases. And we get that if $2\leq k\leq n-2$ and $\text{Nef}^{\, 1}(X)\neq\text{Psef}^{\, 1}(X)$, then for any birational Mori contraction, the dimension of the exceptional locus is no more than $k-1$. Then we study the $G$-equivariant morphisms from a complete spherical variety to some special rational homogeneous spaces, and these morphisms contribute to the proof of Theorem \ref{introduction nef=psef all the cases we have discussed}$(iv)$. In Subsection \ref{subsection toric varieties} and \ref{subsection toroidal varieties}, we show Theorem \ref{introduction nef=psef all the cases we have discussed}$(i)$ and $(ii)$ respectively. In Subsection \ref{subsection horospherical varieties}, we prove Theorem \ref{introduction nef=psef all the cases we have discussed}$(iii)(iv)$. Since the horospherical cases are complicated, we make some preliminaries in the part \ref{subsubsection preliminaries}. In the part \ref{subsubsection horospherical codimension two}, we show in Proposition \ref{exceptional locus of bir. Mori cont. horospherical} that the exceptional locus of a birational Mori contraction on a projective $\mathbb{Q}$-factorial horospherical variety is irreducible, and then prove Theorem \ref{introduction nef=psef all the cases we have discussed}$(iii)$. Finally, we prove Theorem \ref{introduction nef=psef all the cases we have discussed}$(iv)$ in the part \ref{subsubsection horospherical codimension one} after some reductions in the part \ref{subsubsection isomorphic colored fans}.

\medskip

\textbf{\normalsize Conventions and notations.} Schemes are always assumed to be seperated and of finite type over $\mathbb{C}$ and varieties are irreducible and reduced schemes.

Denote by $G$ a connected reductive algebraic group. Let $B$ be a Borel subgroup of $G$, $T$ a maximal torus in $B$. Let $\mathfrak{g}, \mathfrak{b}, \mathfrak{t}$ be the corresponding Lie algebras. Let $S$ be the set of simple roots. If $I$ is a subset of $S$, then we denote by $P_{I}$ the corresponding parabolic subgroup of $G$ containing $B$. Denote by $B^{-}$ the opposite Borel subgroup corresponding to $B$ and by $P_{I}^{-}$ the opposite parabolic subgroup corresponding to $P_{I}$.

For a linear algebraic group $H$, denote by $R_{u}(H)$ the unipotent radical of $H$. Let $\langle\cdot\, , \cdot\rangle_{H}$ be paring between the characters and the coroots on $H$. When there is no confusions, we omit the subscript $H$.

For a group $H$ and an $H$-module $V$, denote by
\begin{eqnarray*}
V^{(H)}=\{v\in V\mid v\neq 0,\text{ and } hv=\chi(h)v\text{ for some character } \chi\in\chi(H)\},
\end{eqnarray*}
where $\chi(H)$ is the group of characters of $H$. On the other hand, we denote by
\begin{eqnarray*}
V^{H}=\{v\in V\mid  hv=v\text{ for any } h\in H\}.
\end{eqnarray*}


\medskip

\textbf{\normalsize Acknowledgements.} I am greatly indebted to St\'{e}phane Druel for suggesting this problem and for a lot of discussions and advices. I would like to thank Michel Brion for answering my questions on algebraic groups and spherical varieties, and to thank Baohua Fu for reading a draft of this paper.

\section{Spherical varieties} \label{section spherical varieties}

In this section, we will recall some basic notations and results from the Luna-Vust Theory and the Mori Theory on spherical varieties which will be frequently used in this paper.

\medskip

Let $H$ be an algebraic group, $X$ a scheme (resp. a variety). If there is a morphism $\varphi: H\times X\rightarrow X$ such that $\varphi(h_{1}, \varphi(h_{2}, x))=\varphi(h_{1}h_{2}, x)$, $\varphi(e, x)=x$, where $h_{1}, h_{2}\in H$, $x\in X$ and $e$ is the unit of $H$, we say that $H$ acts on $X$ or say that there is an $H$-action on $X$. And $X$ is said to be an $H$-scheme (resp. an $H$-variety). For any $h\in H$ and any $x\in X$, denote by $h\cdot x=\varphi(h, x)$.

Let $X$ be an $H$-scheme. Denote by $S_{X, H}$ the set of $H$-orbits on $X$, and $S^{c}_{X, H}$ the set of closed $H$-orbits on $X$. Let $M\subseteq H$, $Y\subseteq X$ be subsets, then denote by $MY=\{m\cdot y| m\in M, y\in Y\}$. If $Y\subseteq X$ is a subscheme such that $HY\subseteq Y$ in the set theory, then we say that $Y$ is $H$-stable. For any point $x\in X$, denote by the isotropy group $H_{x}=\{h\in H\mid h\cdot x=x\}$.

\begin{defi} \label{defi. of spherical horospherical}
A normal $G$-variety $X$ is $G$-spherical if there is an open $B$-orbit on $X$.
\end{defi}

Let $x_{0}$ be a point in the open $G$-orbit of the $G$-spherical variety $X$, and $H=G_{x_{0}}$. Then we say that $X$ is a spherical $G/H$-embedding and identify $G/H$ with an open subset of $X$. Note that $G/H$ is itself a $G$-spherical variety, and we call it a homogeneous $G$-spherical variety. Denote by $\partial X=X\backslash(G/H)$, and we call it the boundary of $X$.

\medskip

Since all Borel subgroups of $G$ are conjugate, a normal $G$-variety $X$ is $G$-spherical if and only if there is an open $B'$-orbit on $X$, where $B'$ is any fixed Borel subgroup of $G$. There are several equivalent definitions of spherical varieties, see for example \cite[Def. 1.0.1, Thm. 2.1.2]{Per12}, \cite{Br86}, and  \cite{BLV86}.

\begin{prop} \label{equi. defi.}
Let $X$ be a normal $G$-variety. Then the following are equivalent.

$(i)$ $X$ is $G$-spherical.

$(ii)$ $\mathbb{C}(X)^{B}=\mathbb{C}$;

$(iii)$ for any (or some) Borel subgroup $B'$ of $G$, there is an open $B'$-orbit on $X$;

$(iv)$ $X$ has finitely many $B$-orbits;

$(v)$ every normal $G$-variety containing $G/H$ as the maximal $G$-orbit has only finitely many $G$-orbits, where $G/H\subseteq X$ is an open $G$-orbit;



\end{prop}

\textit{From now on to the end of Section \ref{section spherical varieties}, we assume that $X$ is a spherical $G/H$-embedding.} Define a map $f\mapsto\chi_{f}$ from $\mathbb{C}(X)^{(B)}$ to $\chi(B)$, where $\chi_{f}\in\chi(B)$ satisfies that for all $b\in B$, $b\cdot f=\chi_{f}(b)f$. This is a morphism of abelian groups with kernel $\mathbb{C}^{*}$. Denote by $M_{X}$ or $M_{G/H}$ the image of this morphism, and identify it with $\mathbb{C}(X)^{(B)}/\mathbb{C}^{*}$. Denote by $\text{rank}(X)=\text{rank}(G/H)=\text{rank}(M_{X})$ and call it the rank of $X$ or the rank of $G/H$.

Let $N_{X}=N_{G/H}=\text{Hom}(M_{X}, \mathbb{Z})$ be the dual of $M_{X}$. Any valuation $\nu$ on $X$ induces a homomorphism $\mathbb{C}(X)^{(B)}\rightarrow\mathbb{Q}$ by $f\mapsto\nu(f)$. Hence, $\nu$ induces an element $\rho(\nu)\in\text{Hom}(M_{X}, \mathbb{Q})$. Thus we can define a morphism $\rho=\rho_{G/H}=\rho_{X}: \{\text{valuations}\}\rightarrow (N_{X})_{\mathbb{Q}}$, where $(N_{X})_{\mathbb{Q}}=N_{X}\otimes\mathbb{Q}$. Denote by $(M_{X})_{\mathbb{Q}}=M_{X}\otimes\mathbb{Q}$.

Denote by $\mathcal{V}(G/H)$ the set of $G$-invariant valuations on $X$. By \cite[Cor. 1.8]{Kn91}, $\rho: \mathcal{V}(G/H)\rightarrow(N_{G/H})_{\mathbb{Q}}$ is injective. We regard $\mathcal{V}(G/H)$ as a subset of $(N_{X})_{\mathbb{Q}}$.

A subset $\mathfrak{C}$ of a vector space $\mathbb{Q}^{n}$ is called a cone, if it's closed under addition and multiplication by $\mathbb{Q}^{+}=\{q\in\mathbb{Q}| q\geq 0\}$. For a cone $\mathfrak{C}$ in $\mathbb{Q}^{n}$, we denote by $\mathfrak{C}^{o}$ the interior of it. The cone $\mathfrak{C}$ is called strictly convex if $\mathfrak{C}\cap(-\mathfrak{C})=\{0\}$. The cone $\mathfrak{C}$ is called finitely generated if there are finitely many elements $v_{1},\ldots,v_{s}\in\mathbb{Q}^{n}$ such that $\mathfrak{C}=\sum\limits_{i=1}^{s}\mathbb{Q}^{+}v_{i}$.

\begin{prop}(\cite[Cor. 3.2]{BP87}, \cite{Br90}) \label{valuation cone is a max dim. cone}
$\mathcal{V}(G/H)$ is a finitely generated cone in $(N_{X})_{\mathbb{Q}}$. Moreover, there exist linear independent forms $\chi_{1},\ldots,\chi_{m}$ in $(M_{G/H})_{\mathbb{Q}}$ such that $\mathcal{V}(G/H)=\{\nu\in(N_{G/H})\mid \chi_{i}(\nu)\geq 0, 1\leq i\leq m\}$.
\end{prop}

It should be noticed that the second assertion of Proposition \ref{valuation cone is a max dim. cone} didn't appear explicitly in \cite{Br90}. Just as what Knop pointed out after Theorem 5.4 in \cite{Kn91}, the proof of this assertion in \cite{Br90} is rather involved and rests ultimately on a case-by-case consideration. We call $\mathcal{V}(G/H)$ the valuation cone of $X$ or the valuation cone of $G/H$.

Denote by $\mathfrak{B}(X)$ the set of irreducible $B$-stable divisors on $X$. Let $\mathfrak{D}(G/H)=\{D\in\mathfrak{B}(X)\mid D\cap(G/H)\neq\emptyset\}$ and $\mathcal{V}_{X}=\{D\in\mathfrak{B}(X)\mid D\subseteq\partial{X}\}$. Thus, $\mathfrak{B}(X)=\mathcal{V}_{X}\cup\mathfrak{D}(G/H)$, and $\mathcal{V}_{X}\cap\mathfrak{D}(G/H)=\emptyset$. We call the elements in the set $\mathfrak{D}(G/H)$ colors of $G/H$, while the elements in the set $\mathcal{V}_{X}$ are called boundary divisors. For any $Y\in S_{X, G}$, we denote by $\mathfrak{D}_{Y}=\{D\in\mathfrak{D}(G/H)\mid Y\subseteq D\}$, $\mathcal{V}_{Y}=\{D\in\mathcal{V}_{X}\mid Y\subseteq D\}$, and $\mathfrak{B}_{Y}=\mathcal{V}_{Y}\cup\mathfrak{D}_{Y}$. Let $\mathfrak{D}_{X}=\bigcup\limits_{Y\in S_{X, G}}\mathfrak{D}_{Y}$, and we call its elements colors of $X$.

\begin{defi} \label{defi. of colored fans}
$(i)$ A colored cone is a pair $(\mathfrak{C}, \mathfrak{D})$ with $\mathfrak{C}\subseteq (N_{G/H})_{\mathbb{Q}}$ and $\mathfrak{D}\subseteq\mathfrak{D}(G/H)$ having the following properties:

$(a)$ $\mathfrak{C}$ is a cone generated by $\rho(\mathfrak{D})$ and finitely many elements in $\mathcal{V}(G/H)$;

$(b)$ $\mathfrak{C}^{\circ}\cap\mathcal{V}(G/H)\neq\emptyset$, i.e. there is a $G$-invariant valuation in the interior of $\mathfrak{C}$.

A colored cone $(\mathfrak{C}, \mathfrak{D})$ is called strictly convex if the following holds:

$(c)$ $\mathfrak{C}$ is a strictly convex cone and $0\notin\rho(\mathfrak{D})$.

$(ii)$ A pair $(\mathfrak{C}_{0}, \mathfrak{D}_{0})$ is a colored face of the colored cone $(\mathfrak{C}, \mathfrak{D})$ if $\mathfrak{C}_{0}$ is a face of $\mathfrak{C}$, $\mathfrak{C}_{0}^{\circ}\cap\mathcal{V}(G/H)\neq\emptyset$ and $\mathfrak{D}_{0}=\mathfrak{D}\cap\rho^{-1}(\mathfrak{C}_{0})$.

$(iii)$ A colored fan $\mathbb{F}$ is a nonempty finite set of colored cones with the following properties:

$(a)$ Every colored face of $(\mathfrak{C}, \mathfrak{D})\in\mathbb{F}$ belongs to $\mathbb{F}$;

$(b)$ For every $\nu\in\mathcal{V}(G/H)$, there is at most one $(\mathfrak{C}, \mathfrak{D})\in\mathbb{F}$ such that $\nu\in\mathfrak{C}^{\circ}$.

A colored fan $\mathbb{F}$ is called strictly convex if $(0, \emptyset)\in\mathbb{F}$, or equivalently, if all elements of $\mathbb{F}$ are strictly convex.

\end{defi}

For any $G$-orbit $Y$ on $X$, we denote by $\mathfrak{C}^{c}_{Y}=(\mathfrak{C}_{Y}, \mathfrak{D}_{Y})$, where $\mathfrak{C}_{Y}$ is the cone in $(N_{X})_{\mathbb{Q}}$ generated by all $\rho(\nu_{D})$, $D\in\mathfrak{B}_{Y}$. Let $\overline{Y}$ be the closure of $Y$ in $X$. For the convenience of discussions, we also denote by $\mathfrak{C}_{\overline{Y}}=\mathfrak{C}_{Y}$, $\mathcal{V}_{\overline{Y}}=\mathcal{V}_{Y}$, $\mathfrak{D}_{\overline{Y}}=\mathfrak{D}_{Y}$ and $\mathfrak{C}^{c}_{\overline{Y}}=\mathfrak{C}^{c}_{Y}$. Denote by $\mathbb{F}_{X}=\{\mathfrak{C}^{c}_{Y}\mid Y\in S_{X, G}\}$, $\mathfrak{C}(X)=\bigcup\limits_{Y\in S_{X, G}}\mathfrak{C}_{Y}$, and $\text{Supp}(\mathbb{F}_{X})=\mathfrak{C}(X)\cap\mathcal{V}(G/H)$.

A $G$-spherical variety $X$ is said to be simple if it has only one closed $G$-orbit. If $X$ is simple with the unique closed $G$-orbit $Y$, then we denote by $\mathfrak{C}^{c}(X)=\mathfrak{C}^{c}_{Y}$.

\begin{thm}(\cite[Thm. 3.1, Lem. 3.2, Thm. 3.3]{Kn91}) \label{correspondence fans and varieties}
$(i)$ The map $X\mapsto\mathfrak{C}^{c}(X)$ is a bijection between isomorphism classes of simple spherical $G/H$-embeddings and strictly convex colored cones in $(N_{G/H})_{\mathbb{Q}}$.

$(ii)$ Let $Y$ be a $G$-orbit on a spherical $G/H$-embedding $X$. Then $Z\mapsto\mathfrak{C}^{c}_{Z}$ is a bijection between $G$-orbits whose closures contain $Y$ and colored faces of $\mathfrak{C}^{c}_{Y}$.

$(iii)$ The map $X\mapsto\mathbb{F}_{X}$ induces a bijection between isomorphism classes of spherical $G/H$-embeddings and strictly convex colored fans in $(N_{G/H})_{\mathbb{Q}}$.
\end{thm}

If $G/H, G/H'$ are two homogeneous $G$-spherical varieties and there is a dominant $G$-equivariant morphism $\phi: G/H\rightarrow G/H'$. Then it induces a surjective morphism $\phi_{*}: N_{G/H}\rightarrow N_{G/H'}$, and $\phi_{*}(\mathcal{V}(G/H))=\mathcal{V}(G/H')$. Let $\mathfrak{D}_{\phi}$ be the set of those $D\in\mathfrak{D}(G/H)$ such that $\phi(D)$ is dense in $G/H'$.

\begin{defi} \label{defi. of morphism of colored cones and colored fans}
Keep notations as above.

$(a)$ Let $(\mathfrak{C}, \mathfrak{D}), (\mathfrak{C}', \mathfrak{D}')$ be colored cones for $G/H, G/H'$ respectively. Then we say that $(\mathfrak{C}, \mathfrak{D})$ maps to $(\mathfrak{C}', \mathfrak{D}')$ if $\phi_{*}(\mathfrak{C})\subseteq\mathfrak{C}'$ and $\phi_{*}(\mathfrak{D}\backslash\mathfrak{D}_{\phi})\subseteq\mathfrak{D}'$.

$(b)$ Let $\mathbb{F}, \mathbb{F}'$ be colored fans for $G/H, G/H'$ respectively. Then we say that $\phi_{*}:\mathbb{F}\rightarrow\mathbb{F}'$ is a morphism of colored fans if every element of $\mathbb{F}$ maps to some element of $\mathbb{F}'$.
\end{defi}

\begin{thm}(\cite[Thm. 4.1]{Kn91}) \label{morphism of fan}
Keep notations as above. Let $X$ and $X'$ be a spherical $G/H$-embedding and a spherical $G/H'$-embedding respectively. Then $\phi$ extends to a morphism $X\rightarrow X'$ if and only if $\mathbb{F}_{X}$ maps to $\mathbb{F}_{X'}$ by $\phi_{*}$.
\end{thm}

\begin{defi}

$(i)$ Let $PL(X)$ be the set such that an element $l\in PL(X)$ is a family $(l_{Y})_{Y\in S_{X, G}}$ as follows

$(a)$ for any $G$-orbit $Y$, $l_{Y}$ is a linear function defined on $\mathfrak{C}_{Y}$ with valuations in $\mathbb{Q}$, and it equals to the restriction of an element in $M_{X}$ in the sense of the inclusion $M_{X}\subseteq\text{Hom}(N_{X}, \mathbb{Q})$;

$(b)$ for any $G$-orbit $Z$ such that $Z\subseteq\overline{Y}$, there is an equality $l_{Z}|_{\mathfrak{C}_{Y}}=l_{Y}$.

$(ii)$ Denote by $L(X)$ the group of linear functions on $(N_{X})_{\mathbb{Q}}$ with valuation in $\mathbb{Q}$.
\end{defi}

We call the elements of $PL(X)$ piecewise linear functions on $\mathfrak{C}(X)$, while it should be noticed that if $Y_{1}, Y_{2}\in S_{X, G}$, maybe $l_{Y_{1}}|_{\mathfrak{C}_{Y_{1}\cap\mathfrak{C}_{Y_{2}}}}\neq l_{Y_{2}}|_{\mathfrak{C}_{Y_{1}\cap\mathfrak{C}_{Y_{2}}}}$.

Suppose that $X$ is moreover complete. Take $l=(l_{Y})_{Y\in S_{X, G}}\in PL(X)$, then for any $Y\in S^{c}_{X, G}$, $\text{dim}(\mathfrak{C}_{Y})=\text{dim}(N_{X})_{\mathbb{Q}}$ by \cite[Thm. 6.3]{Kn91}. Thus, $l_{Y}$ can be uniquely extended as a linear function $l_{Y}^{N}$ defined on $(N_{X})_{\mathbb{Q}}$. If $\delta=\sum\limits_{D\in\mathfrak{B}(X)}n_{D}(\delta)D$ is a Cartier $B$-stable divisor, then by \cite[Prop. 3.1]{Br89}, for each $Z\in S^{c}_{X, G}$, there exists a unique element $\chi_{Z}\in M_{X}$ such that $\nu_{D}(\chi_{Z})=n_{D}(\delta)$ for all $D\in\mathfrak{B}_{Z}$. These $\chi_{Z}$ determine a unique $l(\delta)=(l_{Y})_{Y\in S_{X, G}}$ such that $l_{Z}|_{\mathfrak{C}_{Z}}=\chi_{Z}|_{\mathfrak{C}_{Z}}$ for all $Z\in S^{c}_{X, G}$. Hence, $l_{Y}(\rho(\nu_{D}))=n_{D}(\delta)$ for all $Y\in S_{X, G}$ and $D\in\mathfrak{B}_{Y}$. By \cite[Thm. 3.3]{Br89}, if $X$ is moreover projective, then every Cartier $B$-stable divisor $\delta$ is uniquely associated with a well-defined piecewise linear function $l(\delta)$ on $\mathfrak{C}(X)$.

Let $X$ be a projective $\mathbb{Q}$-factorial spherical $G/H$-embedding. We will define some 1-cycle classes on $X$ as in the proof of \cite[Thm. 3.2]{Br93}. Take $\delta$ to be a Cartier divisor on $X$. Then by \cite[Thm. 1.3(ii)]{Br93}, we can find integers $n_{D}(\delta)$ such that $\delta=\sum\limits_{D\in\mathfrak{B}(X)}n_{D}(\delta)D$. Let $l(\delta)\in PL(X)$ be the corresponding piecewise linear function on $\mathfrak{C}(X)$. In particular, $l(\delta)(\rho(\nu_{D}))=n_{D}(\delta)$ for all $D\in\mathcal{V}_{X}\cup\mathfrak{D}_{X}$, and $l_{Y}=l(\delta)\mid_{\mathfrak{C}_{Y}}$ is linear on $\mathfrak{C}_{Y}$ for each $Y\in S_{X, G}$. For each $Y\in S^{c}_{X, G}$, we denote by $l(\delta, \mathfrak{C}_{Y})$ or $l(\delta, Y)$ or $l^{N}_{Y}$ the linear extension of $l_{Y}$ to the whole vector space $(N_{X})_{\mathbb{Q}}$.

Suppose that $\mu\in\mathbb{F}_{X}$ is a wall, i.e. there are two maximal dimensional colored cone $(\mu_{+}, \mathfrak{D}_{+})$ and $(\mu_{-}, \mathfrak{D}_{-})$ in $\mathbb{F}_{X}$ (i.e. $\text{dim}(\mu_{+})=\text{dim}(\mu_{-})=\text{rank}(G/H)$) such that $\mu=\mu_{+}\cap\mu_{-}$ and for some subset $\mathfrak{D}\subseteq\mathfrak{D}_{+}\cap\mathfrak{D}_{-}$, $(\mu, \mathfrak{D})$ is a colored face of $(\mu_{+}, \mathfrak{D}_{+})$ and $(\mu_{-}, \mathfrak{D}_{-})$ of codimension one. Let $\chi_{\mu}\in M_{X}$ be the primitive lattice point such that $\chi_{\mu}$ vanishes on $\mu$ and it is positive on $\mu_{+}\backslash\mu$, negative on $\mu_{-}\backslash\mu$. Define $C_{\mu}\in N_{1}(X)_{\mathbb{R}}$ such that
\begin{eqnarray}\label{eqn. C_u}
\delta\cdot C_{\mu}=(l(\delta, \mu_{+})-l(\delta, \mu_{-}))/\chi_{\mu}.
\end{eqnarray}

For any closed $G$-orbit $Y$ and any color $D\in\mathfrak{D}\backslash\mathfrak{D}_{Y}$, i.e. $Y\nsubseteq D$, define $C_{D, Y}\in N_{1}(X)_{\mathbb{R}}$ such that
\begin{eqnarray} \label{eqn. C_(D, Y)}
\delta\cdot C_{D, Y}=n_{D}(\delta)-l(\delta, Y)(\rho(\nu_{D})).
\end{eqnarray}

Note that by \cite[Thm. 3.2]{Br93}, these two kinds of curve classes are pseudo-effective and they generate $NE(X)$.

\begin{rmk} \label{fix C_u as a curve instead of a class}
Assume that $X$ is complete. By \cite[Prop. 3.3]{Br93}, there is a unique $B$-stable smooth rational curve $C_{YZ}$ such that $C_{YZ}^{B}=\{p_{1}, p_{2}\}$, $C_{\mu}\cap Y=\{p_{1}\}$, $C_{\mu}\cap Z=\{p_{2}\}$ and $C_{YZ}=C_{\mu}$ in $N_{1}(X)_{\mathbb{R}}$. We identify $C_{\mu}$ with the curve $C_{YZ}$. Moreover, by the proof of \cite[Prop. 3.3]{Br93}, $GC_{\mu}=\overline{V}$, $C_{\mu}=(\overline{V})^{R_{u}(B)}$, $Gp_{1}=Y$ and $Gp_{2}=Z$, where $V$ is the $G$-orbit such that $\mathfrak{C}_{V}=\mu$.
\end{rmk}

\section{Nef versus effective cycle classes on smooth complete spherical varieties} \label{section nef are effective}

We will show in Theorem \ref{Nef<=Eff=Psef} that on smooth complete spherical varieties, nef cycle classes are effective and the products of nef cycle classes are also nef.

\bigskip

Let $\eta=\sum\limits_{i}n_{i}Y_{i}$ be an algebraic cycle of an $H$-scheme $X$, where $H$ is an algebraic group, all $n_{i}\neq 0$ and all $Y_{i}$ are irreducible closed subschemes. If all $Y_{i}$ are $H$-stable, then we call $\eta$ an $H$-stable algebraic cycle of $X$.

If $X$ is a smooth complete variety, and $\eta_{1},\eta_{2},\ldots,\eta_{m}\in A^{*}(X)_{\mathbb{R}}$, then we denote by $\prod\limits_{i=1}^{m}\eta_{i}=\eta_{1}\cdot\eta_{2}\cdot\ldots\cdot\eta_{m}$. If $X_{1},\ldots,X_{m}$ are smooth complete varieties and $\eta_{i}\in A^{*}(X_{i})_{\mathbb{R}}$, then we denote by $\bigotimes\limits_{i=1}^{m}\eta_{i}=\eta_{1}\otimes\eta_{2}\cdots\otimes\eta_{m}=\prod\limits_{i=1}^{m}\pi_{i}^{*}\eta_{i}$, where $\pi_{i}: \prod\limits_{j=1}^{m}X_{j}\rightarrow X_{i}$ is the $i$-th projection.

\begin{prop} \label{cycles are rat. equiv. to stable ones, picard group on spherical varieties}
$(i)$ Let $X$ be a $\Gamma$-scheme, where $\Gamma$ is a connected solvable linear algebraic group. Then any effective algebraic cycle on $X$ is rationally equivalent to an effective $\Gamma$-stable algebraic cycle.

$(ii)$ Let $X$ be a $G$-spherical variety. Then the group of rationally equivalent divisor classes on $X$ is generated by irreducible $B$-stable divisors, while the linearly equivalences are defined by $\text{div}(f)$ for all $f\in\mathbb{C}(X)^{(B)}$. More precisely, there is an exact sequence $M_{X} \rightarrow \mathbb{Z}(\mathfrak{B}(X))\rightarrow A^{1}(X)\rightarrow 0$. If moreover $X$ is complete, then the map $M_{X}\rightarrow \mathbb{Z}(\mathfrak{B}(X))$ is injective.
\end{prop}

Note that these two conclusions should be well-known, but we fail to find proper references stating them explicitly. For the case when $X$ is a projective $\Gamma$-variety, the statement of $(i)$ has appeared in the proof of \cite[Thm. 1]{Hi84}, while the case when $X$ is a projective normal $\Gamma$-variety, the statement of $(i)$ appeared in \cite[Thm. 1.3(i)]{Br93}. Recall that the proof of \cite[Thm. 1]{FMSS95} was deduced to the case when $X$ is a projective $\Gamma$-variety, and the effectiveness of the cycles are preserved in the proof. Then we can get Proposition \ref{cycles are rat. equiv. to stable ones, picard group on spherical varieties}$(i)$. In fact, when proving \cite[Cor. of Thm. 1]{FMSS95}, the statement of $(i)$ has been used. For the first assertion of $(ii)$, we only need to notice that $M_{X}\cong\mathbb{C}(X)^{(B)}/\mathbb{C}^{*}$, then it's just \cite[Thm. 1.3(ii)]{Br93}, and it's also a corollary of \cite[Thm. 1]{FMSS95}. By \cite[Thm. 6.3]{Kn91}, when $X$ is moreover complete, $\rho(\mathfrak{B}(X))$ generates $(N_{X})_{\mathbb{Q}}$. Thus, the map $M_{X}\rightarrow \mathbb{Z}(\mathfrak{B}(X))$ is injective.

The following proposition is a direct consequence of Proposition \ref{cycles are rat. equiv. to stable ones, picard group on spherical varieties} and \cite[Prop. 3.1]{Br89}. It also appeared in \cite[Thm. 3.2.14]{Per12} with a little different statement, while we can also get the statement here from the proof of \cite[Thm. 3.2.14]{Per12}.

\begin{prop} \label{locally fac. Q-fac. criterion}
Let $X$ be a $G$-spherical variety.

$(i)$ The variety $X$ is locally factorial if and only if for any $G$-orbit $Y$, the elements $\rho(v_{D})$ form a part of a $\mathbb{Z}$-basis of $N_{X}$, where $D$ runs over the set $\mathfrak{B}_{Y}$.

$(ii)$ The variety $X$ is $\mathbb{Q}$-factorial if and only if for any $G$-orbit $Y$, the elements $\rho(v_{D})$ are linearly independent in $(N_{X})_{\mathbb{Q}}$, where $D$ runs over the set $\mathfrak{B}_{Y}$.
\end{prop}

Note that by \cite[Thm. 2]{FMSS95}, if $X$ is a smooth complete spherical variety, then for any integer $k$, $N^{k}(X)=A^{k}(X)=H^{2k}(X)$. In particular, the numerical equivalence and the rational equivalence coincide.

\begin{cor} \label{effective cycles on X*X' are group stable effective}
Let $X$ be a $\Gamma$-scheme such that $\Gamma$ has only finitely many orbits, where $\Gamma$ is a connected solvable linear algebraic group. Let $X'$ be another $\Gamma$-scheme and $\eta\in A^{*}(X\times X')$ be an effective cycle. Then $\eta=\sum\limits_{i=1}^{m}c_{i}\delta_{i}\otimes\delta'_{i}$, where $c_{i}\geq 0$, and $\delta_{i}, \delta'_{i}$ can be represented by irreducible $\Gamma$-stable closed subvarieties of $X$ and $X'$ respectively.
\end{cor}

\begin{proof}
By Proposition \ref{cycles are rat. equiv. to stable ones, picard group on spherical varieties}$(i)$, we can assume that $\eta$ is $\Gamma\times\Gamma$-stable and effective. Without loss of generality, we can assume that $\eta$ is represented by an irreducible $\Gamma\times\Gamma$-stable closed subvariety $Z$ of $X\times X'$.

Let $X_{0}=X\times X'$ be a variety with a $\Gamma_{0}$-action such that $\Gamma_{0}=\Gamma$ and $g\cdot(x, x')=(g\cdot x, e\cdot x')=(g\cdot x, x')$ for all $g\in\Gamma_{0}$, $x\in X$, and $x'\in X'$. Note that $\Gamma_{0}$ can be identified with the subgroup $\Gamma\times\{e\}$ of $\Gamma\times\Gamma$. In particular, $Z$ is $\Gamma_{0}$-stable. By \cite[Lem. 3]{FMSS95}, $Z\subseteq X_{0}$ has the form $Z=Y\times Y'$, where $Y\subseteq X$ is a $\Gamma$-stable closed subvariety and $Y'\subseteq X'$ is a closed subvariety. By Proposition \ref{cycles are rat. equiv. to stable ones, picard group on spherical varieties}$(i)$ again, $Y'$ is rationally equivalent to an effective $\Gamma$-stable algebraic cycle. The conclusion follows.
\end{proof}

\begin{thm} \label{Nef<=Eff=Psef}
Let $X$ be a smooth complete $G$-spherical variety of dimension $n$. Then for any integer $k$, the following hold.

$(i)$ $\text{Eff}^{\, k}(X)=\text{Psef}^{\, k}(X)$, and it's a rational polyhedron. Dually, $\text{Nef}^{n-k}(X)$ is also a ratinal polyhedron.

$(ii)$  $\text{Nef}^{\, k}(X)\subseteq\text{Psef}^{\, k}(X)$.

$(iii)$ If $\eta\in A^{k}(X)$ is a cycle such that $\eta\cdot A^{n-k}(X)=0$, then $\eta\cdot A^{*}(X)=0$.

$(iv)$ If $\eta\in \text{Nef}^{\, k}(X)$, then for any effective algebraic cycle class $\delta$ in $A^{*}(X)$, the intersection $\eta\cdot\delta$ is an effective algebraic cycle class in $A^{*}(X)$. In particular, if for all $1\leq i\leq m$, $\eta_{i}\in\text{Nef}^{\, k_{i}}(X)$ , then $\prod\limits_{i=1}^{m}\eta_{i}\in\text{Nef}^{\, k_{0}}(X)$, where $k_{0}=\sum\limits_{i=1}^{m}k_{i}$.
\end{thm}

\begin{proof}
The  conclusion $(i)$ follows from Proposition \ref{equi. defi.}$(i)(iv)$ and Proposition \ref{cycles are rat. equiv. to stable ones, picard group on spherical varieties}$(i)$.

Let $\eta\in A^{k}(X)$ be any cycle class and $Z\subseteq X$ be a $B$-stable closed subvariety. By Proposition \ref{equi. defi.}$(i)(iv)$ and Corollary \ref{effective cycles on X*X' are group stable effective}, we can assume that $\Delta_{*}(Z)=\sum\limits_{i=1}^{m} c_{i} u_{i}\otimes v_{i}$ in $A^{*}(X\times X)$, where $\Delta: X\rightarrow X\times X$ is the diagonal morphism, $m$ is a positive integer, all $c_{i}\geq 0$, and all $u_{i}, v_{i}$ are algebraic cycle classes represented by irreducible $B$-stable closed subvarieties. Let $\pi_{1}, \pi_{2}$ be the two projections from $X\times X$ to the two factors, then $\pi_{1}\Delta=id_{X}$, and $\pi_{2}\Delta=id_{X}$. Hence, $\eta\cdot Z=\pi_{2*}\Delta_{*}(\eta\cdot Z)=\pi_{2*}\Delta_{*}(\Delta^{*}\pi_{1}^{*}\eta\cdot Z)=\pi_{2*}(\pi_{1}^{*}\eta\cdot\Delta_{*}(Z))=\pi_{2*}(\sum\limits_{i=1}^{m}(c_{i}\eta\cdot u_{i})\otimes v_{i})=\sum\limits_{\text{dim}(u_{i})=k}c_{i}(\eta\cdot u_{i})v_{i}$. Then the rest of this theorem follows easily from this formula. In particular, for $(ii)$, we can take $\eta\in\text{Nef}^{\, k}(X)$ and $Z=X$ to apply this formula.
\end{proof}

Note that the conclusion $(i)$ is indeed a classical result: for the case when $X$ is projective, see \cite[Thm. 1.3(ii)]{Br93}; for more general cases, see \cite[Cor. of Thm. 1]{FMSS95}. The conclusion $(iii)$ can be deduced from \cite[Cor. of Thm. 2]{FMSS95} readily. We state and prove these two conclusions here for the later use and for the discussions of the case when $X$ is a projective $\mathbb{Q}$-factorial toric variety, see Remark \ref{Nef<=Eff=Psef proj. Q-fac. toric}.

\begin{rmk} \label{Nef1<=Eff1=Psef1 Q-factorial}
Let $X$ be a $\mathbb{Q}$-factorial complete $G$-spherical variety of dimension $n$. Define $A^{1}(X)_{\mathbb{Q}}\times A^{n-1}(X)_{\mathbb{Q}}\rightarrow \mathbb{Q}$ as the linear extension of the map $\text{Pic}(X)\times A^{n-1}(X)\rightarrow \mathbb{Z}, \delta\times C\mapsto \text{deg}(\delta\mid_{C})$, where $A^{n-1}(X)$ (resp. $A^{1}(X)$) is the Chow group of curves (resp. Weil divisors), $\text{Pic}(X)$ is the Picard group of $X$ and $A^{k}(X)_{\mathbb{Q}}=A^{k}(X)\otimes\mathbb{Q}$ for $k=1, n-1$.  Thus, $N^{k}(X)_{\mathbb{R}}$, $\text{Eff}^{\, k}(X)$, $\text{Psef}^{\, k}(X)$, and $\text{Nef}^{\, k}(X)$ are well-defined for $k=1, n-1$. By Proposition \ref{cycles are rat. equiv. to stable ones, picard group on spherical varieties}$(i)$, $\text{Eff}^{\, k}(X)=\text{Psef}^{\, k}(X)$ for $k=1, n-1$. Since the proof of Theorem \ref{Nef<=Eff=Psef}$(ii)$ also works when $k=1$ under our assumption of $X$ here, we know that $\text{Nef}^{\, 1}(X)\subseteq\text{Psef}^{\, 1}(X)$. By the duality, $\text{Nef}^{\, n-1}(X)\subseteq\text{Psef}^{\, n-1}(X)$.
\end{rmk}

\begin{cor} \label{nef(X)=psef(X) if and only if nef(X_i)=psef(X_i)}
Suppose that for each $1\leq i\leq m$, $X_{i}$ is a smooth complete $G_{i}$-spherical variety, where each $G_{i}$ is a connected reductive algebraic group. Let $X=\prod\limits_{i=1}^{m}X_{i}$ and $r$ be a positive integer. Then $\text{Nef}^{\, k}(X)=\text{Psef}^{\, k}(X)$ for all $1\leq k\leq r$ if and only if $\text{Nef}^{\, k}(X_{i})=\text{Psef}^{\, k}(X_{i})$ for all $1\leq k\leq r$ and all $1\leq i\leq m$.
\end{cor}

\begin{proof}
Note that $X$ is $(\prod\limits_{i=1}^{m}G_{i})$-spherical. By Theorem \ref{Nef<=Eff=Psef}$(i)(ii)$, $\text{Nef}^{\, k}(X)\subseteq\text{Eff}^{\, k}(X)=\text{Psef}^{\, k}(X)$ and $\text{Nef}^{\, k}(X_{i})\subseteq\text{Eff}^{\, k}(X_{i})=\text{Psef}^{\, k}(X_{i})$ for all integers $k$ and all $1\leq i\leq m$.

The ``only if'' part follows from Theorem \ref{Nef<=Eff=Psef}$(i)(ii)$ and the projection formula.

The ``if'' part: Let $n=\text{dim}(X)$ and $n_{i}=\text{dim}(X_{i})$. For any $1\leq k\leq r$, take any class $\eta\in\text{Eff}^{\, k}(X)$, and any class $\delta\in\text{Eff}^{\, n-k}(X)$, then by Corollary \ref{effective cycles on X*X' are group stable effective} and the induction on $m$, we get that the class $\eta\in\sum\limits_{i_{1}+\ldots+i_{m}=k}\bigotimes\limits_{j=1}^{m}\text{Eff}^{\, i_{j}}(X_{j})$ and the class $\delta\in\sum\limits_{i_{1}+\ldots+i_{m}=k}\bigotimes\limits_{j=1}^{m}\text{Eff}^{\, n_{j}-i_{j}}(X_{j})$. By the assumption on $X_{i}$, $\eta\in\sum\limits_{i_{1}+\ldots+i_{m}=k}\bigotimes\limits_{j=1}^{m}\text{Nef}^{\, i_{j}}(X_{j})$. Hence, $\eta\cdot\delta\geq 0$, i.e. $\text{Nef}^{\, k}(X)=\text{Psef}^{\, k}(X)$ for all $1\leq k\leq r$.
\end{proof}

\begin{rmk} \label{nef1(X)=psef1(X) if and only if nef1(X_i)=psef1(X_i) Q-factorial}
Suppose that for each $1\leq i\leq m$, $X_{i}$ is a complete $\mathbb{Q}$-factorial $G_{i}$-spherical variety, where each $G_{i}$ is a connected reductive algebraic group.  Let $X=\prod\limits_{i=1}^{m}X_{i}$. Then $\text{Nef}^{\, 1}(X)=\text{Psef}^{\, 1}(X)$ if and only if $\text{Nef}^{\, 1}(X_{i})=\text{Psef}^{\, 1}(X_{i})$ for all $1\leq i\leq m$. By Remark \ref{Nef1<=Eff1=Psef1 Q-factorial}, $\text{Nef}^{\, 1}(X)\subseteq\text{Eff}^{\, 1}(X)=\text{Psef}^{\, 1}(X)$ and $\text{Nef}^{\, 1}(X_{i})\subseteq\text{Eff}^{\, 1}(X)=\text{Psef}^{\, 1}(X_{i})$ for all $1\leq i\leq m$. Note that if we take $r=1$, then the proof of Corollary \ref{nef(X)=psef(X) if and only if nef(X_i)=psef(X_i)} also works for the case here.
\end{rmk}

\section{Smooth projective spherical varieties whose effective cycle classes of codimension $k$ are nef} \label{section classification nef=psef}

In this section, we mainly study the smooth projective spherical variety $X$ of dimension $n$ such that $\text{Nef}^{\, k}(X)=\text{Psef}^{\, k}(X)$ for some $1\leq k\leq n-1$. We also study some related properties when $X$ may be not smooth, but only $\mathbb{Q}$-factorial. We discuss the general spherical cases, toric cases, toroidal cases and horospherical cases in the corresponding four subsections. The horospherical cases are the most complicated, and we need to make some preparations in the part \ref{subsubsection preliminaries}.

\subsection{General spherical varieties} \label{subsection general spherical varieties}

There are two main results in this subsection, namely Theorem \ref{higher codimension arbitrary varieites} and Theorem \ref{morphism to G/P_0}. The former says that if $X$ is a smooth projective variety such that $\text{Nef}^{\, k}(X)\subseteq\text{Psef}^{\, k}(X)$ for some $2\leq k\leq \text{dim}(X)-2$, then the dimensions of the exceptional loci of the birational Mori contractions on $X$ are less than $k$.

Now let $X$ be a spherical $G/H$-embedding. Denote by $\mathfrak{D}_{0}(G/H)=\{D\in\mathfrak{D}(G/H)\mid \rho(\nu_{D})=0\}$. In Theorem \ref{morphism to G/P_0}, we show that if $X$ is complete, then for any subset $\mathfrak{D}_{1}\subseteq\mathfrak{D}_{0}(G/H)$, we can construct a $G$-equivariant morphism $\pi_{0}: X\rightarrow G/P_{0}$, where $P_{0}$ is a parabolic subgroup of $G$ related to $\mathfrak{D}_{1}$. Moreover, $X$ shares a lot of properties with the fibers. Let $F$ be any fixed fiber of $\pi_{0}$, then $X$ is smooth (resp. projective, $\mathbb{Q}$-factorial, locally factorial) if and only if $F$ is so. If $F$ and $X$ are both projective and $\mathbb{Q}$-factorial, then $\text{Nef}^{\, 1}(X)=\text{Psef}^{\, 1}(X)$ if and only if $\text{Nef}^{\, 1}(F)=\text{Psef}^{\, 1}(F)$.

\begin{thm} \label{higher codimension arbitrary varieites}
Let $X$ be a smooth projective variety of dimension $n$ such that $\text{Psef}^{\, k}(X)\subseteq\text{Nef}^{\, k}(X)$ for some $2\leq k\leq n-2$. Assume that $R$ is an extremal ray of $\overline{NE(X)}$ which can be contracted, i.e. there is a morphism $\pi: X\rightarrow Y$ such that $\pi_{*}\mathcal{O}_{X}=\mathcal{O}_{Y}$ and for any irreducible curve $C$ on $X$, $\pi(C)$ is a point if and only if $C\in R$. Then either $R\subseteq\text{Nef}_{1}(X)$ or $\text{dim}(A)\leq\min\{k-1, n-k-1\}$, where $A\subseteq X$ is the exceptional locus of $\pi$.
\end{thm}

\begin{proof}
Assume that $R\nsubseteq\text{Nef}_{1}(X)$. Then $\pi$ is birational. Let $A\subseteq X$ be the exceptional locus and $A'=\pi(A)$. By the duality, we can assume that $k\leq\frac{n}{2}$. Now assume that $\text{dim}(A)\geq k$. Let $C$ (resp. $E$) be an irreducible curve (resp. a prime divisor) on $X$ such that $C\in R$ and $E\cdot C<0$.

Case 1. Assume that there exists a point $y\in Y$ such that $\text{dim}(\pi^{-1}(y))\geq k$.

Let $F$ be an irreducible component of $\pi^{-1}(y)$ such that $d=\text{dim}(F)\geq k$. Take $d-1$ general ample divisors $D_{1},\ldots,D_{d-1}$ on $X$. Thus, $\prod\limits_{i=1}^{d-1}D_{i}\cdot F=\lambda C\in N_{1}(X)_{\mathbb{R}}$ for some $\lambda>0$. So $\prod\limits_{i=1}^{d-1}D_{i}\cdot F\cdot E<0$.

On the other hand, $\prod\limits_{i=1}^{d-k}D_{i}\cdot F\in\text{Psef}^{\, n-k}(X)$ and $\prod\limits_{i=d-k+1}^{d-1}D_{i}\cdot E\in\text{Psef}^{\, k}(X)\subseteq\text{Nef}^{\, k}(X)$. Thus, $\prod\limits_{i=1}^{d-1}D_{i}\cdot F\cdot E\geq 0$. We get a contradiction.

Case 2. Assume that for any $y\in Y$, $\text{dim}(\pi^{-1}(y))\leq k-1$. Let $A_{0}$ be an irreducible component of $A$ such that $a=\text{dim}(A_{0})\geq k$. Let $A'_{0}=\pi(A_{0})$. Thus, $a'=\text{dim}(A'_{0})\geq 1$.

Take $a'$ general very ample Cartier divisors $H_{1},\ldots,H_{a'}$ on $Y$. Take a general section $D_{i}$ of each Cartier divisor $\pi^{*}\mathcal{O}_{Y}(H_{i})$. Take $a-a'-1$ general ample divisors $D_{a'+1},\ldots,D_{a-1}$ on $X$. Thus, $\prod\limits_{i=1}^{a-k}D_{i}\cdot A_{0}\in\text{Psef}^{\, n-k}(X)$ and $\prod\limits_{i=a-k+1}^{a-1}D_{i}\cdot E\in\text{Psef}^{\, k}(X)\subseteq\text{Nef}^{\, k}(X)$. So $\prod\limits_{i=1}^{a-1}D_{i}\cdot A_{0}\cdot E\geq 0$.

By the choices of $H_{i}$ on $Y$, $\prod\limits_{i=1}^{a'}\pi^{*}H_{i}\cdot A_{0}=\sum\limits_{j=1}^{t}\lambda_{j}\pi^{-1}(y_{j})$, where $t\geq 1$, each $\lambda_{j}>0$, and each $\pi^{-1}(y_{j})$ is an irreducible closed subvariety of $X$ of dimension $a-a'$. Hence, $\prod\limits_{i=1}^{a-1}D_{i}\cdot A_{0}=\lambda C\in N_{1}(X)_{\mathbb{R}}$ for some $\lambda>0$. So $\prod\limits_{i=1}^{a-1}D_{i}\cdot A_{0}\cdot E<0$. We get a contradiction.
\end{proof}

Suppose that $X$ is a smooth projective $G$-spherical variety. By Theorem \ref{Nef<=Eff=Psef}$(i)(ii)$, for any integer $k$, $\text{Nef}^{\, k}(X)\subseteq\text{Eff}^{\, k}(X)=\text{Psef}^{\, k}(X)$. By \cite[Thm. 3.1]{Br93}, every extremal ray of $NE(X)$ can be contracted. So we have the following

\begin{cor} \label{higher codimension spherical}
Let $X$ be a smooth projective $G$-spherical variety of dimension $n$ such that $\text{Nef}^{\, k}(X)=\text{Psef}^{\, k}(X)$ for some $2\leq k\leq n-2$. If there is an extremal ray $R$ of $NE(X)$ such that $R\nsubseteq\text{Nef}_{1}(X)$, then the corresponding contraction $\text{cont}_{R}: X\rightarrow Y$ is birational and the dimension of the exceptional locus $A\subseteq X$ is no more than $\text{min}\{k-1, n-k-1\}$.
\end{cor}

For smooth projective Fano varieties, we can get a strong corollary as follows.

\begin{cor}
Let $X$ be a smooth projective Fano variety of dimension $n$. If there is some integer $2\leq k\leq n-2$ such that $\text{Psef}^{\, k}(X)\subseteq\text{Nef}^{\, k}(X)$, then $\text{Psef}^{\, 1}(X)=\text{Nef}^{\, 1}(X)$.
\end{cor}

\begin{proof}
By the duality, we can assume that $k\leq\frac{n}{2}$. Now assume that $R\subseteq\text{Psef}_{1}(X)\backslash\text{Nef}_{1}(X)$ is an extremal ray. Let $X, Y, \pi, A$ be as in Theorem \ref{higher codimension arbitrary varieites}. Note that the existence of $\pi=\text{cont}_{R}$ follows from the fact that $X$ is Fano. Then by Theorem \ref{higher codimension arbitrary varieites}, $\text{dim}(A)\leq k-1$. By \cite[Thm. (0.4)]{Io86}, $2\text{\,dim}(A)\geq\text{dim}(X)+l(R)-1$, where $l(R)=\min\{-K_{X}\cdot C \mid C \text{ is a rational curve and } C\in R \}$. Since $X$ is a Fano variety, $l(R)\geq 1$. Thus, $\text{dim}(A)\geq\frac{n}{2}$. We get a contradiction.
\end{proof}

Let $X$ be a spherical $G/H$-embedding and $Y$ be a $G$-orbit on $X$. By Theorem \ref{correspondence fans and varieties}, $\mathfrak{C}^{c}_{Y}$ is a strictly convex colored cone in $(N_{X})_{\mathbb{Q}}$. In particular, for each $D\in\mathfrak{D}_{Y}$, $\rho(\nu_{D})\neq 0$. Recall that $\mathfrak{D}_{0}(G/H)=\{D\in\mathfrak{D}(G/H)\mid \rho(\nu_{D})=0\}$. Hence, $\mathfrak{D}_{0}(G/H)\subseteq\mathfrak{D}(G/H)\backslash\mathfrak{D}_{X}$.

\begin{lem} \label{D_0 and C_(D_0) are extremal}
Let $X$ be a $\mathbb{Q}$-factorial projective spherical $G/H$-embedding.  Take $D\in\mathfrak{D}_{0}(G/H)$, then

$(i)$ the numerical class $C_{D, Y}$ doesn't depend on the choice of the closed $G$-orbit $Y$, thus, we can denote it by $C_{D}$;

$(ii)$ $\mathbb{R}^{+}C_{D}$ is an extremal ray of $NE(X)$ as well as one of $\text{Nef}_{1}(X)$;

$(iii)$ $\mathbb{R}^{+}D$ is an extremal ray of $\text{Nef}^{\, 1}(X)$ as well as one of $\text{Psef}^{\, 1}(X)$.

Moreover, suppose that $D_{1},\ldots,D_{m}\in\mathfrak{D}_{0}(G/H)$ are pairwise different, then

$(iv)$ $F=\langle C_{D_{1}}, \ldots, C_{D_{m}}\rangle$ is an extremal face of $NE(X)$ as well as one of $\text{Nef}_{1}(X)$;

$(v)$ $F^{*}=\langle D_{1}, \ldots, D_{m}\rangle$ is an extremal face of $\text{Nef}^{\, 1}(X)$ as well as one of $\text{Psef}^{\, 1}(X)$;
\end{lem}

\begin{proof}
$(i)$ Take any $\delta\in\text{Pic}(X)$. Then there is an equality $\delta\cdot C_{D_{\alpha}, Y_{1}}=\delta\cdot C_{D_{\alpha}, Y_{2}}$ for any $Y_{1}, Y_{2}\in S^{c}_{X, G}$ by the formula (2) on page \pageref{eqn. C_(D, Y)}. This shows $(i)$.

\medskip

$(ii)(iii)$ Since $D\in\mathfrak{D}_{0}(G/H)$, by the formula (2) again, if $D'\in\mathfrak{B}(X)\backslash\{D\}$, then $D'\cdot C_{D}=0$. On the other hand, $D\cdot C_{D}=1$. By Proposition \ref{cycles are rat. equiv. to stable ones, picard group on spherical varieties}$(i)$, $C_{D}\in\text{Nef}_{1}(X)$ and $\mathbb{R}^{+}D$  is an extremal ray of $\text{Psef}^{\, 1}(X)$.

By the formulas (1) and (2) on page \pageref{eqn. C_(D, Y)}, we find that for any wall $\mu\in\mathbb{F}_{X}$, the intersection number $D\cdot C_{\mu}=0$ and that for any $Y\in S^{c}_{X, G}$ and any $D'\in\mathfrak{D}(G/H)\backslash(\mathfrak{D}_{Y}\cup\{D\})$, the intersection number $D\cdot C_{D', Y}=0$. On the other hand, $D\cdot C_{D}=1$. By \cite[Thm. 3.2]{Br93} and $(i)$ of this lemma, $D\in\text{Nef}^{\, 1}(X)$ and $\mathbb{R}^{+}C_{D}$ is an extremal ray of $NE(X)$. By Remark \ref{Nef1<=Eff1=Psef1 Q-factorial}, $\text{Nef}_{1}(X)\subseteq\text{Psef}_{1}(X)$ and $\text{Nef}^{\, 1}(X)\subseteq\text{Psef}^{\, 1}(X)$. Thus, $\mathbb{R}^{+}C_{D}$ is an extremal ray of $\text{Nef}_{1}(X)$ and $\mathbb{R}^{+}D$ is an extremal ray of $\text{Nef}^{\, 1}(X)$.

\medskip

$(iv)$ Now choose $D_{1},\ldots,D_{m}\in\mathfrak{D}_{0}(G/H)$ which are pairwise different. Let $V_{1}\subset N_{1}(X)_{\mathbb{R}}$ be the subspace generated by $C_{D_{1}},\ldots,C_{D_{m}}$, and $V_{2}\subseteq N_{1}(X)_{\mathbb{R}}$ be the subspace generated by those $C_{D', Y}$ who are different from $C_{D_{1}},\ldots,C_{D_{m}}$ and those $C_{\mu}$ where $\mu$ runs over the set of the walls of $\mathbb{F}_{X}$. Take any $C\in V_{1}\cap V_{2}$. If $C\neq 0$, then by the formula (2), the fact $C\in V_{1}$ implies that we can choose some $i_{0}$ such that $D_{i_{0}}\cdot C\neq 0$. By the formulas (1) and (2), the fact $C\in V_{2}$ implies that $D_{i_{0}}\cdot C=0$. We get a contradiction. Hence, $V_{1}\cap V_{2}=0$. By \cite[Thm. 3.2]{Br93}, $F$ is an extremal face of $NE(X)$. By $(ii)$ of this lemma, $F\subseteq\text{Nef}_{1}(X)$. By Remark \ref{Nef1<=Eff1=Psef1 Q-factorial}, $\text{Nef}_{1}(X)\subseteq NE(X)$. Hence, the extremal face $F$ of $NE(X)$ is also an extremal face of $\text{Nef}_{1}(X)$.

\medskip

$(v)$ Let $W_{1}\subseteq N^{1}(X)_{\mathbb{R}}$ be the subspace generated by $D_{1},\ldots, D_{m}$, and $W_{2}\subseteq N^{1}(X)_{\mathbb{R}}$ be the subspace generated by those $D\in\mathfrak{B}(X)\backslash\{D_{1},\ldots, D_{m}\}$. By Proposition \ref{cycles are rat. equiv. to stable ones, picard group on spherical varieties}$(ii)$ and \cite[Cor. 1.3(iv)]{Br93}, $W_{1}\cap W_{2}=\{0\}$. Thus, by Proposition \ref{cycles are rat. equiv. to stable ones, picard group on spherical varieties}$(i)$, $F^{*}$ is an extremal face of $\text{Eff}^{\, 1}(X)$. By $(iii)$ of this lemma, $F^{*}\subseteq\text{Nef}^{\, 1}(X)$. By Remark \ref{Nef1<=Eff1=Psef1 Q-factorial}, $\text{Nef}^{\, 1}(X)\subseteq\text{Eff}^{\, 1}(X)=\text{Psef}^{\, 1}(X)$. Hence, the extremal face $F^{*}$ of $\text{Eff}^{\, 1}(X)$ is also an extremal face of $\text{Nef}^{\, 1}(X)$.
\end{proof}

\begin{thm} \label{morphism to G/P_0}
Let $X$ be a complete spherical $G/H$-embedding, and $\mathfrak{D}_{1}\subseteq\mathfrak{D}_{0}(G/H)$ be a subset. Denote by $D_{0}=\sum\limits_{D\in\mathfrak{D}_{1}}D$. Then the following hold.

$(i)$ $D_{0}$ induces a $G$-equivariant morphism $\pi_{0}: X\rightarrow\widetilde{X}$, where $\widetilde{X}=\text{Proj}(R)$, $R=\bigoplus\limits_{m\geq 0}H^{0}(X, mD_{0})$. Moreover, there is a $G$-equivariant isomorphism $\widetilde{X}\rightarrow G/P_{0}$, where $P_{0}$ is a parabolic subgroup of $G$ containing $B$.

$(ii)$ Let $F$ be any fiber of $\pi_{0}$. Then $X$ is  $\mathbb{Q}$-factorial (resp. locally factorial, smooth, projective) if and only if $F$ is $\mathbb{Q}$-factorial (resp. locally factorial, smooth, projective).

$(iii)$ Assume that $X$ is projective and $\mathbb{Q}$-factorial. Then $\text{Nef}^{\, 1}(X)=\text{Psef}^{\, 1}(X)$ if and only if $\text{Nef}^{\, 1}(F)=\text{Psef}^{\, 1}(F)$.
\end{thm}

\begin{rmk}
The conclusion $(iii)$ in Theorem \ref{morphism to G/P_0} is not true for a general $G$-equivariant morphism from a complete $G$-spherical variety to a rational $G$-homogeneous space. Now we consider such an example. Let $X$ be the blowing-up of $\mathbb{P}^{2}$ at the point $p=[1, 0, 0]$. Then $X$ is a Hizebruch surface and it's a $\mathbb{P}^{1}$-bundle over $\mathbb{P}^{1}$. Denote by $\pi: X\rightarrow Y=\mathbb{P}^{1}$ the fibration. Let $G$ be the set of matrices $(a_{i, j})_{1\leq i, j\leq 3}$ in $GL_{3}(\mathbb{C})$ such that the first row is $(1, 0, 0)$ and the first column is $(1, 0, 0)^{t}$. Thus, $G\cong GL_{2}(\mathbb{C})$ is reductive. There is a natural $G$-action on $\mathbb{A}^{3}_{\mathbb{C}}$. It induces a $G$-action on $\mathbb{P}^{2}$ fixing the point $p$. Then $X$ is a $G$-spherical variety. Moreover, there is a natural $G$-action on $Y$ such that $\pi$ is $G$-equivariant and $Y$ is $G$-homogeneous. Let $F$ be any fiber of $\pi$. Note that $\text{Nef}^{\, 1}(F)=\text{Psef}^{\, 1}(F)$ and $Y$ is a rational $G$-homogeneous space, but $\text{Nef}^{\, 1}(X)\neq\text{Psef}^{\, 1}(X)$, since there is a $(-1)$-curve on $X$.
\end{rmk}

\begin{lem} \label{every Cartier divisor is G'-linearizable}
Let $X$ be a normal $G$-variety. Then there exists a connected reductive algebraic group $G'$ and a finite surjective morphism $\pi: G'\rightarrow G$ of algebraic groups such that under the induced $G'$-action on $X$, every Cartier divisor on $X$ is $G'$-linearizable.
\end{lem}

\begin{proof}
We take a similar proof with \cite[Prop. 2.3.1]{Per}. By \cite[Cor. 2.2.7]{Per}, there exists a connected algebraic group $G'$ and a finite surjective morphism $\pi: G'\rightarrow G$ of algebraic groups such that $\text{Pic}(G')=0$. Since $G$ is reductive, $G'$ is also reductive. Denote by $\varphi: G'\times X\rightarrow X, (g', x)\mapsto \pi(g')\cdot x$, i.e. $\varphi$ defines the $G'$-action on $X$ induced by the $G$-action. Let $p_{G'}$ and $p_{X}$ be the projections from $G'\times X$ to $G'$ and $X$ respectively. Take any Cartier divisor $\delta$ on $X$. By \cite[Lem. 2.2.3]{Per}, there exists a point $x_{0}\in X$ such that $\varphi^{*}\delta=p_{G'}^{*}(\varphi^{*}\delta|_{G'\times\{x_{0}\}})\otimes p_{X}^{*}(\varphi^{*}\delta|_{\{e\}\times X})$. Since $\text{Pic}(G')=0$, $\varphi^{*}\delta|_{G'\times\{x_{0}\}}=\mathcal{O}_{G'}$. Since $e$ has a trivial action on $X$, $\varphi^{*}\delta|_{\{e\}\times X}=\delta$ in $\text{Pic}(X)$. Thus, $\varphi^{*}\delta=p_{X}^{*}\delta$ in $\text{Pic}(G'\times X)$. By \cite[Cor. 2.1.7]{Per}, $\delta$ is $G'$-linearizable.
\end{proof}

\begin{proof}[Proof of Theorem \ref{morphism to G/P_0}$(i)$]
By \cite[Prop. 3.1]{Br89}, $\mathcal{O}_{X}(D_{0})$ is a Cartier divisor.  By \cite[Thm. 1.6]{Su75}, there exists a positive integer $m_{0}$ such that $\mathcal{O}_{X}(m_{0}D_{0})$ is $G$-linearizable. Thus, by \cite[Lem. 2.3.2]{Per}, $H^{0}(X, mm_{0}D_{0})$ is a $G$-module for each $m\geq 0$. For any $m\geq 1$, let $s\in H^{0}(X, mm_{0}D_{0})^{(B)}$ be the global section corresponding to $mm_{0}D_{0}$. Take any $s'\in H^{0}(X, mm_{0}D_{0})^{(B)}$, then there is some element $f\in\mathbb{C}(X)^{(B)}$ such that $s'=fs$. Thus, $\text{div}(f)+mm_{0}D_{0}$ is a $B$-stable effective divisor. In particular, the only possible poles of $f$ are in the support of $D_{0}$. By Proposition \ref{cycles are rat. equiv. to stable ones, picard group on spherical varieties}$(ii)$, $\nu_{D}(f)=0$ for any $D\in\mathfrak{D}_{1}$. Hence, $f$ is a regular function on $X$. Since $X$ is complete, we know that $f$ is a constant function, i.e. $\text{dim } H^{0}(X, mm_{0}D_{0})^{(B)}=1$. Thus, $H^{0}(X, mm_{0}D_{0})$ is an irreducible $G$-module. Hence, $R_{0}=\bigoplus\limits_{m\geq 0}H^{0}(X, mm_{0}D_{0})$ is generated by $H^{0}(X, m_{0}D_{0})$. By \cite[Thm. 3.3]{Br89}, $\mathcal{O}_{X}(m_{0}D_{0})$ is globally generated. Thus, $m_{0}D_{0}$ induces a natural $G$-equivariant morphism $\pi_{0}: X\rightarrow\widetilde{X}$, where $\widetilde{X}=\text{Proj}(R_{0})$.

On the other hand, by Lemma \ref{every Cartier divisor is G'-linearizable}, there exists a connected reductive algebraic group $G'$ and a finite surjective morphism $\pi: G'\rightarrow G$ of algebraic groups such that under the induced $G'$-action on $X$, every Cartier divisor on $X$ is $G'$-linearizable. Let $B'=\pi^{-1}(B)$ and $H'=\pi^{-1}(H)$. Then $S_{X, G'}=S_{X, G}$ and $S_{X, B'}=S_{X, B}$. In particular, $X$ is a complete spherical $G'/H'$-embedding, and the set of prime $B'$-stable divisors is exactly the set of prime $B$-stable divisors. Under the natural inclusion $\chi(B)\subseteq \chi(B')$, we have $M_{G/H}\subseteq M_{G'/H'}$ with a finite index and for all $D\in\mathfrak{B}(X)$, $\rho_{G'/H'}(\nu_{D})=\rho_{G/H}(\nu_{D})$. In particular, for any $D\in\mathfrak{D}_{1}$, $\rho_{G'/H'}(\nu_{D})=0$. Note that $\mathcal{O}_{X}(D_{0})$ is $G'$-linearizable and $H^{0}(X, mD_{0})$ is a $G'$-module for all $m\geq 0$. Therefore, by taking a similar argument with the last paragraph, we can know that $R=\bigoplus\limits_{m\geq 0}H^{0}(X, mD_{0})$ is generated by $H^{0}(X, D_{0})$, and $D_{0}$ induces a natural $G'$-equivariant morphism $\pi': X\rightarrow X'$, where $X'=\text{Proj}(R)$. Note that the inclusion $R_{0}\subseteq R$ induces a natural isomorphism $\text{Proj}(R)\cong \text{Proj}(R_{0})$. We identify them. Hence, $\pi_{0}$ coincides with $\pi'$. Therefore, $D_{0}$ induces a natural $G$-equivariant morphism $\pi_{0}: X\rightarrow \widetilde{X}$, where $\widetilde{X}=\text{Proj}(R)$, and $R=\bigoplus\limits_{m\geq 0}H^{0}(X, mD_{0})$.

We claim that $\widetilde{X}$ is $G$-equivariantly isomorphic to a rational homogeneous space $G/P_{0}$, where $P_{0}$ is a parabolic subgroup of $G$ containing $B$. Otherwise, let $Z$ be a closed $G$-orbit on $\widetilde{X}$ and $\mathcal{O}_{\widetilde{X}}(A)$ be the Cartier divisor on $\widetilde{X}$ such that $H^{0}(\widetilde{X}, A)=H^{0}(X, m_{0}D_{0})$. Thus, $A$ is very ample and $G$-linearizable. There exists some positive integer $m$ such that the sheaf $\mathcal{O}_{\widetilde{X}}(mA)\otimes I_{Z}$ is globally generated, where $I_{Z}$ is the ideal sheaf corresponding to $Z$. On the other hand, $H^{0}(\widetilde{X}, \mathcal{O}_{\widetilde{X}}(mA)\otimes I_{Z})\subseteq H^{0}(\widetilde{X}, \mathcal{O}_{\widetilde{X}}(mA))$ is a nonzero $G$-submodule. However, $H^{0}(\widetilde{X}, \mathcal{O}_{\widetilde{X}}(mA))=H^{0}(X, \mathcal{O}_{X}(mm_{0}D_{0}))$ is an irreducible $G$-module and $\mathcal{O}_{\widetilde{X}}(mA)$ is globally generated. Hence, $\mathcal{O}_{\widetilde{X}}(mA)\otimes I_{Z}=\mathcal{O}_{\widetilde{X}}(mA)$, which is contradicted with the assumption that $Z\neq X$. The conclusion follows.
\end{proof}

Before proving Theorem \ref{morphism to G/P_0}$(ii)(iii)$, we make some remarks. Firstly, note that it suffices to prove these two conclusions for a special fiber. Our approach is as follows. We take a special fiber $X_{0}$ and show that $X_{0}$ is a spherical variety. Roughly speaking, $\mathbb{F}_{X}\subseteq\mathbb{F}_{X_{0}}$ and their difference $\mathbb{F}_{X_{0}}\backslash\mathbb{F}_{X}$ can be described clearly. In this case, we say that $\mathbb{F}_{X}$ is immersed into $\mathbb{F}_{X_{0}}$, (for the precise meaning, see Definition \ref{defi. of immersion of colored fans}). The study of these two colored fans helps us to complete the proof of Theorem \ref{morphism to G/P_0}.

\begin{defi} \label{defi. of immersion of colored fans}
Let $X$ be a spherical $G/H$-embedding and $X'$ be a spherical $G'/H'$-embedding, where $G, G'$ are connected reductive algebraic groups. We say that there is an immersion from $\mathbb{F}_{X}$ to $\mathbb{F}_{X'}$ if the following conditions are satisfied.

$(i)$ There is an isomorphism of abelian groups $\phi: N_{G/H}\rightarrow N_{G'/H'}$. We also use $\phi$ to denote the linear extension to $(N_{G/H})_{\mathbb{Q}}\rightarrow (N_{G'/H'})_{\mathbb{Q}}$.

$(ii)$ $\phi(\mathcal{V}(G/H))\subseteq\mathcal{V}(G'/H')$.

$(iii)$ There is a bijective map $\varphi: S(X, X')\rightarrow\mathcal{V}_{X'}\cup\mathfrak{D}_{X'}$ such that $\mathcal{V}_{X}\cup\mathfrak{D}_{X}\subseteq S(X, X')\subseteq\mathfrak{B}(X)$ and for any $D\in S(X, X')$, there is an equality $\rho_{G'/H'}(\nu_{\varphi(D)})=\phi(\rho_{G/H}(\nu_{D}))$.

$(iv)$ There is an injective map $\psi: S_{X, G}\rightarrow S_{X', G'}$ such that $\psi(S_{X, G})=\{Y'\in S_{X', G'}\mid \mathfrak{C}_{Y'}^{o}\cap\phi(\mathcal{V}(G/H))\neq\emptyset\}$, and for any $Y\in S_{X, G}$ and any $D\in S(X, X')$, $D\supseteq Y$ if and only if $\varphi(D)\supseteq\psi(Y)$.
\end{defi}

\begin{rmk} \label{explanation of defi. of immmersion of colored fans}
Keep notations as in Definition \ref{defi. of immersion of colored fans} and assume that all the conditions $(i)-(iv)$ in it are satisfied.

$(i)$ We can define an injective map $\Phi: \mathbb{F}_{X}\rightarrow\mathbb{F}_{X'}$ such that for any $Y\in S_{X, G}$, $\Phi(\mathfrak{C}^{c}_{Y})=\mathfrak{C}^{c}_{\psi(Y)}$. We say that $\Phi$ is induced by $\phi$. We identify $M_{X}$ with $M_{X'}$ as groups and identify $S(X, X')$ with $\mathcal{V}_{X'}\cup\mathfrak{D}_{X'}$ in the set theory. Thus, $\Phi$ is a natural inclusion, i.e. $\mathbb{F}_{X}\subseteq\mathbb{F}_{X'}$. This is the reason why we use the word ``immersion" in Definition \ref{defi. of immersion of colored fans}.
 Moreover, we know that the set $\mathbb{F}_{X'}\backslash\mathbb{F}_{X}=\{\mathfrak{C}^{c}_{Y'}\in\mathbb{F}_{X'}\mid \mathfrak{C}_{Y'}^{o}\cap\mathcal{V}(G/H)=\emptyset\}$.
When there is an immersion $\Phi: \mathbb{F}_{X}\rightarrow\mathbb{F}_{X'}$, we will always identify $M_{X}$ with $M_{X'}$, $S(X, X')$ with $\mathcal{V}_{X'}\cup\mathfrak{D}_{X'}$, and $\Phi$ with the natural inclusion.

\medskip

$(ii)$ Assume that $S(X, X')\subseteq\Omega\subseteq\mathfrak{B}(X)$. If there is an injective map $\Psi: \Omega\rightarrow \mathfrak{B}(X')$ such that $\Psi|_{S(X, X')}=\varphi$, and for any $D\in\Omega$, there is an equality $\rho_{G'/H'}(\nu_{\Psi(D)})=\phi(\rho_{G/H}(\nu_{D}))$, then we say that the immersion $\Phi: \mathbb{F}_{X}\rightarrow\mathbb{F}_{X'}$ is compatible with $\Psi$. We also say $\phi: N_{G/H}\rightarrow N_{G'/H'}$ is compatible with $\Psi$, or say that $\phi^{*}: M_{G'/H'}\rightarrow M_{G/H}$ is compatible with $\Psi$, where $\phi^{*}$ is the dual map of $\phi$.

\medskip

$(iii)$ If moreover $\mathcal{V}(G/H)=\mathcal{V}(G'/H')$, then $\mathbb{F}_{X}=\mathbb{F}_{X'}$. We say that $\mathbb{F}_{X}$ is isomorphic to $\mathbb{F}_{X'}$. When there is no confusion, we also say that $\mathbb{F}_{X}$ is identified with $\mathbb{F}_{X'}$.

\medskip

$(iv)$ Assume that $G=G'$ and $H\subseteq H'$. Then there is a natural inclusion $M_{G/H'}\rightarrow M_{G/H}$. By Theorem \ref{morphism of fan}, there is a corresponding $G$-equivariant morphism $X\rightarrow X'$. If moreover $H=H'$, then $\mathbb{F}_{X}=\mathbb{F}_{X'}$ and $X$ is $G$-equivariantly isomorphic to $X'$ by Theorem \ref{correspondence fans and varieties}. This is the reason why we use the word ``isomorphic" in $(iii)$.
\end{rmk}

The following is the key lemma of Theorem \ref{morphism to G/P_0}, and the lemma itself is meaningful.

\begin{lem} \label{morphism to G/P_0 then F_X <= F_(X_0)}
Keep notations as in Theorem \ref{morphism to G/P_0}. Let $L_{0}=P_{0}\cap P_{0}^{-}$ be the standard Levi factor of $P_{0}$ and $P_{0}^{-}$. Let $\tilde{x}_{0}\in\widetilde{X}$ be the point such that $G_{\tilde{x}_{0}}=P_{0}^{-}$. Denote by $X_{0}=\pi_{0}^{-1}(\tilde{x}_{0})$. Identify $\widetilde{X}$ with $G/P_{0}^{-}$. Then the following hold.

$(a)$ $X_{0}$ is a complete spherical $L_{0}/H_{0}$-embedding, where $H_{0}=L_{0}\cap G_{x_{0}}$, and $x_{0}$ is any point in the open $B$-orbit of $X$ such that $x_{0}\in X_{0}$. Moreover, $B_{0}x_{0}$ is an open subset of $X_{0}$, where $B_{0}=B\cap L_{0}$.

$(b)$ There is a $P_{0}$-equivariant isomorphism $\tau:R_{u}(P_{0})\times X_{0}\rightarrow\pi_{0}^{-1}(B\tilde{x}_{0})$.

$(c)$ The natural morphism $\phi: M_{X}\rightarrow M_{X_{0}}$ induced by the restriction of rational functions is an isomorphism of free abelian groups.

$(d)$ The map $\Psi: \mathfrak{B}(X)\backslash\mathfrak{D}_{1}\rightarrow\mathfrak{B}(X_{0}), D\mapsto D\cap X_{0}$ is bijective and it's compatible with the morphism $\phi$. And $\pi_{0}^{\sharp}: \mathfrak{D}(G/P_{0}^{-})\rightarrow\mathfrak{D}_{1}, D'\rightarrow\pi_{0}^{-1}(D')$ is a bijective map.

$(e)$ The isomorphism $\phi^{-1}: M_{X_{0}}\rightarrow M_{X}$ induces an immersion $\Phi: \mathbb{F}_{X}\rightarrow \mathbb{F}_{X_{0}}$ of colored fans, where $\Phi(\mathfrak{C}^{c}_{Y})=\mathfrak{C}^{c}_{Y\cap X_{0}}$ for all closed $G$-orbit on $X$.

$(f)$ If $\mathfrak{D}_{1}=\mathfrak{D}_{0}(G/H)$, then $\mathfrak{D}_{0}(L_{0}/H_{0})=\emptyset$.
\end{lem}

\begin{proof}
$(a)$ Firstly, note that $B\tilde{x}_{0}$ is the open $B$-orbit on $\widetilde{X}$. Since $\widetilde{X}$ is $G$-homogeneous and $\pi_{0}$ is $G$-equivariant, for any $Y\in S_{X, G}$, $\pi_{0}(Y)=\widetilde{X}$ and $\pi_{0}(Y^{o})=B\tilde{x}_{0}$, where $Y^{o}$ is the open $B$-orbit on $Y$, which exists by \cite[Cor. 2.2]{Kn91}. In particular, we can choose a point $x_{0}$ on the open $B$-orbit of $X$ such that $\pi_{0}(x_{0})=\tilde{x}_{0}$. Without loss of generality, we can assume that $H=G_{x_{0}}$. In particular, $P_{0}^{-}=\{g\in G| gX_{0}=X_{0}\}$ and $H\subseteq P_{0}^{-}$.

Let $B_{0}=B\cap L_{0}=B\cap P_{0}^{-}$ and $H_{0}=H\cap L_{0}$. Then $Bx_{0}\cap X_{0}=\{b\cdot x_{0}\mid b\in B, \pi_{0}(b\cdot x_{0})=\tilde{x}_{0}\}=(B\cap P_{0}^{-})x_{0}=B_{0}x_{0}$. Since $Bx_{0}$ is an open subset of $X$, $B_{0}x_{0}$ is an open subset of $X_{0}$. Moreover, $(L_{0})_{x_{0}}=G_{x_{0}}\cap L_{0}=H\cap L_{0}=H_{0}$. Since $X$ is complete and $X_{0}$ is a closed subscheme of $X$, we know that $X_{0}$ is complete. Hence, to show that $X_{0}$ is a complete spherical $L_{0}/H_{0}$-embedding, we only need to show that $X_{0}$ is irreducible and normal.

Now we consider the following diagram $(*)$:
\begin{eqnarray*}
\xymatrix{G\times X_{0}\ar[r]^-{\theta}\ar[d]_-{p}&X\ar[d]^-{\pi_{0}}\\
G\ar[r]^-{\pi}&G/P_{0}^{-}.
}
\end{eqnarray*}

In the diagram $(*)$, $\theta(g, x)=g\cdot x$, $p(g, x)=g$, and $\pi(g)=g\cdot\tilde{x}_{0}$ for all $g\in G$, and $x\in X_{0}$. It's easy to see that the diagram $(*)$ is cartesian, i.e. $G\times X_{0}$ is the fiber product of $G$ and $X$ over $G/P_{0}^{-}$. Since $\pi$ is a smooth morphism, $\theta$ is a smooth morphism to a normal variety. Thus, $G\times X_{0}$ is normal, which implies that $X_{0}$ is normal. On the other hand, since $\pi_{0}$ is $G$-equivariant and $P_{0}^{-}/H$ is irreducible, the fibers of $\pi_{0}$ are irreducible. In particular, $X_{0}$ is irreducible. Hence, $X_{0}$ is a complete spherical $L_{0}/H_{0}$-embedding and $B_{0}\cdot x_{0}\subseteq X_{0}$ is an open subset.

\medskip

$(b)$ Since $R_{u}(P_{0})\cap G_{\tilde{x}_{0}}=R_{u}(P_{0})\cap P_{0}^{-}=\{e\}$, we know that $R_{u}(P_{0})\tilde{x}_{0}\cong R_{u}(P_{0})$. There is a natural morphism $\tau: R_{u}(P_{0})\times X_{0}\rightarrow R_{u}(P_{0})X_{0}, (v, x)\mapsto v\cdot x$. Note that $BP_{0}^{-}=P_{0}P_{0}^{-}=R_{u}(P_{0})P_{0}^{-}$ in $G$, which implies that $B\tilde{x}_{0}=P_{0}\tilde{x}_{0}=R_{u}(P_{0})\tilde{x}_{0}$ in $X$. Hence, $R_{u}(P_{0})X_{0}=\pi_{0}^{-1}(R_{u}(P_{0})\tilde{x}_{0})=\pi_{0}^{-1}(B\tilde{x}_{0})=\pi_{0}^{-1}(P_{0}\tilde{x}_{0})$ is a $P_{0}$-stable open subset of $X$. Consider the following commutative diagram:
\begin{eqnarray*}
\xymatrix{R_{u}(P_{0})\times X_{0}\ar[r]^-{\tau}\ar[d]_-{\phi_{1}}&R_{u}(P_{0})X_{0}\ar[d]^-{\phi_{2}}\\
R_{u}(P_{0})\tilde{x}_{0}\ar@{=}[r]&R_{u}(P_{0})\tilde{x}_{0}.
}
\end{eqnarray*}

For any $v\in R_{u}(P_{0})$, the morphism $\phi_{1}^{-1}(v\cdot\tilde{x}_{0})\xrightarrow{\tau}\phi_{2}^{-1}(v\cdot\tilde{x}_{0})$ is an isomorphism. Thus, $\tau: R_{u}(P_{0})\times X_{0}\rightarrow R_{u}(P_{0})X_{0}$ is an isomorphism of varieties.

Now we define a $P_{0}$-action on $R_{u}(P_{0})\times X_{0}$ as follows. Since $P_{0}$ is a semi-direct product of $R_{u}(P_{0})$ and $L_{0}$, for every element $p\in P_{0}$, there is a unique pair $(l, v)\in L_{0}\times R_{u}(P_{0})$ such that $p=lv$. Then for any point $(v_{0}, x)\in R_{u}(P_{0})\times X_{0}$, define the action $p\cdot (v_{0}, x)=(lvv_{0}l^{-1}, l\cdot x)$. Thus, the morphism $P_{0}\times R_{u}(P_{0})\times X_{0}\rightarrow R_{u}(P_{0})\times X_{0}, (p, v_{0}, x)\mapsto p\cdot (v_{0}, x)$ defines a $P_{0}$-action on $R_{u}(P_{0})\times X_{0}$, and $\tau$ is a $P_{0}$-equivariant isomorphism.

\medskip

$(c)$ Firstly, we want to use the isomorphism $\tau$ in $(b)$ to show that there is an isomorphism $\mathbb{C}(R_{u}(P_{0})\times X_{0})^{(B)}\cong\mathbb{C}(X_{0})^{(B_{0})}$, which is similar with the proof of \cite[Lem. 10.1.11]{Per}.

Now we define the following two morphisms:
$$\varphi_{1}: \mathbb{C}(R_{u}(P_{0})\times X_{0})^{(B)}\rightarrow\mathbb{C}(X_{0})^{(B_{0})}$$
$$f\mapsto f|_{\{e\}\times X_{0}},$$
and
$$\varphi_{2}: \mathbb{C}(X_{0})^{(B_{0})}\rightarrow\mathbb{C}(R_{u}(P_{0})\times X_{0})^{(B)}$$
$$f_{0}\mapsto (R_{u}(P_{0})\times X_{0}\rightarrow\mathbb{C}, (v_{0}, x)\mapsto f_{0}(x)).$$

We need to show that $\varphi_{2}$ is well-defined. Take $f_{0}\in\mathbb{C}(X_{0})^{(B_{0})}$, $(v_{0}, x)\in R_{u}(P_{0})\times X_{0}$ and $b\in B$. Since $b\in P_{0}$, there is a unique pair $(l, v)\in L_{0}\times R_{u}(P_{0})$ such that $b^{-1}=lv$. Thus, $v\in B$ and $l\in L_{0}\cap B=B_{0}$. Let $f=\varphi_{2}(f_{0})$. We get that $(b\cdot f)(v_{0}, x)=f(b^{-1}\cdot (v_{0}, x))=f(lvv_{0}l^{-1}, l\cdot x)=f_{0}(l\cdot x)=(l^{-1}\cdot f_{0})(x)=\chi_{f_{0}}(l^{-1})f_{0}(x)$. Note that $T\subseteq B_{0}$ and $\chi(B_{0})=\chi(T)=\chi(B)$. Thus, $\chi_{f_{0}}\in\chi(B)$ and $\chi_{f_{0}}(b)=\chi_{f_{0}}(v^{-1})\chi_{f_{0}}(l^{-1})=\chi_{f_{0}}(l^{-1})$, i.e. $b\cdot f=\chi_{f_{0}}(b)f$. Hence, $f\in\mathbb{C}(R_{u}(P_{0})\times X_{0})^{(B)}$, and $\varphi_{2}$ is well-defined.

It's clear that $\varphi_{1}$ is well-defined and $\varphi_{1}$, $\varphi_{2}$ are the inverse of each other. Therefore, $\mathbb{C}(R_{u}(P_{0})\times X_{0})^{(B)}\cong\mathbb{C}(X_{0})^{(B_{0})}$.

By the isomorphism $\tau$ in $(b)$, $\mathbb{C}(X)^{(B)}=\mathbb{C}(\pi_{0}^{-1}(B\tilde{x}_{0}))^{(B)}\cong\mathbb{C}(R_{u}(P_{0})\times X_{0})^{(B)}\cong\mathbb{C}(X_{0})^{(B_{0})}$. Thus, there is an isomorphism of groups $\phi: M_{X}\rightarrow M_{X_{0}}$. Denote by $\phi^{*}$ the dual map $N_{X_{0}}\rightarrow N_{X}$.

\medskip

$(d)$ Take any $D'\in\mathfrak{D}(G/P_{0}^{-})$. Since $\pi_{0}$ is $G$-equivariant and the fibers are irreducible, $\pi_{0}^{-1}(D')$ is an irreducible $B$-stable divisor, i.e. $\pi_{0}^{-1}(D')\in\mathfrak{D}(G/H)$. On the other hand, by the definition of $\pi_{0}$, for any $D''\in\mathfrak{D}_{1}$, $\pi_{0}(D'')\in\mathfrak{D}(G/P_{0}^{-})$. Moreover, $\sum\limits_{D\in\mathfrak{D}_{1}}\pi_{0}(D)$ is an ample $B$-stable divisor on $G/P_{0}^{-}$. By the intersection theory on rational homogeneous spaces (see for example \cite[Prop. 1.3.6]{Br05}), this implies that $D'\subseteq\bigcup\limits_{D\in\mathfrak{D}_{1}}\pi_{0}(D)$. So $\pi_{0}^{-1}(D')\in\mathfrak{D}_{1}$, $\pi_{0}^{-1}\pi_{0}(D'')=D''$ and $\pi_{0}\pi_{0}^{-1}(D')=D'$. Therefore, $\pi_{0}^{\sharp}: \mathfrak{D}(G/P_{0}^{-})\rightarrow\mathfrak{D}_{1}, D'\rightarrow\pi_{0}^{-1}(D')$ is a bijective map, and for any $D\in\mathfrak{B}(X)$, $\pi_{0}(D)\neq G/P_{0}^{-}$ if and only if $D\in\mathfrak{D}_{1}$.

Let $S_{0}=\{D\in\mathfrak{B}(X)\mid D\cap\pi_{0}^{-1}(B\tilde{x}_{0})\neq\emptyset\}$. By $(b)$, there is a bijective map $S_{0}\rightarrow\mathfrak{B}(X_{0}), D\mapsto D\cap X_{0}$. We claim that $S_{0}=\mathfrak{B}(X)\backslash\mathfrak{D}_{1}$. Note that the claim is equivalent to the assertion that for any $D\in\mathfrak{B}(X)$, $\pi_{0}(D)=G/P_{0}^{-}$ if and only if $D\in S_{0}$.

Consider the following commutative diagram $(**)$:
\begin{eqnarray*}
\xymatrix{R_{u}(P_{0})\times X_{0}\ar[r]^-{\tau}\ar[d]_-{\phi_{1}}&X\ar[d]^-{\pi_{0}}\\
R_{u}(P_{0})\tilde{x}_{0}\ar[r]^-{i}&G/P_{0}^{-}.
}
\end{eqnarray*}

In the diagram $(**)$, the morphism $i$ is the inclusion and $R_{u}(P_{0})\times X_{0}\rightarrow X$ is the composition $R_{u}(P_{0})\times X_{0}\rightarrow \pi_{0}^{-1}(B\tilde{x}_{0})\subseteq X$, and $\tau$, $\phi_{1}$ are as those in the commutative diagram in the proof of $(b)$.

For any $D\in S_{0}$, $\pi_{0}(D)=B\pi_{0}(D)\supseteq B\tilde{x}_{0}$, i.e. $\pi_{0}(D)=G/P_{0}^{-}$. For any $D\in \mathfrak{B}(X)\backslash S_{0}$, $D\subseteq X\backslash\pi_{0}^{-1}(B\tilde{x}_{0})$. So $\pi_{0}(D)\neq G/P_{0}^{-}$. Hence, the claim holds, i.e. $\Psi: \mathfrak{B}(X)\backslash\mathfrak{D}_{1}\rightarrow\mathfrak{B}(X_{0}), D\mapsto D\cap X_{0}$ is a bijective map.

Now we consider the inverse of the map $\Psi$. By $(b)$, for any $D\in\mathfrak{B}(X_{0})$, $\Psi^{-1}(D)=\overline{R_{u}(P_{0})D}=\overline{BD}$. Hence, for any $f\in\mathbb{C}(X)^{(B)}$, we have $\nu_{D}(f|_{X_{0}})=\nu_{\Psi^{-1}(D)}(f)$, i.e. $\rho_{X_{0}}(D)=\phi(\rho_{X}(\Psi^{-1}(D)))$. So $\Psi$ is compatible with $\phi$. Thus, $(d)$ holds.

\medskip

$(e)$ Remark that by $(c)$, the condition $(i)$ in Definition \ref{defi. of immersion of colored fans} is satisfied and we can identify $(N_{X})_{\mathbb{Q}}$ with $(N_{X_{0}})_{\mathbb{Q}}$ as vector spaces. By $(d)$, the condition $(iii)$ in Definition \ref{defi. of immersion of colored fans} is satisfied. We identify $\mathfrak{B}(X)\backslash\mathfrak{D}_{1}$ with $\mathfrak{B}(X_{0})$ in the set theory, when there is no confusion.

Now we check the condition $(ii)$ in Definition \ref{defi. of immersion of colored fans}. Let $p_{X_{0}}: R_{u}(P_{0})\times X_{0}\rightarrow X_{0}$ be the projection. Define
$$p_{X_{0}}^{*}: \mathbb{C}(X_{0})\rightarrow\mathbb{C}(R_{u}(P_{0})\times X_{0})\cong\mathbb{C}(X),$$
$$f_{0}\mapsto f_{0}\circ p_{X_{0}},$$
and
$$(p_{X_{0}})_{*}: \{\text{valuations on } X\}\rightarrow\{\text{valuations on } X_{0}\},$$
$$\mu\mapsto (\mathbb{C}(X_{0})\backslash\{0\}\rightarrow\mathbb{Q}, f_{0}\mapsto\mu(p_{X_{0}}^{*}f_{0})).$$
Note that $\mathbb{C}(X_{0})^{(B_{0})}\xrightarrow{p_{X_{0}}^{*}}\mathbb{C}(X)^{(B)}$ induces a morphism $M_{X_{0}}\rightarrow M_{X}$, which is exactly the inverse of the isomorphism $\phi: M_{X}\rightarrow M_{X_{0}}$ in $(c)$.

On the other hand, take any $\mu\in\mathcal{V}(G/H)$, i.e. $\mu$ is a $G$-invariant valuation on $X$. Take any $l\in L_{0}$, then $(p_{X_{0}})_{*}\mu (l\cdot f_{0})=\mu(p_{X_{0}}^{*}(l\cdot f_{0}))=\mu(l\cdot p_{X_{0}}^{*}(f_{0}))=\mu(p_{X_{0}}^{*}(f_{0}))=(p_{X_{0}})_{*}\mu (f_{0})$, i.e. $(p_{X_{0}})_{*}\mu$ is a $L_{0}$-invariant valuation on $X_{0}$. Therefore, the condition $(ii)$ in Definition \ref{defi. of immersion of colored fans} is also satisfied.

Now we define a map $\psi: S_{X_{0}, L_{0}}\rightarrow S_{X, G}, Y_{0}\mapsto GY_{0}$. For any $D\in\mathfrak{B}_{GY_{0}}$, $\Psi(D)=D\cap X_{0}\supseteq GY_{0}\cap X_{0}\supseteq Y_{0}$, i.e. $\Psi(D)\in\mathfrak{B}_{Y_{0}}$, which implies that $\mathfrak{C}^{c}_{GY_{0}}\subseteq\mathfrak{C}^{c}_{Y_{0}}$. Since $X=GX_{0}$ and $X_{0}=\bigcup\limits_{Y_{0}\in S_{X_{0}, L_{0}}}Y_{0}$, we know that $X=\bigcup\limits_{Y\in\psi(S_{X_{0}, L_{0}})}Y$, which implies that $\psi$ is surjective.

For any $Y\in S^{c}_{X, G}$, take any $Y_{1}, Y_{2}\in S_{Y\cap X_{0}, L_{0}}$. Note that $\pi_{0}(Y)=G/P_{0}^{-}$, thus, $Y\cap X_{0}\neq\emptyset$. Since $Y$ is $G$-homogeneous, $\psi(Y_{1})=\psi(Y_{2})=Y$. Thus, $\mathfrak{C}_{Y}\subseteq\mathfrak{C}_{Y_{1}}\cap\mathfrak{C}_{Y_{2}}$. Since $X$ is a complete variety, $Y$ is $G$-equivariantly isomorphic to a rational homogeneous space. In particular, as a $G$-spherical variety, $\text{rank}(Y)=0$. Thus, by \cite[Thm. 6.3]{Kn91}, $\text{dim}(\mathfrak{C}_{Y})=\text{rank}(M_{X})$, which implies that $\mathfrak{C}_{Y}^{o}\subseteq\mathfrak{C}_{Y_{1}}^{o}\cap\mathfrak{C}_{Y_{2}}^{o}$. By Theorem \ref{correspondence fans and varieties}, $\mathfrak{C}_{Y_{1}}^{o}\cap\mathfrak{C}_{Y_{2}}^{o}\cap\mathcal{V}(L_{0}/H_{0})\supseteq\mathfrak{C}_{Y}^{0}\cap\mathcal{V}(G/H)\neq\emptyset$. By the definition of colored fans, $\mathfrak{C}^{c}_{Y_{1}}=\mathfrak{C}^{c}_{Y_{2}}$, i.e. $Y_{1}=Y_{2}$. Hence, the $L_{0}$-stable set $Y\cap X_{0}$ is indeed a $L_{0}$-orbit. So the map $\eta: S^{c}_{X, G}\rightarrow S^{c}_{X_{0}, L_{0}}, Y\mapsto Y\cap X_{0}$ is well-defined and $\psi\circ\eta=id_{S^{c}_{X, G}}$.

For any $Y\in S^{c}_{X, G}$ and any $D\in\mathfrak{B}(X_{0})$, if $D\supseteq\eta(Y)$, then $\Psi^{-1}(D)=\overline{BD}\supseteq B(Y^{o}\cap X_{0})=Y^{o}$, where $Y^{o}$ is the open $B$-orbit on $Y$, which exists by \cite[Cor. 2.2]{Kn91}. It should also be noticed that the fact $Y^{o}\cap X_{0}\neq\emptyset$ follows from the facts that $\pi_{0}(Y)=G/P_{0}^{-}$ and $\pi_{0}(Y^{o})$ is the open $B$-orbit on $G/P_{0}^{-}$. Thus, $\mathfrak{C}^{c}_{Y\cap X_{0}}\subseteq\mathfrak{C}^{c}_{Y}$. Hence, $\mathfrak{C}^{c}(Y)=\mathfrak{C}^{c}_{Y\cap X_{0}}$ for all $Y\in S^{c}_{X, G}$.

By the fact $\mathcal{V}(G/H)\subseteq\mathcal{V}(L_{0}/H_{0})$ and Theorem \ref{correspondence fans and varieties}, $\eta$ can be extended uniquely to $S_{X, G}\rightarrow S_{X_{0}, L_{0}}$ satisfying the condition $(iv)$ in Definition \ref{defi. of immersion of colored fans}.

Therefore, we have shown that there is an immersion of colored fans $\Phi: \mathbb{F}_{X}\rightarrow \mathbb{F}_{X_{0}}$ induced by the map $\phi^{-1}: M_{X_{0}}\rightarrow M_{X}$ in the sense of Definition \ref{defi. of immersion of colored fans}, and $\Phi(\mathfrak{C}^{c}_{Y})=\mathfrak{C}^{c}_{Y\cap X_{0}}$ for all $Y\in S^{c}_{X, G}$.

\medskip

$(f)$ It's a direct consequence of the compatibility of $\Psi$ and $\phi$.  More precisely, if $\mathfrak{D}_{1}=\mathfrak{D}_{0}(G/H)$, then by $(d)$, for any $D\in\mathfrak{D}(L_{0}/H_{0})$, $\Psi^{-1}(D)\in\mathfrak{D}(G/H)\backslash\mathfrak{D}_{0}(G/H)$, and $\rho_{L_{0}/H_{0}}(D)=\phi(\rho_{G/H}(\Psi^{-1}(D)))\neq 0$, i.e. $\mathfrak{D}_{0}(L_{0}/H_{0})=\emptyset$.
\end{proof}

Now we can complete the proof of Theorem \ref{morphism to G/P_0}.

\begin{proof}[Proof of Theorem \ref{morphism to G/P_0}$(ii)(iii)$]
Keep notations as in Lemma \ref{morphism to G/P_0 then F_X <= F_(X_0)}.

$(ii)$ Note that by Lemma \ref{morphism to G/P_0 then F_X <= F_(X_0)}$(d)(e)$ and Remark \ref{explanation of defi. of immmersion of colored fans}$(i)$, we get the following fact
\begin{eqnarray}
\mathbb{F}_{X}\subseteq\mathbb{F}_{X_{0}}, \text{ and } \mathbb{F}_{X_{0}}\backslash\mathbb{F}_{X}=\{\mathfrak{C}^{c}_{Y_{0}}\in\mathbb{F}_{X_{0}}\mid \mathfrak{C}_{Y_{0}}^{o}\cap\mathcal{V}(G/H)=\emptyset\}.
\end{eqnarray}

By Proposition \ref{locally fac. Q-fac. criterion}, the fact $(3)$ implies that if $X_{0}$ is locally factorial (resp. $\mathbb{Q}$-factorial), then $X$ is locally factorial (resp. $\mathbb{Q}$-factorial).

Conversely, if $X$ is locally factorial (resp. $\mathbb{Q}$-factorial, smooth), then by Lemma \ref{morphism to G/P_0 then F_X <= F_(X_0)}$(b)$, $X_{0}$ is also locally factorial (resp. $\mathbb{Q}$-factorial, smooth).

If $X_{0}$ is smooth, then the morphism $\pi_{0}$ is a smooth morphism from $X$ to a smooth variety $G/P_{0}^{-}$. Hence, $X$ is a smooth variety.

If $X$ is projective, then its closed subvariety $X_{0}$ is also projective.

Now we assume that $X_{0}$ is projective. Take any Cartier $B_{0}$-stable divisor $\delta_{0}=\sum\limits_{D\in\mathfrak{B}(X_{0})}n_{D}D$. Then by the fact $(3)$, and \cite[Prop. 3.1]{Br89}, the $B$-stable divisor $\delta=\sum\limits_{D\in\mathfrak{B}(X_{0})}n_{D}\Psi^{-1}(D)$ is Cartier. Note that by \cite[Prop. 3.1]{Br89}, the $B$-stable divisor $D_{0}$ is also Cartier.

If $\delta_{0}$ is moreover ample, then by the fact $(3)$, Lemma \ref{morphism to G/P_0 then F_X <= F_(X_0)}$(d)$, Definition \ref{defi. of colored fans}, and the formulas $(1)(2)$ on page \pageref{eqn. C_(D, Y)}, we can get the following two conclusions.

$\bullet$ If $\mu$ is a wall in $\mathbb{F}_{X}$, then $\mu_{0}$ is a wall in $\mathbb{F}_{X_{0}}$, where $\mu=\phi^{*}(\mu_{0})$ and $\phi^{*}: (N_{X_{0}})_{\mathbb{Q}}\rightarrow (N_{X})_{\mathbb{Q}}$ is the dual map induced by the isomorphism $\phi$ in Lemma \ref{morphism to G/P_0 then F_X <= F_(X_0)}$(c)$. Moreover, $(D_{0}+\delta)\cdot C_{\mu}=\delta\cdot C_{\mu}=\delta_{0}\cdot C_{\mu_{0}}>0$.

$\bullet$ If $Y\in S^{c}_{X, G}$ and $D\in\mathfrak{D}(G/H)\backslash(\mathfrak{D}_{1}\cup\mathfrak{D}_{Y})$, then $Y\cap X_{0}\in S^{c}_{X_{0}, L_{0}}$, $\Psi(D)\in\mathfrak{D}(L_{0}/H_{0})\backslash D_{Y\cap X_{0}}$ and $(D_{0}+\delta)\cdot C_{D, Y}=\delta\cdot C_{D, Y}=\delta_{0}\cdot C_{\Psi(D), Y\cap X_{0}}>0$.

Note that if $Y\in S^{c}_{X, G}$ and $D\in\mathfrak{D}_{1}$, then by the formula $(2)$ on page \pageref{eqn. C_(D, Y)}, $(D_{0}+\delta)\cdot C_{D, Y}=D\cdot C_{D, Y}=1>0$. Hence, by \cite[Thm. 3.2$(ii)$]{Br93}, the Cartier divisor $D_{0}+\delta$ is ample, and $X$ is projective.

\medskip

$(iii)$ The ``only if" part: Denote by $l: X_{0}\rightarrow X$ the natural inclusion. Then by Proposition \ref{cycles are rat. equiv. to stable ones, picard group on spherical varieties}$(i)$ and Lemma \ref{morphism to G/P_0 then F_X <= F_(X_0)}$(d)$, $\text{Psef}^{\, 1}(X_{0})\subseteq l^{*}(\text{Psef}^{\, 1}(X))$.  By the assumption that $\text{Psef}^{\, 1}(X)=\text{Nef}^{\, 1}(X)$ and the projection formula, $\text{Psef}^{\, 1}(X_{0})\subseteq\text{Nef}^{\, 1}(X_{0})$. Hence, $\text{Nef}^{\, 1}(X_{0})=\text{Psef}^{\, 1}(X_{0})$.

\medskip

For the ``if" part, we want to apply \cite[Thm. 3.2$(ii)$]{Br93} and the formulas $(1)(2)$ on page \pageref{eqn. C_(D, Y)}.

Let $D_{1}\in\mathfrak{B}(X)\backslash\mathfrak{D}_{1}$, $\mu$ is a wall in $\mathbb{F}_{X}$, $Y\in S^{c}_{X, G}$ and $D\in\mathfrak{D}(G/H)\backslash (\mathfrak{D}_{Y}\cup\mathfrak{D}_{1})$. Then by Lemma \ref{morphism to G/P_0 then F_X <= F_(X_0)}$(d)$, the fact $(3)$, and the formulas $(1)(2)$ on page \pageref{eqn. C_(D, Y)}, we know that $D_{1}\cdot C_{\mu}=\Psi(D_{1})\cdot C_{\mu_{0}}\geq 0$ and $D_{1}\cdot C_{D, Y}=\Psi(D_{1})\cdot C_{\Psi(D), Y\cap X_{0}}\geq 0$, where $\mu_{0}=(\phi^{*})^{-1}(\mu)$ is the corresponding wall of $\mathbb{F}_{X_{0}}$.

On the other hand, by Lemma \ref{D_0 and C_(D_0) are extremal}$(i)(ii)(iii)$, for any $D\in\mathfrak{D}_{1}\subseteq\mathfrak{D}_{0}(G/H)$,  $D\in\text{Nef}^{\, 1}(X)$ and $C_{D}\in\text{Nef}_{1}(X)$. Hence, by Proposition \ref{cycles are rat. equiv. to stable ones, picard group on spherical varieties}$(ii)$ and \cite[Thm. 3.2(ii)]{Br93}, $\text{Nef}^{\, 1}(X)=\text{Psef}^{\, 1}(X)$.
\end{proof}

\begin{rmk} \label{X is a quotient of G*X_0 by (P_0)-}
Keep notations as in Theorem \ref{morphism to G/P_0} and Lemma \ref{morphism to G/P_0 then F_X <= F_(X_0)}. Define $G\times^{P_{0}^{-}}X_{0}$ as the set of the quotient of $G\times X_{0}$ by equivalences $(gp, x)\sim(g, px)$ for all $g\in G, p\in P_{0}^{-}$ and $x\in X_{0}$. Then $G\times^{P_{0}^{-}}X_{0}$ is bijective to $X$ as sets. Thus, $X$ is the geometric quotient of $G\times X_{0}$ by $P_{0}^{-}$.

In general, $X$ is not isomorphic to the product $G/P_{0}^{-}\times X_{0}$. In Subsection \ref{subsection horospherical varieties}, when discussing the pseudo-effective and nef cones of smooth projective horospherical varieties, we will give a conterexample (see Example \ref{example}). The parabolic subgroup $P_{0}^{-}$ has a close relationship with the set $\mathfrak{D}_{1}$, especially in the horospherical cases (see Theorem \ref{morphism to G/P_0 horospherical}). Moreover, we will prove the converse of Theorem \ref{morphism to G/P_0} in the horospherical cases (see Theorem \ref{converse of morphism to G/P_0 horospherical}).
\end{rmk}

\subsection{Toric varieties} \label{subsection toric varieties}

By the discussions in \cite[Chapter 5]{Fult93}, we can define an intersection theory on $\mathbb{Q}$-factorial complete toric varieties with the intersection numbers in $\mathbb{Q}$. The main result in this subsection is Theorem \ref{Nef=Psef toric}, who says that if $X$ is a projective $\mathbb{Q}$-factorial variety such that $\text{Nef}^{\, k}(X)=\text{Psef}^{\, k}(X)$ for some $1\leq k\leq\text{dim}(X)-1$, then $X$ is a quotient of the product of some projective spaces by a finite group.

Let $X$ be a $\mathbb{Q}$-factorial complete toric variety, and $\mathbb{F}_{X}$ be the fan of $X$. Take a cone $\sigma\in\mathbb{F}_{X}$, and assume $v_{1}, \ldots, v_{k}$ to be the primitive lattice points along rays of $\sigma$. Let $N_{\sigma}=N_{X}\cap\text{span}(\sigma)$, where $\text{span}(\sigma)\subseteq (N_{X})_{\mathbb{Q}}$ is the subspace generated by $\sigma$. Define $\text{mult}(\sigma)=[N_{\sigma}: \sum\limits_{i=1}^{k}\mathbb{Z}v_{i}]$. Let $\tau\in\mathbb{F}_{X}$ be a cone such that $\sigma\cap\tau=\{0\}$, then define
\begin{eqnarray*}
V(\sigma)\cdot V(\tau)=\left\{ \begin{array}{ll}
\frac{\text{mult}(\sigma).\text{mult}(\tau)}{\text{mult}(\gamma)}V(\gamma), &\text{if } \gamma=\langle\sigma,\tau\rangle\in\mathbb{F}_{X},\\
0, &\text{otherwise},
\end{array} \right.
\end{eqnarray*}
where $V(\sigma)$, $V(\tau)$, and $V(\gamma)$ are the orbit closures on $X$ of the corresponding cones. When $X$ is a smooth complete toric variety, this intersection theory coincides with the usual one.

\begin{rmk} \label{Nef<=Eff=Psef proj. Q-fac. toric}
Let $X$ be a $\mathbb{Q}$-factorial complete toric variety. By the intersection theory on it, the cones $\text{Eff}^{\, k}(X)$, $\text{Psef}^{\, k}(X)$, and $\text{Nef}^{\, k}(X)$ are well-defined.

$(i)$ Let $Y$ be another $\mathbb{Q}$-factorial complete toric variety and $f: Y\rightarrow X$ be a toric morphism. We want to show the projection formula for $f$. More precisely, we claim that the equality $f_{*}(f^{*}\eta\cdot\xi)=\eta\cdot f_{*}\xi$ holds for all $\eta\in A^{k}(X)_{\mathbb{Q}}$ and $\xi\in A^{m}(Y)_{\mathbb{Q}}$, where $0\leq k\leq\text{dim}(X)$ and $0\leq m\leq\text{dim}(Y)$.

By Proposition \ref{cycles are rat. equiv. to stable ones, picard group on spherical varieties}$(i)$ and the linearity of the formulas at the two sides, to show the claim, we can assume that $\eta$ is represented by an orbit closure. Since $X$ is $\mathbb{Q}$-factorial, there exist Cartier divisors $\eta_{1},\ldots,\eta_{k}$ and a positive rational number $\lambda$ such that $\eta=\lambda\prod\limits_{i=1}^{k}\eta_{i}$. Thus, by the induction on $k$ and the projection formula for Cartier divisors, $f_{*}(f^{*}\eta\cdot\xi)=f_{*}(f^{*}(\lambda\prod\limits_{i=1}^{k}\eta_{i})\cdot\xi)=\lambda f_{*}(\prod\limits_{i=1}^{m}f^{*}\eta_{i}\cdot\xi)=\lambda\prod\limits_{i=1}^{k}\eta_{i}\cdot f_{*}\xi=\eta\cdot f_{*}\xi$.

$(ii)$ By $(i)$, we can see that the proof of Theorem \ref{Nef<=Eff=Psef} holds for $X$ without any change. In particular, for any integer $k$, $\text{Nef}^{\, k}(X)\subseteq\text{Eff}^{\, k}(X)=\text{Psef}^{\, k}(X)$ and the products of nef cycle classes are nef.
\end{rmk}

\begin{thm} \label{Nef=Psef toric}
Let $X$ be a $\mathbb{Q}$-factorial projective toric variety of dimension $n$. Then the following are equivalent.

$(i)$ There exists some $1\leq k\leq n-1$ such that $\text{Psef}^{\, k}(X)=\text{Nef}^{\, k}(X)$.

$(ii)$ For all $1\leq k\leq n-1$, $\text{Psef}^{\, k}(X)=\text{Nef}^{\, k}(X)$.

$(iii)$ There is a finite surjective toric morphism $ f: \mathbb{P}^{d_{1}}\times\ldots\times\mathbb{P}^{d_{\rho}}\rightarrow X$, where $\rho$ is the Picard number of $X$ and $d_{1},\ldots,d_{\rho}$ are some positive integers such that their sum is just $n$.

If $X$ is moreover smooth, then the conditions above are also equivalent to

$(iv)$  $X\cong\mathbb{P}^{d_{1}}\times\ldots\times\mathbb{P}^{d_{\rho}}$, where $\rho, d_{1},\ldots,d_{\rho}$ are as in $(iii)$.
\end{thm}

\begin{proof} Note that by the projection formula, $(iii)$ implies $(ii)$. Hence, by \cite[Prop. 5.3]{FS09}, to complete the proof, we only need to show the following two claims for the case when $X$ is a $\mathbb{Q}$-factorial projective toric variety of dimension $n$.

\textit{Claim 1: If $\text{Psef}^{\, 1}(X)=\text{Nef}^{\, 1}(X)$, then for all $2\leq k\leq n-1$, $\text{Psef}^{\, k}(X)=\text{Nef}^{\, k}(X)$.}

\textit{Claim 2: If $\text{Psef}^{\, k}(X)=\text{Nef}^{\, k}(X)$ for some $2\leq k\leq n-1$, then $\text{Psef}^{\, 1}(X)=\text{Nef}^{\, 1}(X)$.}

\medskip

By Proposition \ref{cycles are rat. equiv. to stable ones, picard group on spherical varieties}$(i)$, for both claims, we only need to analysis the intersections of orbit closures.

The claim 1 is a direct consequence of the fact that the products of nef cycles on $X$ are nef and the fact that each orbit closure on $X$ is the intersection of some prime divisors up to a multiplicity.

To show the claim 2, it suffices to show that for any $(n-1)$-dimensional cone $\sigma\in\mathbb{F}_{X}$ and any primitive lattice point $v$ which generates a ray of $\mathbb{F}_{X}$, the intersection number  $V(\sigma)\cdot V(v)$ is nonnegative .

Assume that $\sigma=\langle v_{1},\ldots,v_{n-1}\rangle$, where $v_{i}$ is the primitive lattice point on a ray $R_{i}$ of $\sigma$. By reordering $v_{i}$, we can assume that $v_{i}\neq v$ for all $i=1,\ldots,k-1$. Since $X$ is $\mathbb{Q}$-factorial, $\mathbb{Q}^{+}v\cap\langle v_{1}, \cdots, v_{k-1}\rangle=\{0\}$.

If $v,v_{1},\ldots,v_{k-1}$ don't generate any $k$-dimensional cone in $\mathbb{F}_{X}$, then $V(v)\cdot\prod\limits_{i=1}^{k-1}V(v_{i})=0$.

If $v,v_{1},\ldots,v_{k-1}$ generate a cone $\gamma$ in $\mathbb{F}_{X}$, then $V(v)\cdot\prod\limits_{i=1}^{k-1}V(v_{i})=\frac{1}{\text{mult}(\gamma)}V(\gamma)$.

All in all, $V(v)\cdot\prod\limits_{i=1}^{k-1}V(v_{i})\in\text{Psef}^{\, k}(X)=\text{Nef}^{\, k}(X)$. Denote by $\tau$ the face of $\sigma$ generated by $v_{k},\ldots,v_{n-1}$, then $\prod\limits_{i=k}^{n-1}V(v_{i})=\frac{1}{\text{mult}(\tau)}V(\tau)\in\text{Psef}^{\, n-k}(X)$. Hence, $V(\sigma)\cdot V(v)=\text{mult}(\sigma)V(v)\cdot\prod\limits_{i=1}^{n-1}V(v_{i})=\frac{\text{mult}(\sigma)}{\text{mult}(\tau)} (V(v)\cdot\prod\limits_{i=1}^{k-1}V(v_{i}))\cdot V(\tau)\geq 0$. Then the claim 2 holds. The conclusion follows.
\end{proof}

\subsection{Toroidal varieties} \label{subsection toroidal varieties}

In this subsection, we show in Theorem \ref{nef=psef toroidal} that if $X$ is a smooth projective $G$-spherical variety such that $\mathfrak{D}_{X}=\emptyset$ and $\text{Nef}^{\, k}(X)=\text{Psef}^{\, k}(X)$ for some $1\leq k\leq \text{dim}(X)-1$, then $X$ must be isomorphic to a rational homogeneous space. Note that this theorem is a generalization of Theorem \ref{Nef=Psef toric}.

\begin{defi} \label{defi. of toroidal regular log homogeneous}
$(i)$ A spherical $G/H$-embedding $X$ is toroidal, if  $\mathfrak{D}_{X}=\emptyset$.

$(ii)$ Let $\Gamma$ be a connected linear algebraic group. A $\Gamma$-variety $X$ is said to be regular if the following conditions are satisfied:

$(a)$ there is an open $\Gamma$-orbit on $X$;

$(b)$ the closure of every $\Gamma$-orbit is smooth;

$(c)$ for any $Y\in S_{X, \Gamma}$, if its closure $\overline{Y}\neq X$, then $\overline{Y}$ is the transversal intersection of the orbit closures of codimension one containing $\overline{Y}$;

$(d)$ for any point $x\in X$, the isotropy group $\Gamma_{x}$ has a dense orbit in the normal space of the orbit $\Gamma\cdot x$ in $X$.

$(iii)$ Let $\Gamma$ be a connected algebraic group, $X$ be a smooth $\Gamma$-variety, and $D\subseteq X$ be a $\Gamma$-stable divisor with normal crossings. We say that $X$ is a log homogeneous $\Gamma$-variety with boundary $D$, if the natural morphism of $G$-linearized sheaves $op_{X, D}: \mathcal{O}_{X}\otimes\mathfrak{h}\rightarrow\mathbb{T}_{X}(-\text{log}D)$ is surjective, where $\mathfrak{h}$ is the Lie algebra of $\Gamma$ and $\mathbb{T}_{X}(-\text{log}D)$ is the logarithmic tangent sheaf corresponding to $D$.
\end{defi}

\begin{prop} (\cite[Prop. 2.2.1]{BB96}, \cite[Cor. 2.1.4, Cor. 3.2.2]{Br07}) \label{smooth complete toroidal iff regular iff log homogeneous}
Let $X$ be a smooth complete $G$-variety. Then $X$ is a toroidal spherical variety if and only if it's a regular variety if and only if it's a log homogeneous variety.
\end{prop}

\begin{thm} \label{nef=psef toroidal}
Let $X$ be a smooth projective toroidal $G/H$-embedding of dimension $n$. Then the following are equivalent:

$(i)$ there exists some $1\leq k\leq n-1$ such that $\text{Nef}^{\, k}(X)=\text{Psef}^{\, k}(X)$;

$(ii)$ for all  $1\leq k\leq n-1$, $\text{Nef}^{\, k}(X)=\text{Psef}^{\, k}(X)$;

$(iii)$ $X$ is isomorphic to a rational homogeneous space.
\end{thm}

Note that in the condition $(iii)$ of the theorem, $X$ may be not $G$-homogeneous, but it's homogeneous under the action of a lager group.

\begin{proof}
We depart the condition $(i)$ into the following two conditions:

$(ia)$ there is an equality $\text{Nef}^{\, 1}(X)=\text{Psef}^{\, 1}(X)$;

$(ib)$ there exists some $2\leq k\leq n-2$ such that $\text{Nef}^{\, k}(X)=\text{Psef}^{\, k}(X)$.

\medskip

$(iii)\Rightarrow(ii)$ Since $X$ is a rational homogeneous space, it's well-known that there are two dual bases  of $H^{*}(X)$, see for example \cite[Prop. 1.3.6]{Br05}. Note that for the rational homogeneous space $X$, the natural cycle map $A^{*}(X)\rightarrow H^{*}(X)$ is an isomorphism. Thus, the conclusion $(ii)$ follows.

\medskip

$(ii)\Rightarrow (ib)$ is trivial.

\medskip

$(ib)\Rightarrow(ia)$ Assume that $\text{Nef}^{\, 1}(X)\neq\text{Psef}^{\, 1}(X)$, then by the duality, we can choose an extremal ray $R$ of $NE(X)$ such that $R\nsubseteq\text{Nef}_{1}(X)$. Assume that $E$ is an irreducible $B$-stable divisor such that $E\cdot R<0$. Choose a curve $C\in R$ and let $F=\overline{GC}$. Since $E\cdot C<0$ and $\mathfrak{D}_{X}=\emptyset$, we know that $E\in\mathcal{V}_{X}$ and $E\supseteq F$.

By \cite[Thm. 3.1]{Br93}, there exists a corresponding contraction $\pi=\text{cont}_{R}: X\rightarrow X'$. Thus, $F$ is a component of the exceptional locus. Assume that $a=\text{dim}(F)$ and $b=\text{dim}(\pi(F))$. Note that $F=\overline{GC}$ is a $G$-orbit closure. Let $D_{1}=E, D_{2} \ldots, D_{t}$ be all the boundary divisors containing $F$. By Proposition \ref{smooth complete toroidal iff regular iff log homogeneous} and the definition of regular varieties, $t$ is exactly the codimension of $F$ in $X$, and $\prod\limits_{i=1}^{t}D_{i}=F$ in $A^{*}(X)$.

Take $b$ general very ample divisors $H'_{t+1},\ldots,H'_{t+b}$ on $X'$. Let $D_{j}$ be a general global section of $\pi^{*}H'_{j}$ for $t+1\leq j\leq t+b$. Take $n-t-b-1$ general  ample divisor $D_{t+b+1},\ldots,D_{n-1}$ on $X$. Thus, $\prod\limits_{i=1}^{n-1}D_{i}=\lambda C$ in $N_{1}(X)_{\mathbb{R}}$, where $\lambda>0$. Hence, $E\cdot\prod\limits_{i=1}^{n-1}D_{i}<0$.

On the other hand, $\prod\limits_{i=1}^{k}D_{i}\in\text{Psef}^{\, k}(X)=\text{Nef}^{\, k}(X)$, and $E\cdot\prod\limits_{i=k+1}^{n-1}D_{i}\in\text{Psef}^{\, n-k}(X)$. Thus, $E\cdot\prod\limits_{i=1}^{n-1}D_{i}=(E\cdot\prod\limits_{i=k+1}^{n-1}D_{i})\cdot \prod\limits_{i=1}^{k}D_{i}\geq 0$. We get a contradiction. Thus, the conclusion $(ia)$ holds.

\medskip

$(ia)\Rightarrow(iii)$ Denote by $X_{1}, \ldots, X_{m}$ all the boundary divisors and identify $\partial{X}$ with the divisor $\sum\limits_{i=1}^{m}X_{i}$. By Proposition \ref{smooth complete toroidal iff regular iff log homogeneous}, $X$ is a regular variety. Hence, each $X_{i}$ is a smooth variety.

Denote by $l_{i}: X_{i}\rightarrow X$ the natural inclusion and $\mathcal{N}_{X_{i}/X}$ the normal bundle of $X_{i}$ in $X$. Then by \cite[Prop. 2.3.2, Prop. 4.1.1]{BB96}, there are two short exact sequences
$$0\rightarrow \mathbb{T}_{X}(-\text{log}\partial{X})\rightarrow \mathbb{T}_{X}\rightarrow \bigoplus\limits_{i=1}^{m} (l_{i})_{*}\mathcal{N}_{X_{i}/X}\rightarrow0,$$
and
$$0\rightarrow H^{0}(X, \mathbb{T}_{X}(-\text{log}\partial{X}))\rightarrow H^{0}(X, \mathbb{T}_{X})\rightarrow \bigoplus\limits_{i=1}^{m}H^{0}(X, (l_{i})_{*}\mathcal{N}_{X_{i}/X})\rightarrow 0.$$
Note that the second exact sequence appeared in the proof of \cite[Prop. 4.11]{BB96}.

Now we consider the following commutative diagram:
\begin{eqnarray*}
\xymatrix{
H^{0}(X, \mathbb{T}_{X}(-\text{log}\partial{X}))\otimes\mathcal{O}_{X}\ar@{^{(}->}[r]\ar[d]^-{\phi_{1}}&H^{0}(X, \mathbb{T}_{X})\otimes\mathcal{O}_{X}\ar@{>>}[r]\ar[d]^-{\phi_{2}}& \bigoplus\limits_{i=1}^{m}H^{0}(X, (l_{i})_{*}\mathcal{N}_{X_{i}/X})\otimes\mathcal{O}_{X}\ar[d]^-{\phi_{3}}\\
\mathbb{T}_{X}(-\text{log}\partial{X})\ar@{^{(}->}[r]&\mathbb{T}_{X}\ar@{>>}[r]&\bigoplus\limits_{i=1}^{m}(l_{i})_{*}\mathcal{N}_{X_{i}/X}.
}
\end{eqnarray*}

By Proposition \ref{smooth complete toroidal iff regular iff log homogeneous} and the definition of log homogeneous varieties, $\mathbb{T}_{X}(-\text{log}\partial{X})$ is globally generated, i.e. $\phi_{1}$ is surjective. By $(ia)$, each $X_{i}$ is a nef Cartier divisor on $X$. By \cite[Cor. 3.2.11]{Per12}, $\mathcal{O}_{X}(X_{i})$ is globally generated, which implies that $\mathcal{N}_{X_{i}/X}=(\mathcal{O}_{X}(X_{i}))|_{X_{i}}$ is globally generated. Thus, $\phi_{3}$ is surjective. Then by the Short Five Lemma, $\phi_{2}$ is surjective, i.e. the tangent bundle $\mathbb{T}_{X}$ is globally generated, which implies that $X$ is a homogeneous variety. Since spherical varieties are rational (see for example \cite[Cor. 2.1.3]{Per12}), $X$ is isomorphic to a rational homogeneous space.
\end{proof}

\subsection{Horospherical varieties} \label{subsection horospherical varieties}

The main results in this subsection are as follows. Let $X$ be a smooth projective horospherical $G/H$-embedding such that $H\supseteq R_{u}(B)$ and $N_{G}(H)=P_{I}$. Theorem \ref{nef2=psef2 horospherical} says that if $\text{Nef}^{\, 2}(X)=\text{Psef}^{\, 2}(X)$ and $\text{dim}(X)\geq 3$, then $\text{Nef}^{\, 1}(X)=\text{Psef}^{\, 1}(X)$. Corollary \ref{nef1=psef1 horospherical description} says that if $\text{Nef}^{\, 1}(X)=\text{Psef}^{\, 1}(X)$, then there is a $G$-equivariant morphism $\pi: X\rightarrow G/P_{S\backslash\mathfrak{D}_{0}(G/H)}$ and each fiber is isomorphic to the product of some smooth projective $L$-horospherical varieties of Picard number one, where $\mathfrak{D}_{0}(G/H)$ is identified with a subset of $S\backslash I$ by Remark \ref{D=S-I, rho(D_a)=a*} in the following, and $L$ is a Levi factor of $P_{S\backslash\mathfrak{D}_{0}(G/H)}$. Corollary \ref{nef1=psef1 horospherical description} is a direct consequence of Theorem \ref{morphism to G/P_0 horospherical} and Theorem \ref{Nef1=Psef1 horospherical is more or less product}.

This subsection is organized as follows. In the the part \ref{subsubsection preliminaries}, we study some properties on horospherical varieties. In the part \ref{subsubsection horospherical codimension two}, we prove Theorem \ref{nef2=psef2 horospherical}, which has been described as above. In the part \ref{subsubsection isomorphic colored fans}, we study the $G$-equivariant morphism $\pi$ mentioned above. In the part \ref{subsubsection horospherical codimension one}, we describe of the fiber of $\pi$.

\subsubsection{Preliminaries} \label{subsubsection preliminaries}

In order to study the smooth projective horospherical varieties whose effective cycles of codimension $k$ are nef for $k=1 \text{ or } 2$, we need to study some related properties on horospherical varieties. These results are easy to prove, but we fail to find references to include them.

Now we describe the contents in this part. Let $X$ be a horospherical $G/H$-embedding such that $H\supseteq R_{u}(B)$. Firstly, we make a remark on the isotropy group $G_{x}$ and its normalizer $N_{G}(G_{x})$ for a point $x\in X$. Then we show in Corollary \ref{correspondence rays and divisors horospherical Q-factorial} that if $X$ is $\mathbb{Q}$-factorial, then there is a one to one correspondence between the rays of $\mathbb{F}_{X}$ and the elements in $\mathcal{V}_{X}\cup\mathfrak{D}_{X}$. We continue to recall the description of the Picard number of $X$ (assuming $X$ to be complete) and the descriptions of the $G$-orbits on $X$. Finally, we study the relationship between the images of $G$-orbits and the images of colored cones in Proposition \ref{all faces are colored faces, maps of orbits, C_u maps to a point}.

\bigskip

\begin{defi}
A spherical $G/H$-embedding $X$ is $G$-horospherical if there is a point $x\in G/H$ such that $G_{x}\supseteq R_{u}(B)$. We also say that $X$ is a horospherical $G/H$-embedding.
\end{defi}

Firstly, we will summarize some well-known results on horospherical varieties, which will be used frequently in the following.

\begin{rmk} \label{D=S-I, rho(D_a)=a*}
Let $X$ be a horospherical $G/H$-embedding such that $H\supseteq R_{u}(B)$. Then the following hold.

$(i)$ $N_{G}(H)$ is a parabolic subgroup of $G$ containing $B$, and $TH=HT=N_{G}(H)$.

$(ii)$ Let $I$ be the subset of $S$ such that $P_{I}=N_{G}(H)$. Then $P_{I}/H$ is isomorphic to a torus, $M_{G/H}\cong \chi(P_{I}/H)$ and $H=\text{Ker}_{P_{I}}M_{G/H}$. We identify $M_{G/H}$ and $\chi(P_{I}/H)$. Moreover, $P_{I}$ is the unique parabolic subgroup $P$ of $G$ containing $B$ such that $H$ is the intersection of some characters of $P$.

$(iii)$ There is a bijective map $S\backslash I\rightarrow \mathfrak{D}(G/H), \alpha\mapsto D_{\alpha}$. When there is no confusions, we will identify them. Moreover, this identity only depends on the triple $(B, T, U)$, where $U=G/H$ is the open $G$-orbit.

$(iv)$ For any $\alpha\in S\backslash I$, $\rho(\nu_{D_{\alpha}})=\alpha^{\vee}|_{M_{G/H}}$, where $\alpha^{\vee}$ is the corresponding coroot.

$(v)$ Let $B'$ be any Borel subgroup of $G$, then by \cite[Cor. 21.3A]{Hu75}, there exists some element $g\in G$ such that $R_{u}(B')=gR_{u}(B)g^{-1}$. Let $x=H/H\in G/H\subseteq X$. Then $G_{g\cdot x}=gHg^{-1}\supseteq R_{u}(B')$. So $X$ is also a horospherical $G/H'$-embedding such that $H'\supseteq R_{u}(B')$ for any given Borel subgroup $B'$ and some corresponding $H'$. Therefore, when saying that $X$ is a horospherical $G/H$-embedding, we can either assume $H\supseteq R_{u}(B)$ or assume $H\supseteq R_{u}(B')$ for any fixed Borel group $B'$ of $G$. This doesn't matter. In particular, we can assume $H\supseteq R_{u}(B^{-})$ and $N_{G}(H)=P_{I}^{-}$ if necessary.
\end{rmk}

Note that by \cite[Prop. 1.3]{Pa06}, $N_{G}(H)$ is a parabolic subgroup $P$ of $G$, and it's the unique parabolic subgroup $P$ of $G$ containing $B$ such that $H$ is the intersection of the kernels of some characters of $P$. Then we can easily know the conclusions $(i)$ and $(ii)$. The conclusion $(iii)$ follows from the Bruhat decomposition of a rational homogeneous space and the one to one correspondence between the $B$-orbits on $G/H$ and the $B$-orbits on $G/P_{I}$ induced by the natural morphism $\pi: G/H\rightarrow G/P_{I}$. The conclusion $(iv)$ follows from the isomorphism in $(ii)$.

For a spherical $G/H$-embedding $X$, we also use $\rho$ to denote the composition of the maps $\mathfrak{B}(X)\rightarrow\{\text{valuations}\}\rightarrow (N_{X})_{\mathbb{Q}}, D\mapsto\nu_{D}\mapsto\rho(\nu_{D})$.

\begin{prop} \label{description of V_Y}
Let $X$ be a spherical $G/H$-embedding and $Y$ be a $G$-orbit on $X$. Then $\mathcal{V}_{Y}=\mathcal{V}_{X}\cap\rho^{-1}(\mathfrak{C}_{Y})$.
\end{prop}

\begin{proof}
By the definition, $\mathcal{V}_{Y}\subseteq\mathcal{V}_{X}\cap\rho^{-1}(\mathfrak{C}_{Y})$.

On the other hand, take any $D\in\mathcal{V}_{X}\cap\rho^{-1}(\mathfrak{C}_{Y})$. Let $X_{Y}$ be the open $G$-subset which is a simple spherical $G/H$-embedding with the unique closed $G$-orbit $Y$. By \cite[Thm. 2.5b]{Kn91}, the center of $\nu_{D}$ exists on $X_{Y}$. This implies that $D\cap X_{Y}\neq\emptyset$, since the center of $\nu_{D}$ on $X$ is $D$. Since $Y$ is the unique $G$-orbit on $X_{Y}$, $D\cap X_{Y}\supseteq Y$. Thus, $D\supseteq Y$ and $D\in\mathcal{V}_{Y}$. The conclusion follows.
\end{proof}

\begin{prop} \label{description of D_Y}
Let $X$ be a spherical $G/H$-embedding and $Y$ be a $G$-orbit. Then $\mathfrak{D}_{Y}\cap\rho^{-1}(\mathcal{V}(G/H))=\{D\in\mathfrak{D}_{X}|\rho(\nu_{D})\in\mathfrak{C}_{Y}\cap\mathcal{V}(G/H)\}$. If $X$ is moreover $G$-horospherical, then $\mathfrak{D}_{Y}=\{D\in\mathfrak{D}_{X}|\rho(\nu_{D})\in\mathfrak{C}_{Y}\}$, and $\mathfrak{B}_{Y}=\{D\in\mathcal{V}_{X}\cup\mathfrak{D}_{X}\mid \rho(\nu_{D})\in\mathfrak{C}_{Y}\}$.
\end{prop}

\begin{proof}
By the definition of $\mathfrak{C}_{Y}$, $\rho(\mathfrak{D}_{Y})\subseteq\mathfrak{C}_{Y}$. On the other hand, assume that $D\in\mathfrak{D}_{Z}\cap\rho^{-1}(\mathfrak{C}_{Y}\cap\mathcal{V}(G/H))$, where $Z$ is a $G$-orbit on $X$.

Let $\mathfrak{C}$ be the minimal face of $\mathfrak{C}_{Y}$ containing $\rho(\nu_{D})$. Then $\rho(\nu_{D})\in\mathfrak{C}^{o}$. Since $\rho(\nu_{D})\in\mathcal{V}(G/H)$, $\mathfrak{C}^{o}\cap\mathcal{V}(G/H)\neq\emptyset$. Then by Definition \ref{defi. of colored fans}, there is a subset $\mathfrak{D}\subseteq\mathfrak{D}(G/H)$ such that $(\mathfrak{C}, \mathfrak{D})\in\mathbb{F}_{X}$. By Theorem \ref{correspondence fans and varieties}, there is a $G$-orbit $V$ on $X$ such that $Y\subseteq\overline{V}$ and $\mathfrak{C}^{c}_{V}=(\mathfrak{C}, \mathfrak{D})$. Note that $\rho(\nu_{D})\in\mathfrak{C}_{V}^{o}$, and $\mathfrak{C}^{c}_{V}$ is a colored face of $\mathfrak{C}^{c}_{Y}$.

A similar discussion implies that there is a $G$-orbit $V'$ on $X$ such that $\rho(\nu_{D})\in\mathfrak{C}_{V'}^{o}$, and $\mathfrak{C}^{c}_{V'}$ is a colored face of $\mathfrak{C}^{c}_{Z}$. In particular, $\rho(\nu_{D})\in\mathfrak{C}_{V}^{o}\cap\mathfrak{C}_{V'}^{o}\cap\mathcal{V}(G/H)$. Then by the condition $(ii)$ in Definition \ref{defi. of colored fans}, $(\mathfrak{C}_{V}, \mathfrak{D}_{V})=(\mathfrak{C}_{V'}, \mathfrak{D}_{V'})$, which implies $V=V'$ by Theorem \ref{correspondence fans and varieties}.

On the other hand, by the fact $D\in\mathfrak{D}_{Z}$ and the definition of colored faces, $D\in\mathfrak{D}_{V'}=\mathfrak{D}_{V}$, i.e. $D\supseteq V$. Hence, $D\supseteq\overline{V}\supseteq Y$ and $D\in\mathfrak{D}_{Y}$. So $\mathfrak{D}_{Y}\cap\rho^{-1}(\mathcal{V}(G/H))=\{D\in\mathfrak{D}_{X}|\rho(\nu_{D})\in\mathfrak{C}_{Y}\cap\mathcal{V}(G/H)\}$.

If $X$ is moreover $G$-horospherical, then the conclusion follows from \cite[Cor. 6.2]{Kn91} and Proposition \ref{description of V_Y}.
\end{proof}

\begin{cor} \label{correspondence rays and divisors horospherical Q-factorial}
Let $X$ be a $\mathbb{Q}$-factorial horospherical $G/H$-embedding. Then there is a bijective map $\psi: \mathcal{V}_{X}\cup\mathfrak{D}_{X}\rightarrow\{\text{rays of }\mathbb{F}_{X}\}, D\mapsto\mathbb{Q}^{+}\rho(\nu_{D})$.
\end{cor}

\begin{proof}
By Proposition \ref{locally fac. Q-fac. criterion}$(ii)$, $\psi$ is well-defined. Take any ray $R$ of $\mathbb{F}_{X}$. By the definition of $\mathbb{F}_{X}$, $R$ is generated by some $D\in\mathcal{V}_{X}\cup\mathfrak{D}_{X}$, i.e. $\psi$ is surjective.

By \cite[Cor. 6.2]{Kn91} and Theorem \ref{correspondence fans and varieties}, there exists a $G$-orbit $Y$ on $X$ such that $\mathfrak{C}_{Y}=R$. If $D'\in\mathcal{V}_{X}\cup\mathfrak{D}_{X}$ also satisfies that $\rho(\nu_{D'})\in R$, then by Proposition \ref{description of D_Y}, $D, D'\in\mathfrak{B}_{Y}$. By Proposition \ref{locally fac. Q-fac. criterion}$(ii)$, $D=D'$, i.e. $\psi$ is injective.
\end{proof}

Let $X$ be a $\mathbb{Q}$-factorial horospherical $G/H$-embedding. Let $R=\mathbb{Q}^{+}e$ be a ray of $\mathbb{F}_{X}$, then we denote by $D_{R}$ or $D_{e}$ the unique element in $\mathcal{V}_{X}\cup\mathfrak{D}_{X}$ such that $\rho^{-1}(R)=\{D_{R}\}=\{D_{e}\}$.

By Proposition \ref{cycles are rat. equiv. to stable ones, picard group on spherical varieties}$(ii)$ and Corollary \ref{correspondence rays and divisors horospherical Q-factorial}, we can get the formula to compute the Picard number of a complete $\mathbb{Q}$-factorial horospherical variety. More precisely, we have the following

\begin{prop} \label{Picard number horospherical}
Let $X$ be a complete $\mathbb{Q}$-factorial $G$-horospherical variety. Then the Picard number of $X$ is $\rho(X)=m-r+d$,  where $m$ is the number of rays in $\mathbb{F}_{X}$, $r=\text{rank}(G/H)$, and $d$ is the number of elements in the set $\mathfrak{D}(G/H)\backslash\mathfrak{D}_{X}$.
\end{prop}

\begin{prop}(\cite[Prop. 2.4]{BM13}) \label{dim. of horospherical var.}
Let $X$ be a horospherical $G/H$-embedding and $Y$ be a $G$-orbit on $X$. Assume that $H\supseteq R_{u}(B)$ and $N_{G}(H)=P_{I}$. Then

$(i)$ $M_{\mathfrak{C}_{Y}}\subseteq \chi(P_{I\cup\mathfrak{D}_{Y}})$, where $M_{\mathfrak{C}_{Y}}=M_{G/H}\cap\mathfrak{C}_{Y}^{\bot}$;

$(ii)$ $Y$ is $G$-equivariantly isomorphic to $G/K$, where $K=\text{Ker}_{P_{I\cup\mathfrak{D}_{Y}}}M_{\mathfrak{C}_{Y}}$;

$(iii)$ $\text{dim}(Y)=\text{rank}(M_{\mathfrak{C}_{Y}})+\text{dim}(G/P_{I\cup\mathfrak{D}_{Y}})$.
\end{prop}

\begin{lem} \label{parabolic containing R_u(B) is unique}
Let $P$ be a parabolic subgroup of $G$ containing $B$. Then $(G/P)^{R_{u}(B)}=\{P/P\}$.
\end{lem}

\begin{proof}
Denote by $x_{0}=P/P$. Assume that $x\in (G/P)^{R_{u}(B)}$, i.e. $G_{x}\supseteq R_{u}(B)$. Since $G/P$ is $G$-homogeneous, there exists an element $g\in G$ such that $g\cdot x_{0}=x$. Thus, $G_{x}=gPg^{-1}$ is a parabolic subgroup. By \cite[Prop. 11.1]{Gi07}, there is a maximal torus $T'\subseteq G_{x}\cap P$. Thus, $G_{x}\supseteq R_{u}(B)T'=B$. Then by \cite[Cor. 23.1]{Hu75}, $G_{x}=P$. In particular, $g\in N_{G}(P)=P$. Hence, $x=g\cdot x_{0}=x_{0}$.
\end{proof}

\begin{cor} \label{description of M_Y and N_G(K)}
Keep notations as in Proposition \ref{dim. of horospherical var.}. Then the following hold.

$(i)$ $K\supseteq H$, $M_{\mathfrak{C}_{Y}}=M_{G/K}$ and $N_{G}(K)=P_{I\cup\mathfrak{D}_{Y}}$.

$(ii)$ Denote by $x=H/H\in G/H$. If there is a $G$-equivariant morphism $\phi: G/H\rightarrow Y$, then $N_{G}(G_{\phi(x)})=P_{I\cup\mathfrak{D}_{Y}}$.
\end{cor}

\begin{proof}
$(i)$ By Proposition \ref{dim. of horospherical var.} and Remark \ref{D=S-I, rho(D_a)=a*}$(i)(ii)$, $K\supseteq H$, $M_{\mathfrak{C}_{Y}}\subseteq M_{G/K}$ and $N_{G}(K)$ is a parabolic subgroup of $G$ such that $N_{G}(K)\subseteq P_{I\cup\mathfrak{D}_{Y}}$.

Note that $G/K$ is homogeneous horospherical variety. Thus, by Remark \ref{D=S-I, rho(D_a)=a*}$(i)$, $\text{dim}(G/K)=\text{rank}(M_{G/K})+\text{dim}(G/N_{G}(K))$. By the facts $M_{\mathfrak{C}_{Y}}\subseteq M_{G/K}$, $N_{G}(K)\subseteq P_{I\cup\mathfrak{D}_{Y}}$ and Proposition \ref{dim. of horospherical var.}$(iii)$, $\text{rank}(M_{G/K})=\text{rank}(M_{\mathfrak{C}_{Y}})$ and $\text{dim}(N_{G}(K))=\text{dim}(P_{I\cup\mathfrak{D}_{Y}})$. Thus, $M_{\mathfrak{C}_{Y}}$ is a subgroup of $M_{G/K}$ with a finite index. Since $N_{G}(K)$ and $P_{I\cup\mathfrak{D}_{Y}}$ are parabolic subgroups with the same dimension and $N_{G}(K)\subseteq P_{I\cup\mathfrak{D}_{Y}}$, we know that $N_{G}(K)=P_{I\cup\mathfrak{D}_{Y}}$.

By Remark \ref{D=S-I, rho(D_a)=a*}$(ii)$, $M_{G/K}=\chi(P_{I\cup\mathfrak{D}_{Y}}/K)\subseteq\chi(P_{I}/H)=M_{G/H}$. Since $M_{\mathfrak{C}_{Y}}$ is a subgroup of $M_{G/K}$ with a finite index, $(M_{\mathfrak{C}_{Y}})_{\mathbb{Q}}=(M_{\mathfrak{C}_{G/K}})_{\mathbb{Q}}$ as subspaces of $(M_{G/H})_{\mathbb{Q}}$. In particular, $M_{G/K}\subseteq M_{G/H}\cap (M_{\mathfrak{C}_{Y}})_{\mathbb{Q}}$. By Proposition \ref{dim. of horospherical var.}$(i)$, $(M_{\mathfrak{C}_{Y}})_{\mathbb{Q}}\subseteq (M_{G/H})_{\mathbb{Q}}\cap\mathfrak{C}_{Y}^{\bot}$. Thus, $M_{G/K}\subseteq M_{G/H}\cap(M_{G/H})_{\mathbb{Q}}\cap\mathfrak{C}_{Y}^{\bot}=M_{G/H}\cap \mathfrak{C}_{Y}^{\bot}=M_{\mathfrak{C}_{Y}}$. Hence, $M_{G/K}=M_{\mathfrak{C}_{Y}}$.

\medskip

$(ii)$ Identify $Y$ with $G/K$. Denote by $y=K/K\in G/K$, $y'=\phi(x)$, $K'=G_{y'}$, $P=P_{I\cup\mathfrak{D}_{Y}}$ and $z=P/P\in G/P$. Denote by $\pi_{1}: G/H\rightarrow Y$ and $\pi_{2}: Y\rightarrow G/P$ the natural morphisms induced by the inclusions $H\subseteq K\subseteq P$. By Lemma \ref{parabolic containing R_u(B) is unique}, $\pi_{2}(y')=z$. Thus, $y'\in(\pi_{2})^{-1}(z)=P/K\subseteq Y$. In particular, there is some element $g\in P$ such that $g\cdot y=y'$. Thus, $K'=gKg^{-1}$ and $N_{G}(K')=gN_{G}(K)g^{-1}=gPg^{-1}=P$.
\end{proof}

The following properties are direct consequences of the corresponding propositions in Section \ref{section spherical varieties}, but we will use them frequently later. For the convenience of discussions, we state them explicitly.

\begin{prop} \label{all faces are colored faces, maps of orbits, C_u maps to a point}
$(i)$ Let $X$ be a horospherical $G/H$-embedding, $Y$ be a $G$-orbit on $X$ and $\sigma$ be a face of $\mathfrak{C}_{Y}$. Then there exists a unique $G$-orbit $V(\sigma)$ on $X$ such that $\mathfrak{C}_{V(\sigma)}=\sigma$. Moreover, $Y\subseteq\overline{V(\sigma)}$ and $\mathfrak{C}^{c}_{V(\sigma)}$ is a colored face of $\mathfrak{C}^{c}_{Y}$.

$(ii)$ Let $X$ be a spherical $G/H$-embedding and $X'$ be spherical $G/H'$-embedding. Assume that $H\subseteq H'$ and there is a $G$-equivariant morphism $\pi: X\rightarrow X'$ extending the natural morphism $G/H\rightarrow G/H'$. Denote by $\phi: (N_{X})_{\mathbb{Q}}\rightarrow (N_{X'})_{\mathbb{Q}}$ the corresponding morphism. Let $Y\subseteq X$ and $Y'\subseteq X'$ be $G$-orbits such that $\phi(\mathfrak{C}_{Y})^{o}\cap\mathfrak{C}_{Y'}^{o}\neq\emptyset$. Then $\pi(Y)=Y'$.

$(iii)$  Let $X$ be a horospherical $G/H$-embedding and $X'$ be a horospherical $G/H'$-embedding. Assume that $X$ is projective and $\mathbb{Q}$-factorial, $H\subseteq H'$ and there is a $G$-equivariant morphism $\pi: X\rightarrow X'$ extending the natural morphism $G/H\rightarrow G/H'$. Denote by $\phi: (N_{X})_{\mathbb{Q}}\rightarrow (N_{X'})_{\mathbb{Q}}$ the corresponding morphism. Let $Y\subseteq X$ and $Y'\subseteq X'$ be two closed $G$-orbits such that $\phi(\mathfrak{C}_{Y})\subsetneqq\mathfrak{C}_{Y'}$. Then there is a wall $\mu$ in $\mathbb{F}_{X}$ such that $\pi(C_{\mu})$ is a point, where $C_{\mu}$ is the curve described in Remark \ref{fix C_u as a curve instead of a class}.
\end{prop}

\begin{proof}
$(i)$ It's a direct consequence of Theorem \ref{correspondence fans and varieties} and \cite[Cor. 6.2]{Kn91}.

\medskip

$(ii)$ By Theorem \ref{morphism of fan} and the fact $\phi(\mathfrak{C}_{Y})^{o}\cap\mathfrak{C}_{Y'}^{o}\neq\emptyset$, we know that $\phi(\mathfrak{C}_{Y})^{o}\subseteq\mathfrak{C}_{Y'}^{o}$. Then by Theorem \ref{morphism of fan} and Theorem \ref{correspondence fans and varieties}, $\pi$ maps $X_{Y}$ into $X'_{Y'}$ where $X_{Y}$ (resp. $X'_{Y'}$) is the simple spherical open subvariety of $X$ (resp. $X'$) with the unique closed $G$-orbit $Y$ (resp. $Y'$). In particular, the $G$-orbit $Z=\pi(Y)$ on $X'$ satisfies that $\overline{Z}\subseteq Y'$. By Theorem \ref{morphism of fan}, $\phi(\mathfrak{C}_{Y})\subseteq\mathfrak{C}_{Z}$ and $\mathfrak{C}_{Z}$ is a face of $\mathfrak{C}_{Y'}$. Since $\mathfrak{C}_{Y'}$ is a strictly convex cone and $\mathfrak{C}_{Z}\cap\mathfrak{C}_{Y'}^{o}\supseteq\phi(\mathfrak{C}_{Y})\cap\mathfrak{C}_{Y'}^{o}\neq\emptyset$, we know that $\mathfrak{C}_{Z}=\mathfrak{C}_{Y'}$. By $(i)$, $Z=Y'$, i.e. $\pi(Y)=Y'$.

\medskip

$(iii)$ By \cite[Thm. 6.3]{Kn91}, $\text{dim}(\mathfrak{C}_{Y})=\text{dim}(N_{X})_{\mathbb{Q}}$. Note that $\phi: (N_{X})_{\mathbb{Q}}\rightarrow (N_{X'})_{\mathbb{Q}}$ is a surjective morphism of vector spaces. Hence, $\text{dim}(\phi(\mathfrak{C}_{Y}))=\text{dim}(N_{X'})_{\mathbb{Q}}$. By Theorem \ref{correspondence fans and varieties}, $\mathfrak{C}_{Y}$ and $\mathfrak{C}_{Y'}$ are strictly convex cones. Then $\phi(\mathfrak{C}_{Y})$ is also a convex cone. Note that $\phi(\mathfrak{C}_{Y}^{o})\subseteq\phi(\mathfrak{C}_{Y})^{o}$, i.e. $\phi(\partial{\mathfrak{C}_{Y}})\supseteq \partial{\phi(\mathfrak{C}_{Y})}$. Note that $\partial{\phi(\mathfrak{C}_{Y})}\cap \mathfrak{C}_{Y'}^{o}\neq\emptyset$. Thus, there exists a face $\mu\subseteq\mathfrak{C}_{Y}$ of codimension one such that $\phi(\mu)\cap\mathfrak{C}_{Y'}^{o}\neq\emptyset$. Thus, $\phi(\mu)^{o}\subseteq\mathfrak{C}_{Y'}^{o}$. By $(i)$, $\mu$ is a wall of $\mathbb{F}_{X}$ and there exists a $G$-orbit $V$ on $X$ such that $\mathfrak{C}_{V}=\mu$. By $(ii)$, $\pi(V)=Y'$.

By Remark \ref{fix C_u as a curve instead of a class}, $C_{\mu}=(\overline{V})^{R_{u}(B)}$. Thus, $\pi(C_{\mu})\subseteq (Y')^{R_{u}(B)}$. Note that $Y'$ is a complete $G$-orbit, then by Lemma \ref{parabolic containing R_u(B) is unique}, there is only one $R_{u}(B)$-fixed point on $Y'$. The conclusion follows.
\end{proof}

\subsubsection{Smooth projective horospherical varieties whose effective cycle classes of codimension two are nef} \label{subsubsection horospherical codimension two}

In this part, we firstly  show in Proposition \ref{exceptional locus of bir. Mori cont. horospherical} that the exceptional locus of a birational Mori contraction of a projective $\mathbb{Q}$-factorial $G$-horospherical variey is irreducible and we can even describe the corresponding colored cone. Then we show that if $X$ is a smooth projective horospherical $G/H$-embedding of dimension $n\geq 3$ such that $\text{Nef}^{\, 2}(X)=\text{Psef}^{\, 2}(X)$, then $\text{Nef}^{\, 1}(X)=\text{Psef}^{\, 1}(X)$.

Let $X$ be a projective $\mathbb{Q}$-factorial spherical $G/H$-embedding of rank $r$ and $\mu$ be wall of $\mathbb{F}_{X}$. Choose primitive lattice points $e_{1},\ldots,e_{r+1}$ in $N_{X}$ such that $\mu=\langle e_{1},\ldots, e_{r-1}\rangle$, $\mu_{+}=\langle e_{1},\ldots, e_{r}\rangle$ and $\mu_{-}=\langle e_{1},\ldots, e_{r-1}, e_{r+1}\rangle$, where $(\mu_{+}, \mathfrak{D}_{+}), (\mu_{-}, \mathfrak{D}_{-}), (\mu, \mathfrak{D})\in \mathbb{F}_{X}$ for some $\mathfrak{D}_{+}, \mathfrak{D}_{-}\subseteq\mathfrak{D}(G/H)$ and $\mathfrak{D}\subseteq\mathfrak{D}_{+}\cap\mathfrak{D}_{-}$. Then there is a unique sequence $a_{1},\ldots, a_{r+1}\in\mathbb{Q}$ such that $\sum\limits_{i=1}^{r+1}a_{i}e_{i}=0$ and $a_{r+1}=1$. Note that $a_{r}>0$.

By reordering $e_{i}$, we can assume that
\begin{eqnarray*}
a_{i}\left\{ \begin{array}{ll}
<0, &1\leq i\leq\alpha,\\
=0, & \alpha+1\leq i\leq\beta,\\
>0, &\beta+1\leq i\leq r+1.
\end{array} \right.
\end{eqnarray*}

If all $a_{i}$ are nonnegative, then let $\alpha=0$. Note that $0\leq\alpha\leq\beta\leq r-1$.

\begin{prop} \label{exceptional locus of bir. Mori cont. horospherical}
Let $X$ be a projective $\mathbb{Q}$-factorial horospherical $G/H$-embedding, and $R$ be an extremal ray of $NE(X)$. Denote by $\pi=\text{cont}_{R}: X\rightarrow X'$ the corresponding contraction. Assume that $\pi$ is a birational morphism. Let $A$ be the exceptional locus of $\pi$ and $A'=\pi(A)$. Then $A$ is an $G$-orbit closure.

$(i)$ Suppose that $R=\mathbb{R}^{+}C_{\mu}$, where $\mu$ is a wall in $\mathbb{F}_{X}$. Keep notations as the discussions above. Then $A$ is the $G$-orbit closure on $X$ such that $\mathfrak{C}_{A}=\langle e_{1},\ldots,e_{\alpha}\rangle$, and $\mathfrak{C}_{A'}=\langle e_{1},\ldots,e_{\alpha},e_{\beta+1},\ldots,e_{r+1}\rangle$.

$(ii)$ Suppose that $R$ doesn't contain any $C_{\mu}$ for walls $\mu$ in $\mathbb{F}_{X}$ and that $R=\mathbb{R}^{+}C_{D, Y}$, where $Y$ is a closed $G$-orbit on $X$ and $D\in\mathfrak{D}(G/H)\backslash\mathfrak{D}_{Y}$. Then $\rho(\nu_{D})\in\mathfrak{C}_{Y}$, $A$ is the $G$-orbit closure on $X$ such that $\mathfrak{C}_{A}=\langle e_{1},\ldots,e_{s}\rangle$, and $\mathfrak{C}^{c}_{A'}=(\mathfrak{C}_{A}, \mathfrak{D}_{A}\cup\{D\})$, where $\mathfrak{C}_{Y}=\langle e_{1},\ldots,e_{r}\rangle$, $\rho(\nu_{D})=\sum\limits_{i=1}^{s}a_{i}e_{i}$, and $a_{i}>0$ for all $1\leq i\leq s$.
\end{prop}

\begin{proof}
By \cite[Thm. 3.2$(ii)$]{Br93}, either $R=\mathbb{Q}^{+}C_{\mu}$ for some wall $\mu$ in $\mathbb{F}_{X}$ or $R=\mathbb{Q}^{+}C_{D, Y}$ for some closed $G$-orbit $Y$ and some $D\in\mathfrak{D}(G/H)\backslash\mathfrak{D}_{Y}$. Hence, if the conclusions $(i)(ii)$ hold, then $A$ is a $G$-orbit closure. So we can turn to show the conclusions $(i)$ and $(ii)$.

\medskip

$(i)$ Denote by $\varphi: \mathbb{F}_{X}\rightarrow\mathbb{F}_{X'}$ the morphism of colored fans corresponding to $\pi$. Let $\sigma=\langle e_{1},\ldots,e_{\alpha}\rangle$, and $\tau=\langle e_{1},\ldots,e_{\alpha},e_{\beta+1},\ldots,e_{r+1}\rangle$. Denote by $Z$ the closed $G$-orbit such that $\mathfrak{C}_{Z}=\langle e_{1}, \ldots, e_{r} \rangle$.  By \cite[Prop. 4.6]{Br93}, $\mathfrak{C}_{\pi(Z)}=\langle e_{1}, \ldots, e_{r+1}\rangle$, and $\tau$ is a face of $\mathfrak{C}_{\pi(Z)}$. Thus, by Proposition \ref{all faces are colored faces, maps of orbits, C_u maps to a point}$(i)$, there exists a $G$-orbit $V(\sigma)$ on $X$ such that $\mathfrak{C}_{V(\sigma)}=\sigma$ and a $G$-orbit $V'(\tau)$ on $X'$ such that $\mathfrak{C}_{V'(\tau)}=\tau$.

Denote by $X_{-\sigma}=X\backslash\overline{V(\sigma)}$ and $X'_{-\tau}=X'\backslash\overline{V'(\tau)}$. By Theorem \ref{correspondence fans and varieties}, $\mathbb{F}_{X}\backslash\mathbb{F}_{X_{-\sigma}}=\{(\mathfrak{C}, \mathfrak{D})\mid \mathfrak{C}\supseteq\sigma\}$, and $\mathbb{F}_{X'}\backslash\mathbb{F}'_{X'_{-\tau}}=\{(\mathfrak{C}', \mathfrak{D}')\mid \mathfrak{C'}\supseteq\tau\}$. By \cite[Prop. 4.6]{Br93}, $\mathbb{F}_{X_{-\sigma}}$ is identified with $\mathbb{F}'_{X'_{-\tau}}$ by $\varphi$. By Theorem \ref{correspondence fans and varieties} and Theorem \ref{morphism of fan}, $X_{-\sigma}$ is $G$-equivariantly isomorphic to $X'_{-\tau}$ by the morphism $\pi$. In particular, $A\subseteq\overline{V(\sigma)}$.

Denote by $X_{V(\sigma)}$ (resp. $X'_{V'(\tau)}$) the simple spherical open subvariety of $X$ (resp. $X'$) with the unique closed $G$-orbit $V(\sigma)$ (resp. $V'(\tau)$). Since $X_{V(\sigma)}\backslash V(\sigma)\subseteq X_{-\sigma}$, $\pi$ maps $X_{V(\sigma)}\backslash V(\sigma)$ isomorphically to its image.

On the other hand, $\sigma\subseteq\tau$, i.e. $\varphi$ maps $\sigma$ into $\tau$. By Theorem \ref{morphism of fan}, $\pi$ maps $X_{V(\sigma)}$ into $X'_{V'(\tau)}$. In particular, $\overline{\pi(V(\sigma))}\supseteq V'(\tau)$. Denote by $Z=\pi(V(\sigma))$. Then $\mathfrak{C}^{c}_{Z}$ is a colored face of  $\mathfrak{C}^{c}_{V'(\tau)}$. The fact that $a_{i}<0$ for all $1\leq i\leq\alpha$ and $a_{i}>0$ for all $\beta+1\leq i\leq r+1$ impies $\sigma^{o}\subseteq\tau^{o}$. By Proposition \ref{all faces are colored faces, maps of orbits, C_u maps to a point}$(ii)$, $\pi(V(\sigma))=Z=V'(\tau)$. By Corollary \ref{description of M_Y and N_G(K)}$(i)$, $\text{rank}(V(\sigma))=\text{rank}(V'(\tau))+r-\beta$. In particular, $\text{rank}(V(\sigma))\neq\text{rank}(V'(\tau))$. Thus, $\pi|_{V(\sigma)}: V(\sigma)\rightarrow V'(\tau)$ is not an isomorphism. Hence, $A\cap V(\sigma)\neq\emptyset$.

Since $A$ is closed and $G$-stable, $A\supseteq\overline{G(A\cap V(\sigma))}=\overline{V(\sigma)}$. Hence, $A=\overline{V(\sigma)}$. Since $A$ is complete, $A'=\overline{\pi(V(\sigma))}=\overline{V'(\tau)}$.

\medskip

$(ii)$  Denote by $\varphi: \mathbb{F}_{X}\rightarrow\mathbb{F}_{X'}$ the morphism of colored fans corresponding to $\pi$.

Assume that $\rho(\nu_{D})\notin\mathfrak{C}_{Y}$. By \cite[Prop. 3.4]{Br93}, $(\langle\mathfrak{C}_{Y}, \rho(\nu_{D})\rangle, \mathfrak{D}_{Y}\cup\{D\})$ is a colored cone in $\mathbb{F}_{X'}$. By the fact $\mathfrak{C}_{Y}\subsetneqq \langle \mathfrak{C}_{Y}, \rho(\nu_{D}) \rangle$ and Proposition \ref{all faces are colored faces, maps of orbits, C_u maps to a point}$(iii)$, there is a wall $\mu$ in $\mathbb{F}_{X}$ such that $C_{\mu}$ is contracted by $\pi$, i.e. $C_{\mu}\in R$. It's contradicted with our assumption. Hence, $\rho(\nu_{D})\in\mathfrak{C}_{Y}$.

Assume that $\mathfrak{C}_{Y}=\langle e_{1}, \ldots, e_{r}$ and $\rho(\nu_{D})=\sum\limits_{i=1}^{s}a_{i}e_{i}$, where $a_{i}>0$ for all $1\leq i\leq s$. Denote by $\sigma=\langle e_{1}, \ldots, e_{s} \rangle$. Note that $\sigma=\langle\sigma, \rho(\nu_{D})\rangle$ is a face of $\mathfrak{C}_{Y}=\langle\mathfrak{C}_{Y}, \rho(\nu_{D})\rangle$. By \cite[Prop. 3.4]{Br93} and Theorem \ref{correspondence fans and varieties}, there is a $G$-orbit $Y'$ on $X'$ such that $\mathfrak{C}^{c}_{Y'}=(\mathfrak{C}_{Y}, \mathfrak{D}_{Y}\cup\{D\})$. By Proposition \ref{all faces are colored faces, maps of orbits, C_u maps to a point}$(ii)$, $\pi(Y)=Y'$. By Proposition \ref{all faces are colored faces, maps of orbits, C_u maps to a point}$(i)$, there is a $G$-orbit $V$ on $X$ and a $G$-orbit $V'$ on $X'$ such that $\mathfrak{C}_{V}=\mathfrak{C}_{V'}=\sigma$. Denote by $\mathbb{F}_{-\sigma}$  the subset of $\mathbb{F}_{X}$ such that $\mathbb{F}_{X}\backslash\mathbb{F}_{-\sigma}=\{(\mathfrak{C}, \mathfrak{D})| \mathfrak{C}\supseteq\sigma\}$, and $\mathbb{F}'_{-\sigma}$ the subset of $\mathbb{F}_{X'}$ such that $\mathbb{F}_{X'}\backslash\mathbb{F}'_{-\sigma}=\{(\mathfrak{C}', \mathfrak{D}')| \mathfrak{C'}\supseteq\sigma\}$. By Theorem \ref{correspondence fans and varieties}, the spherical $G/H$-embeddings $X_{-V}$ and $X'_{-V'}$ satisfy that $\mathbb{F}_{X_{-V}}=\mathbb{F}_{-\sigma}$ and $\mathbb{F}_{X'_{-V'}}=\mathbb{F}'_{-\sigma}$, where $X_{-V}=X\backslash\overline{V}$ and $X'_{-V'}=X'\backslash\overline{V'}$. By \cite[Prop. 3.4]{Br93}, $\mathbb{F}_{-\sigma}$ is identified with $\mathbb{F}'_{-\sigma}$ by $\varphi$. By Theorem \ref{correspondence fans and varieties} and Theorem \ref{morphism of fan}, $X_{-V}$ is $G$-equivariantly isomorphic to $X'_{-V'}$ by the morphism $\pi$. In particular, $A\subseteq\overline{V}$.

On the other hand, by Proposition \ref{all faces are colored faces, maps of orbits, C_u maps to a point}$(ii)$, $\pi(V)=V'$. By Theorem \ref{correspondence fans and varieties}, $\mathfrak{C}^{c}_{V}$ (resp. $\mathfrak{C}^{c}_{V'}$) is a colored face of $\mathfrak{C}^{c}_{Y}$ (resp. $\mathfrak{C}^{c}_{Y'}$). Since $\mathfrak{D}_{Y'}=\mathfrak{D}_{Y}\cup\{D\}$, $\mathfrak{D}_{V'}=\mathfrak{D}_{V}\cup\{D\}$. Then by Proposition \ref{dim. of horospherical var.}$(iii)$, $\text{dim}(V)\neq\text{dim}(V')$. In particular, $A\cap V\neq\emptyset$.

Since $A$ is closed and $G$-stable, $A\supseteq\overline{G(A\cap V)}=\overline{V}$. Hence, $A=\overline{V}$, and $\mathfrak{C}_{A}=\mathfrak{C}_{V}=\sigma$. Since $A$ is complete, $A'=\overline{\pi(V)}=\overline{V'}$, and $\mathfrak{C}^{c}_{A'}=\mathfrak{C}^{c}_{V'}=(\mathfrak{C}_{V}, \mathfrak{D}_{V}\cup\{D\})=(\mathfrak{C}_{A}, \mathfrak{D}_{A}\cup\{D\})$.
\end{proof}

\begin{thm} \label{nef2=psef2 horospherical}
Let $X$ be a smooth projective horospherical $G/H$-embedding of dimension $n\geq 3$ such that $\text{Nef}^{\, 2}(X)=\text{Psef}^{\, 2}(X)$. Then $\text{Nef}^{\, 1}(X)=\text{Psef}^{\, 1}(X)$.
\end{thm}

\begin{proof}
When $n=3$, the conclusion follows from the duality. So we can assume $n\geq 4$. Now assume that $H\supseteq R_{u}(B)$, $N_{G}(H)=P_{I}$ and there is an extremal ray $R$ of $NE(X)$ such that $R\nsubseteq\text{Nef}_{1}(X)$.

By \cite[Thm. 3.1]{Br93}, there exists a corresponding contraction $\pi=\text{cont}_{R}: X\rightarrow Y$. Let $A\subseteq X$ be the exceptional locus, and $A'=\pi(A)$. By Corollary \ref{higher codimension spherical}, $\text{dim}(A)=1$ and $\text{dim}(A')=0$.

By Proposition \ref{exceptional locus of bir. Mori cont. horospherical}, $A$ is a $G$-orbit closure. Thus, $A$ is an irreducible $G$-stable curve on $X$ and $A'$ is a $G$-stable point on $Y$. Now we will get contradictions case by case as follows.

Case 1. Assume that there exist irreducible divisors $D$ and $E$ such that

$(i)$ $D$ is $B$-stable and $D\cdot A<0$,

$(ii)$ $E$ is $G$-stable and $A\subseteq E$, and

$(iii)$ $D\neq E$.

By Proposition \ref{locally fac. Q-fac. criterion}$(ii)$, $\mathbb{Q}^{+}\rho(v_{E})$ is an extremal ray of $\mathfrak{C}_{A}$, and there is a unique proper face $\mathfrak{C}$ of $\mathfrak{C}_{A}$ of codimension one not containing $\rho(v_{E})$. By Proposition \ref{all faces are colored faces, maps of orbits, C_u maps to a point}$(i)$, there is a unique $G$-orbit closure $W$ such that $\mathfrak{C}_{W}=\mathfrak{C}$, and $A\subsetneqq W$. Note that $\mathfrak{D}_{W}\subseteq\mathfrak{D}_{A}\subseteq\mathfrak{D}_{W}\cup\{E\}$ by Proposition \ref{description of D_Y} and Proposition \ref{locally fac. Q-fac. criterion}. Since $E\in\mathcal{V}_{X}$, $\mathfrak{D}_{A}=\mathfrak{D}_{W}$.

Since $E$ and $W$ are both $G$-stable, $E\cap W$ is the union of some $G$-orbits. Take any $G$-orbit $V$ contained in $E\cap W$. By Theorem \ref{correspondence fans and varieties}, $\mathfrak{C}_{V}\supseteq\langle \mathfrak{C}_{W}, \rho(\nu_{E}) \rangle=\mathfrak{C}_{A}$. Thus, $V\subseteq A$. Hence, $E\cap W\subseteq A$, i.e. $A=E\cap W$. Therefore, $\text{dim}(W)=2$, and $E\cdot W=\lambda A$ in $N_{1}(X)_{\mathbb{R}}$ for some $\lambda>0$. Thus, $D\cdot E\cdot W<0$.

On the other hand, the fact $D\neq E$ implies that $D\cdot E\in\text{Psef}^{\, 2}(X)=\text{Nef}^{\, 2}(X)$. Hence, $D\cdot E\cdot W\geq 0$. We get a contradiction.

\medskip

Case 2. Assume that every irreducible $G$-stable divisor $E$ has property $E\cdot A\geq 0$.

By \cite[Thm. 3.3.9]{Per12}, $-K_{X}\cdot A\geq 0$. Thus, the length $l(R)= -K_{X}\cdot A$ of the extremal ray $R$ is nonnegative. By \cite[Thm. (0.4)]{Io86}, $\text{dim}(X)\leq 2\text{\,dim}(A)-l(R)+1\leq 3$. It's contradicted with our assumption on the dimension of $X$.

\medskip

Case 3. Assume that there is an irreducible $G$-stable divisor $E$ such that $E\cdot A<0$, and $E$ is the unique irreducible $B$-stable divisor having a negative degree on the curve $A$.

By \cite[Thm. 3.2$(ii)$]{Br93}, either $A=\lambda C_{\mu}$ in $N_{1}(X)_{\mathbb{R}}$ for some wall $\mu$ of $\mathbb{F}_{X}$ and $\lambda >0$, or $C=\lambda C_{D, Z}$ in $N_{1}(X)_{\mathbb{R}}$ for some closed $G$-orbit $Z$, $D\in\mathfrak{D}(G/H)\backslash\mathfrak{D}_{Z}$ and $\lambda >0$.

\medskip

Case 3.1. Assume that $A=\lambda C_{\mu}$ in $N_{1}(X)_{\mathbb{R}}$, where $\mu$ is a wall in $\mathbb{F}_{X}$ and $\lambda>0$.

Let $\mu, \mu_{+}, \mu_{-}, e_{i}, a_{i}, \alpha, \beta, r$ be as in the discussions before Proposition \ref{exceptional locus of bir. Mori cont. horospherical}. Let $D_{e}$ be the irreducible $B$-stable divisor on $X$ such that $\rho(v_{D_{e}})\in\mathbb{Q}^{+}e$, which is unique by Corollary \ref{correspondence rays and divisors horospherical Q-factorial}. By Proposition \ref{exceptional locus of bir. Mori cont. horospherical}, $\mathfrak{C}_{A}=\langle e_{1},\ldots, e_{\alpha}\rangle$, $\mathfrak{C}_{A'}=\langle e_{1},\ldots, e_{\alpha}, e_{\beta+1},\ldots, e_{r+1}\rangle$. By Proposition \ref{dim. of horospherical var.}, $1=\text{dim}(A)=r-\alpha+\text{dim}(G/P_{I\cup \mathfrak{D}_{A}})$. Since $\alpha\leq\beta\leq r-1$, we know that $\alpha=\beta=r-1$ and $S=I\cup \mathfrak{D}_{A}$. Thus, $\mathfrak{D}(G/H)=S\backslash I=\mathfrak{D}_{A}\subseteq\{D_{e_{1}},\ldots, D_{e_{\alpha}}\}$.

On the other hand, by the formula (1) on page \pageref{eqn. C_(D, Y)}, the intersection number $D_{e_{i}}\cdot C_{\mu}$ has the same sign as $a_{i}$. In particular, $D_{e_{i}}\cdot C_{\mu}<0$ for all $1\leq i\leq \alpha$. By the uniqueness of $E$ in the assumption of Case 3, $\alpha=1$ and $D_{e_{1}}=E$. Hence, $\text{rank}(G/H)=r=\alpha+1=2$ and $\mathfrak{D}(G/H)\subseteq\{E\}$. However, $E$ is $G$-stable. So $\mathfrak{D}(G/H)=\emptyset$, i.e. $P_{I}=G$. Thus, $\text{dim}(X)=\text{dim}(G/H)=\text{rank}(G/H)+\text{dim}(G/P_{I})=2$, which is contradicted with our assumption on the dimension of $X$.

\medskip

Case 3.2. Assume that there doesn't exist any wall $\mu$ in $\mathbb{F}_{X}$ such that $A=\lambda C_{\mu}$ in $N_{1}(X)_{\mathbb{R}}$ for some $\lambda>0$. By \cite[Thm. 3.2$(ii)$]{Br93}, we can assume that $C=\lambda C_{D, Z}$ in $N_{1}(X)_{\mathbb{R}}$, where $\lambda>0$, $Z$ is a closed $G$-orbit on $X$ and $D\in\mathfrak{D}(G/H)\backslash\mathfrak{D}_{Z}$. By the assumption, we can apply \cite[Prop. 3.4]{Br93} to get $\mathbb{F}_{Y}$.

Denote by $\mathfrak{C}_{Z}=\langle e_{1},\ldots, e_{r}\rangle$, and $\rho(v_{D})=\sum\limits_{i=1}^{r}a_{i}e_{i}$. Thus, by the formula (2) on page \pageref{eqn. C_(D, Y)}, $D\cdot C_{D, Z}=1$, $D_{e_{i}}\cdot C_{D, Z}=-a_{i}$ for all $1\leq i\leq r$, and for any $D'\in\mathfrak{B}(X)\backslash\{D_{e_{1}}, \ldots, D_{e_{r}}\}$, the intersection number $D'\cdot C_{D, Z}\geq 0$.

By Proposition \ref{exceptional locus of bir. Mori cont. horospherical}, $\rho(\nu_{D})\in\mathfrak{C}_{Z}$, i.e. $a_{i}\geq 0$ for all $1\leq i\leq r$. Then by the uniqueness of $E$ in the assumption of Case 3, we can assume that $E=D_{e_{1}}$, $a_{1}>0$ and $a_{i}=0$ for all $2\leq i\leq r$, i.e. $\rho(v_{D})=a_{1}e_{1}=a_{1}\rho(v_{E})$. By Proposition \ref{exceptional locus of bir. Mori cont. horospherical}$(ii)$, $E$ is the exceptional locus of $\pi$, i.e. $E=A$. Since $A$ is a curve and $E$ is a prime divisor on $X$, we know that $\text{dim}(X)=2$. It is contradicted with our assumption on the dimension of $X$. Hence, the conclusion follows.
\end{proof}

\subsubsection{Reduction morphisms on horospherical varieties} \label{subsubsection isomorphic colored fans}

In this part, we analysis the horospherical cases of Theorem \ref{morphism to G/P_0} and show its converse in the horospherical cases.

Let $X$ be a complete horospherical $G/H$-embedding such that $H\supseteq R_{u}(B^{-})$ and $N_{G}(H)=P_{I}^{-}$. By Remark \ref{D=S-I, rho(D_a)=a*}, $\mathfrak{D}(G/H)$ is identified with $S\backslash I$. Let $\mathfrak{D}_{1}$ be a subset of $\mathfrak{D}_{0}(G/H)$. By Theorem \ref{morphism to G/P_0}$(i)$, $D_{0}$ induces a $G$-equivariant morphism $\pi_{0}: X\rightarrow G/P_{0}^{-}$, where $D_{0}=\sum\limits_{D\in\mathfrak{D}_{1}}D$, and $P_{0}^{-}$ is some parabolic subgroup of $G$ containing $B^{-}$. Let $x_{0}=H/H\in G/H$, $\tilde{x}_{0}=\pi_{0}(x_{0})$ and $X_{0}=\pi_{0}^{-1}(\tilde{x}_{0})$. Note that $Bx_{0}$ is the open $B$-orbit on $X$ and $G_{\tilde{x}_{0}}=P_{0}^{-}$. Then by Lemma \ref{morphism to G/P_0 then F_X <= F_(X_0)}$(a)(e)$, $X_{0}$ is a complete spherical $L_{0}/H_{0}$-embedding and there is an immersion $\Phi: \mathbb{F}_{X}\rightarrow \mathbb{F}_{X_{0}}$, where $L_{0}=P_{0}\cap P_{0}^{-}$ and $H_{0}=H\cap L_{0}$.

\begin{thm} \label{morphism to G/P_0 horospherical}
Keep notations as in the discussions above. Then the following hold.

$(i)$ $P_{0}=P_{S\backslash \mathfrak{D}_{1}}$, $H_{0}\supseteq R_{u}(B_{0}^{-})$ and $X_{0}$ is a complete horospherical $L_{0}/H_{0}$-embedding.

$(ii)$ For any $\alpha\in I\cup \mathfrak{D}_{1}$, $\langle M_{L_{0}/H_{0}}, \alpha^{\vee} \rangle_{G}=0$.

$(iii)$ $R_{u}(P_{S\backslash \mathfrak{D}_{1}}^{-})$ acts trivially on $X_{0}$.

$(iv)$ The map $\Phi: \mathbb{F}_{X}\rightarrow\mathbb{F}_{X_{0}}$ is an isomorphism of colored fans.
\end{thm}

\begin{proof}
$(i)$ Consider the restricted morphism $\pi_{0}|_{G/H}: G/H\rightarrow G/P_{0}^{-}$. By Lemma \ref{parabolic containing R_u(B) is unique}, this morphism is induced by the inclusion $H\subseteq P_{0}^{-}$. By Lemma \ref{morphism to G/P_0 then F_X <= F_(X_0)}$(d)$, $\pi^{\sharp}: \mathfrak{D}(G/P_{0}^{-})\rightarrow\mathfrak{D}_{1}, D'\mapsto\pi_{0}^{-1}(D')$ is a bijective map. Apply Remark \ref{D=S-I, rho(D_a)=a*}$(iii)$ to $G/H$ and $G/P_{0}^{-}$ respectively, we get that $P_{0}=P_{S\backslash \mathfrak{D}_{1}}$. As a closed subscheme of $X$, $X_{0}$ is complete. Note that $H_{0}=H\cap L_{0}\supseteq R_{u}(B^{-})$. Hence, the spherical $L_{0}/H_{0}$-embedding is $L_{0}$-horospherical.

\medskip

$(ii)$ By Lemma \ref{morphism to G/P_0 then F_X <= F_(X_0)}$(c)$ and Remark \ref{D=S-I, rho(D_a)=a*}$(ii)$, $M_{L_{0}/H_{0}}=M_{G/H}\subseteq\chi(P_{I}^{-})$. Thus, for all $\alpha\in I$, $\langle M_{L_{0}/H_{0}}, \alpha^{\vee} \rangle_{G}=0$.

By Remark \ref{D=S-I, rho(D_a)=a*}$(iii)(iv)$ and the fact $\mathfrak{D}_{1}\subseteq \mathfrak{D}_{0}(G/H)$, $\beta^{\vee}|_{M_{G/H}}=\rho(\nu_{D_{\beta}})=0$ for all $\beta\in \mathfrak{D}_{1}$. Hence, $\langle M_{L_{0}/H_{0}}, \beta^{\vee} \rangle_{G}=0$.

\medskip

$(iii)$ Note that $(B\cdot x_{0})\cap X_{0}=\{b\cdot x_{0}\mid b\in B, \pi(b\cdot x_{0})=\tilde{x}_{0}\}=\{b\cdot x_{0}\mid b\in G_{\tilde{x}_{0}}\cap B\}=B_{0}\cdot x_{0}$. Thus, $P_{0}^{-}\cdot x_{0}\supseteq B_{0}\cdot x_{0}$ is an open subset of $X_{0}$.

Note that $(P_{0}^{-})_{x_{0}}=H\cap P_{0}^{-}=H\supseteq R_{u}(B^{-})\supseteq R_{u}(P_{0}^{-})$. Take any point $x\in P_{0}^{-}\cdot x_{0}$. There exists some element $p\in P_{0}^{-}$ such that $p\cdot x_{0}=x$. Thus, $(P_{0}^{-})_{x}=p(P_{0}^{-})_{x_{0}}p^{-1}\supseteq pR_{u}(P_{0}^{-})p^{-1}=R_{u}(P_{0}^{-})$. Hence, $R_{u}(P_{0}^{-})\subseteq\bigcap\limits_{x\in P_{0}^{-}\cdot x_{0}}(P_{0}^{-})_{x}$, i.e. $R_{u}(P_{0}^{-})$ act trivially on the open subset $P_{0}^{-}\cdot x_{0}$ of the irreducible variety $X_{0}$. Therefore, $R_{u}(P_{0}^{-})$ act trivially on $X_{0}$.

\medskip

$(iv)$ By Lemma \ref{morphism to G/P_0 then F_X <= F_(X_0)}$(e)$, $\mathbb{F}_{X}\subseteq\mathbb{F}_{X_{0}}$ and $\mathbb{F}_{X_{0}}\backslash\mathbb{F}_{X}=\{(\mathfrak{C}, \mathfrak{D})\in\mathbb{F}_{X_{0}}\mid \mathfrak{C}^{o}\cap\mathcal{V}(G/H)=\emptyset\}$. However, by \cite[Cor. 6.2]{Kn91}, $\mathcal{V}(G/H)=(N_{X})_{\mathbb{Q}}$. Hence, $\mathbb{F}_{X}=\mathbb{F}_{X_{0}}$.
\end{proof}

\begin{thm} \label{converse of morphism to G/P_0 horospherical}
Let $\mathfrak{D}_{1}\subseteq S$ be a subset. Denote by $P_{0}=P_{S\backslash \mathfrak{D}_{1}}$, $L_{0}=P_{0}\cap P_{0}^{-}$ the standard Levi factor and $B_{0}=B\cap L_{0}$. Suppose that the following three conditions hold.

$(a)$ There is a $P_{0}^{-}$-action on a normal variety $X_{0}$ such that $R_{u}(P_{0}^{-})$ acts trivially on $X_{0}$, and $X_{0}$ is a horospherical $L_{0}/H_{0}$-embedding, where $H_{0}\supseteq R_{u}(B_{0}^{-})$.

$(b)$ For any $\alpha\in \mathfrak{D}_{1}$, $\langle M_{L_{0}/H_{0}}, \alpha^{\vee} \rangle_{G}=0$.

$(c)$ $N_{L_{0}}(H_{0})=P_{I, L_{0}}^{-}$ is the parabolic subgroup of $L_{0}$ containing $B_{0}$ corresponding to the subset $I$ of the set $S\backslash \mathfrak{D}_{1}$ of simple roots of $L_{0}$. And for any $\alpha\in I$, $\langle M_{L_{0}/H_{0}}, \alpha^{\vee} \rangle_{G}=0$.

Let $H=\text{Ker}_{P_{I}^{-}}M_{L_{0}/H_{0}}$, where $P_{I}$ is the parabolic subgroup of $G$ containing $B$ corresponding to the subset $I$ of the set $S$ of simple roots of $G$. Then the following hold.

$(i)$ The variety $X=G\times^{P_{0}^{-}}X_{0}$ is a horospherical $G/H$-embedding, and $\mathbb{F}_{X}=\mathbb{F}_{X_{0}}$.

$(ii)$ There is an inclusion $\mathfrak{D}_{1}\subseteq \mathfrak{D}_{0}(G/H)$.

$(iii)$ Assume that $X_{0}$ is complete. Let $\pi_{0}: X\rightarrow\widetilde{X}$ be the morphism corresponding to $\mathfrak{D}_{1}$ as in Theorem \ref{morphism to G/P_0}$(i)$. Then $\widetilde{X}$ is $G$-equivariantly isomorphic to $G/P_{0}^{-}$ and $\pi_{0}^{-1}\pi_{0}(x_{0})$ is $P_{0}^{-}$-equivariantly isomorphic to $X_{0}$.
\end{thm}

Note that $M_{L_{0}/H_{0}}\subseteq\chi(B_{0})=\chi(B)$. Hence, the pairing $\langle -, - \rangle_{G}$ in the conditions $(a)$ and $(b)$ are well-defined. By Remark \ref{D=S-I, rho(D_a)=a*}$(i)$, $N_{L_{0}}(H_{0})$ is a parabolic subgroup of $L_{0}$ containing $B_{0}^{-}$. Since $L_{0}=P_{S\backslash \mathfrak{D}_{1}}\cap P_{S\backslash \mathfrak{D}_{1}}^{-}$, the set of the simple roots of $L_{0}$ corresponding to $(B_{0}, T)$ is $S\backslash \mathfrak{D}_{1}$. Hence, the condition $(c)$ is well-defined. Remark that the conditions $(a)(b)(c)$ are corresponding to the conclusions $(i)(ii)(iii)$ in Theorem \ref{morphism to G/P_0 horospherical}. So Theorem \ref{converse of morphism to G/P_0 horospherical} is indeed the converse of Theorem \ref{morphism to G/P_0} in the horospherical cases.

\begin{proof}
The condition $(c)$ implies that $M_{L_{0}/H_{0}}\subseteq\chi(P_{I})$. So $H=\text{Ker}_{P_{I}}M_{L_{0}/H_{0}}$ is well-defined, and $H\supseteq R_{u}(B)$. In particular, $G/H$ is a homogeneous horospherical variety. We will show this theorem in two steps.

\medskip

Step 1. We will construct a horospherical $G/H$-embedding $X'$ with the colored fan isomorphic to $\mathbb{F}_{X_{0}}$.

Note that $P_{I}^{-}\cap L_{0}$ is a parabolic subgroup of $L_{0}$ containing $B_{0}^{-}$. Then by the one to one correspondence between parabolic subgroups of $L_{0}$ containing $B_{0}^{-}$ and subsets of $S\backslash \mathfrak{D}_{1}$, $P_{I, L_{0}}^{-}=P_{I}^{-}\cap L_{0}$. Hence, $H\cap L_{0}=\text{Ker}_{P_{I}^{-}\cap L_{0}}M_{L_{0}/H_{0}}=\text{Ker}_{P_{I, L_{0}}^{-}}M_{L_{0}/H_{0}}=H_{0}$, where the last equality follows from Remark \ref{D=S-I, rho(D_a)=a*}$(ii)$.

Note that $M_{L_{0}/H_{0}}\subseteq\chi(P_{I}^{-}/H)\subseteq\chi(P_{I, L_{0}}^{-}/H_{0})$. By Remark \ref{D=S-I, rho(D_a)=a*}$(ii)$, $M_{G/H}=\chi(P_{I}^{-}/H)$ and $M_{L_{0}/H_{0}}=\chi(P_{I, L_{0}}^{-}/H_{0})$. Hence, $M_{G/H}=M_{L_{0}/H_{0}}$, and $N_{G/H}=N_{L_{0}/H_{0}}$. Remark that $(ii)$ of this theorem follows from the condition $(b)$ and the identity $M_{G/H}=M_{L_{0}/H_{0}}$.

By considering the natural morphisms $L_{0}\rightarrow P_{0}^{-}\rightarrow P_{0}^{-}/H$, we get an inclusion $L_{0}/H_{0}\subseteq P_{0}^{-}/H$. On the other hand, $P_{0}^{-}=R_{u}(P_{0}^{-})L_{0}$ and $R_{u}(P_{0}^{-})\subseteq H$. Hence, $L_{0}/H_{0}=P_{0}/H$, and $G/H\cong G\times^{P_{0}^{-}}P_{0}^{-}/H\cong G\times^{P_{0}^{-}}L_{0}/H_{0}$. This implies that there is an injective map $\Psi: \mathfrak{D}(L_{0}/H_{0}) \rightarrow \mathfrak{D}(G/H), D\mapsto\overline{B\times D}$. From the inclusion $L_{0}/H_{0}=P_{0}^{-}/H\subseteq G/H$, we can get a natural morphism $\phi: \mathbb{C}(G/H)^{(B)}\rightarrow \mathbb{C}(L_{0}/H_{0})^{(B_{0})}, f\mapsto f|_{L_{0}/H_{0}}$. Since $M_{G/H}=M_{L_{0}/H_{0}}$, $\phi$ is an isomorphism. Moreover, $\rho_{L_{0}/H_{0}}(\nu_{D})=\rho_{G/H}(\nu_{\Psi(D)})$ for all $D\in\mathfrak{D}(L_{0}/H_{0})$. Thus, $\mathfrak{D}(L_{0}/H_{0})$ can be identified with a subset of $\mathfrak{D}(G/H)$. By \cite[Cor. 6.2]{Kn91}, $\mathcal{V}(G/H)=(N_{G/H})_{\mathbb{Q}}$ and $\mathcal{V}(L_{0}/H_{0})=(N_{L_{0}/H_{0}})_{\mathbb{Q}}$. Therefore, $\mathbb{F}_{X_{0}}$ is a colored fan in $(N_{G/H})_{\mathbb{Q}}=(N_{L_{0}/H_{0}})_{\mathbb{Q}}$. By Theorem \ref{correspondence fans and varieties}, there exists a horospherical $G/H$-embedding $X'$ such that $\mathbb{F}_{X'}=\mathbb{F}_{X_{0}}$.

\medskip

Step 2. We claim that $X'$ is $G$-equivariantly isomorphic to $G\times^{P_{0}^{-}}X_{0}$.

For the moment, we assume that $X_{0}$ is complete. Then by \cite[Thm. 4.2, Cor. 6.2]{Kn91}, $X'$ is also complete. By Theorem \ref{morphism to G/P_0 horospherical}$(i)(iii)$ and the condition $(a)$, the conclusion $(iii)$ holds if we replace $X$ by $X'$. By Remark \ref{X is a quotient of G*X_0 by (P_0)-}, $X'$ is $G$-equivariantly isomorphic to $G\times^{P_{0}^{-}}X_{0}$. Hence, when $X_{0}$ is complete, the conclusions $(i)$ and $(iii)$ hold.

Now we deal with the general case when $X_{0}$ may be not complete. Let $\overline{X_{0}}$ be a normal $G$-equivariant completion of $X_{0}$. Let $\overline{X'}$ be a horospherical $G/H$-embedding such that $\mathbb{F}_{\overline{X'}}=\mathbb{F}_{\overline{X_{0}}}$ as in Step 1. By the discussions above, $\overline{X'}=G\times^{P_{0}^{-}}\overline{X_{0}}$. By Theorem \ref{correspondence fans and varieties}, $X'$ is the fiber product of $G\times^{P_{0}^{-}}X_{0}$ and $\overline{X'}$ over $G\times^{P_{0}}\overline{X_{0}}$, i.e. $X'=G\times^{P_{0}}X_{0}$. Hence, $(i)$ holds.
\end{proof}

\subsubsection{Smooth projective horospherical varieties whose effective divisors are nef} \label{subsubsection horospherical codimension one}

Our main result in this part is Corollary \ref{nef1=psef1 horospherical description}, which has been described in the introduction of Subsection \ref{subsection horospherical varieties}. It's a direct consequence of Theorem \ref{morphism to G/P_0 horospherical} and Theorem \ref{Nef1=Psef1 horospherical is more or less product}. So our main job in this part is to show Theorem \ref{Nef1=Psef1 horospherical is more or less product}, who says that if $X$ is a smooth projective horospherical $G/H$-embedding such that $\mathfrak{D}_{0}(G/H)=\emptyset$ and $\text{Nef}^{\, 1}(X)=\text{Psef}^{\, 1}(X)$, then $X$ is isomorphic to the product of some smooth projective $G$-horospherical varieties of Picard number one.

Our approach to show Theorem \ref{Nef1=Psef1 horospherical is more or less product} is as follows. Firstly, by studying the toric variety associated with $X$, we show that $\mathbb{F}_{X}$ is a product $\prod\limits_{i}\mathbb{F}_{i}$, (for the precise meaning, see Definition \ref{defi. of product of colored fans}). Then we show that each $\mathbb{F}_{i}$ can be identified with the colored fan of a $G$-horospherical variety $X_{i}$. Finally, we show that $\prod\limits_{i}X_{i}$ is a horospherical $G/H$-embedding such that $\mathbb{F}_{\prod\limits_{i}X_{i}}=\mathbb{F}_{X}$, which implies that $X$ is $G$-equivariantly isomorphic to $\prod\limits_{i}X_{i}$.

Guided by the approach above, we organize this part as follows. We firstly study the properties of the toric variety associated with $X$. Then we study the conditions when the product of several $G$-horospherical varieties $X_{i}$ is itself $G$-horospherical, and study the colored fan of $\prod\limits_{i}X_{i}$. After these discussions, we turn to the proofs of Theorem \ref{Nef1=Psef1 horospherical is more or less product} and Corollary \ref{nef1=psef1 horospherical description}. In the end, we give an example to explain Corollary \ref{nef1=psef1 horospherical description} and to show that the fiberation $\pi_{0}$ in Theorem \ref{morphism to G/P_0} may be nontrivial.

\medskip

Let $X$ be a  horospherical $G/H$-embedding such that $H\supseteq R_{u}(B)$ and $N_{G}(H)=P_{I}$. We will define a toroidal horospherical variety $\mathcal{T}(X)=G\times^{P_{I}}Y$ as follows. The maximal torus $T\subseteq B$ satisfies that $P_{I}=TH$. Thus the torus $T_{H}=P_{I}/H$ is a quotient of $T$. Let $\mathbb{F}=\{\mathfrak{C}\mid (\mathfrak{C}, \mathfrak{D})\in\mathbb{F}_{X}\}$. Thus, $\mathbb{F}$ is a fan in $(N_{X})_{\mathbb{Q}}$. By Theorem \ref{correspondence fans and varieties}, there is a unique toric variety $Y$ with the fan $\mathbb{F}_{Y}=\mathbb{F}$ and the maximal torus orbit $T_{H}$. Since $P_{I}/H=T_{H}\subseteq Y$, we get that $H$ has a trivial action on $Y$, and $P_{I}$ has a natural action on $Y$. In the set theory, we define $\mathcal{T}(X)=G\times^{P_{I}}Y$ to be the quotient of $G\times Y$ by equivalent relations that $(g, py)\thicksim(gp, y)$ for all $g\in G$, $p\in P_{I}$, and $y\in Y$. Thus, $\mathcal{T}(X)$ is a toroidal horospherical $G/H$-embedding such that $N_{\mathcal{T}(X)}=N_{X}$ and $\mathbb{F}_{\mathcal{T}(X)}=\{(\mathfrak{C}, \emptyset)\mid (\mathfrak{C}, \mathfrak{D})\in\mathbb{F}_{X}\}$.

Every $G$-orbit on $\mathcal{T}(X)$ has the form $G\times^{P_{I}}Z$, where $Z$ is a $T_{H}$-orbit on $Y$. Note that for each cone $(\mathfrak{C}, \emptyset)\in\mathbb{F}_{\mathcal{T}(X)}$, there is a unique subset $I(\mathfrak{C})\subseteq S\backslash I$ such that  $(\mathfrak{C}, I(\mathfrak{C}))\in\mathbb{F}_{X}$. By Theorem \ref{morphism of fan}, there is a unique birational $G$-equivariant morphism $\mathcal{T}_{X}: \mathcal{T}(X)\rightarrow X$ extending the identity $G/H=G/H$ and inducing the morphism $\mathbb{F}_{\mathcal{T}(X)}\rightarrow\mathbb{F}_{X}, (\mathfrak{C}, \emptyset)\rightarrow (\mathfrak{C}, I(\mathfrak{C}))$.

\begin{rmk} \label{toroidal horospherical var.  has form G*(P)Y}
Let $X$ be a toroidal $G$-horospherical variety. By Theorem \ref{correspondence fans and varieties}, $\mathcal{T}(X)=X$. Hence, there is a parabolic subgroup $P$ of $G$ and a toric variety $Y$ such that $Y$ is a $P$-variety and
$X$ is $G$-equivariantly isomorphic to $G\times^{P}Y$. This can also be seen from \cite[Example 1.13(2), Thm. 1.17, Example 1.19(3))]{Pa06}.
\end{rmk}

Keep notations as in the discussions above. In fact, we can read a lot of information of $X$ from $Y$ or inversely. Remark that $Y$ is a $T$-horospherical variety which implies that theories and results on horospherical varieties can be applied to $Y$. If $\mu$ is a wall of $\mathbb{F}_{X}$, then it is also a wall of $\mathbb{F}_{Y}$. Denote by $C^{Y}_{\mu}$ the $T$-stable irreducible curve on $Y$ corresponding to $\mu$. For any ray $R=\mathbb{R}^{+}e$ of $\mathbb{F}_{X}$, denote by $D^{Y}_{R}$ or $D^{Y}_{e}$ the corresponding $T$-stable prime divisor on $Y$.

\begin{prop} \label{information on toric Y getting from horospherical X}
Let $X$ be a projective $\mathbb{Q}$-factorial horospherical $G/H$-embedding such that $H\supseteq R_{u}(B)$, $N_{G}(H)=P_{I}$, and $\mathcal{T}(X)=G\times^{P_{I}} Y$. Then the following hold.

$(i)$ $\mathcal{T}(X)$ is a projective $\mathbb{Q}$-factorial toroidal horospherical $G/H$-embedding and $Y$ is a projective $\mathbb{Q}$-factorial toric variety.

$(ii)$ The variety $X$ is locally factorial if and only if $Y$ is smooth if and only if $\mathcal{T}(X)$ is smooth.

$(iii)$ If moreover $\text{Nef}^{\, 1}(X)=\text{Psef}^{\, 1}(X)$, then $\text{Nef}^{\, 1}(Y)=\text{Psef}^{\, 1}(Y)$.
\end{prop}

\begin{proof}
$(i)$ The conclusion that $\mathcal{T}(X)$ and $Y$ are $\mathbb{Q}$-factorial follows from Proposition \ref{locally fac. Q-fac. criterion}$(ii)$.

Take any ray $R$ in $\mathbb{F}_{X}$. By Corollary \ref{correspondence rays and divisors horospherical Q-factorial}, there exists a unique prime divisor $D_{R}$ such that $\rho_{X}(\nu_{D_{R}})\in R$. Since $\rho_{Y}(\nu_{D^{Y}_{R}})$ is the primitive lattice point in $R$, there is a positive integer $d_{R}$ such that $\rho_{X}(\nu_{D_{R}})=d_{R}\rho_{Y}(\nu_{D^{Y}_{R}})$.

Let $\delta=\sum\limits_{R}n_{R}(\delta)D_{R}+\sum\limits_{D\in\mathfrak{D}(G/H)\backslash \mathfrak{D}_{X}}n_{D}(\delta)D$ be an ample Cartier divisor on $X$, where $R$ runs over all the rays in $\mathbb{F}_{X}$. Denote by $\delta^{Y}=\sum\limits_{R}\frac{n_{R}(\delta)}{d_{R}}D_{R}^{Y}$, where $R$ runs over all the rays in $\mathbb{F}_{X}$. By \cite[Prop. 3.1]{Br89}, $\delta^{Y}$ is a Cartier divisor on $Y$. By the formula $(1)$ on page \pageref{eqn. C_u}, for any wall $\mu$ in $\mathbb{F}_{X}$ and any $D\in\mathfrak{D}(G/H)\backslash\mathfrak{D}_{X}$, $D\cdot C_{\mu}=0$ and $\delta^{Y}\cdot C_{\mu}^{Y}=\delta\cdot C_{\mu}>0$. Thus, $\delta^{Y}$ is an ample Cartier divisor on $Y$, i.e. $Y$ is projective.

Let $\delta'$ be an ample Cartier divisor on $G/P_{I}$. Then by \cite[Prop. 3.1, Thm. 3.3]{Br89}, $\delta^{Y}+\pi^{*}\delta'$ is an ample Cartier divisor on $\mathcal{T}(X)$. Hence, $\mathcal{T}(X)$ is projective.

\medskip

$(ii)$ follows from Proposition \ref{locally fac. Q-fac. criterion}$(i)$ and \cite[Thm. 2.6]{Pa06}.

\medskip

$(iii)$ By the formula $(1)$ on page \pageref{eqn. C_u}, for each wall $\mu$ of $\mathbb{F}_{X}$, $C_{\mu}^{Y}\in\text{Nef}_{1}(Y)$. Thus, $\text{Nef}_{1}(Y)=\text{Psef}_{1}(Y)$ and $\text{Nef}^{\, 1}(Y)=\text{Psef}^{\, 1}(Y)$.
\end{proof}

\begin{lem} \label{D=D_X uinon D_0 now}
Let $X$ be a projective $\mathbb{Q}$-factorial horospherical $G/H$-embedding such that $\text{Nef}^{\, 1}(X)=\text{Psef}^{\, 1}(X)$. Then $\mathfrak{D}(G/H)\backslash\mathfrak{D}_{X}=\mathfrak{D}_{0}(G/H)$, i.e. for all $D\in\mathfrak{D}(G/H)\backslash\mathfrak{D}_{X}$, $\rho(\nu_{D})=0$;
\end{lem}

\begin{proof}
By Theorem \ref{correspondence fans and varieties}, for each $G$-orbit $Y$ on $X$, $\mathfrak{C}^{c}_{Y}$ is a strictly convex colored cone. Hence, for all $D\in\mathfrak{D}_{Y}$, $\rho(\nu_{D})\neq 0$. Therefore, $\mathfrak{D}_{0}(G/H)\subseteq\mathfrak{D}(G/H)\backslash\mathfrak{D}_{X}$. Now assume that the inclusion is strict, then there is a divisor $D\in\mathfrak{D}(G/H)\backslash\mathfrak{D}_{X}$ such that $\rho(\nu_{D})\neq 0$.

Note that $\text{Supp}(\mathbb{F}_{X})=\mathcal{V}(G/H)=(N_{X})_{\mathbb{Q}}$ by \cite[Thm. 4.2, Cor. 6.2]{Kn91}. Thus, there is a closed $G$-orbit $Y$ such that $\rho(\nu_{D})\in\mathfrak{C}_{Y}$. The fact $D\in\mathfrak{D}(G/H)\backslash\mathfrak{D}_{X}$ implies that $D\in\mathfrak{D}(G/H)\backslash\mathfrak{D}_{Y}$. Then there is a $B$-stable 1-cycle class $C_{D, Y}$ in $N_{1}(X)_{\mathbb{R}}$ defined as in the formula $(2)$ on page \pageref{eqn. C_(D, Y)}.

Assume $\mathfrak{C}_{Y}=\langle e_{1},\ldots,e_{r}\rangle$. Then $\rho(\nu_{D})=\sum\limits_{i=1}^{r}a_{i}e_{i}$, where all $a_{i}\geq 0$ and there is an $i_{0}$ such that $a_{i_{0}}>0$.

By the definition of $\mathfrak{C}_{Y}$, there is an irreducible $B$-stable divisor $D'$ such that $\rho(\nu_{D'})=b_{i_{0}}e_{i_{0}}$, where $b_{i_{0}}>0$. Then by the formula $(2)$ on page \pageref{eqn. C_(D, Y)}, the intersection number $D'\cdot C_{D, Y}=-\frac{a_{i_{0}}}{b_{i_{0}}}<0$. Hence, $D'\in\text{Psef}^{\, 1}(X)\backslash\text{Nef}^{\, 1}(X)$. We get a contradiction.
\end{proof}

Remark that the condition $\text{Nef}^{\, 1}(X)=\text{Psef}^{\, 1}(X)$ in Lemma \ref{D=D_X uinon D_0 now} is necessary to guarantee $\mathfrak{D}(G/H)\backslash \mathfrak{D}_{X}=\mathfrak{D}_{0}(G/H)$. In general, the inclusion $\mathfrak{D}_{0}(G/H)\subseteq\mathfrak{D}(G/H)\backslash\mathfrak{D}_{X}$ is strict. For example, we have the following Proposition \ref{example D is not the uinon of D_X and D_0}. Note that by \cite[Thm. 1.4, Thm. 1.7]{Pa09}, there exists a smooth projective horospherical $G/H$-embedding $X$ such that $\text{rank}(X)=1$ and $\mathfrak{D}_{X}\neq\emptyset$.

\begin{prop} \label{example D is not the uinon of D_X and D_0}
Let $X$ be a smooth projective horospherical $G/H$-embedding such that $\text{rank}(X)=1$ and $\mathfrak{D}_{X}\neq\emptyset$. Choose any $D\in\mathfrak{D}_{X}$, then $\rho(\nu_{D})\neq 0$. There is a closed $G$-orbit $Z$ such that $\mathfrak{C}_{Z}=\mathbb{Q}^{+}\rho(\nu_{D})\in\mathbb{F}_{X}$.  Let $\widetilde{X}=Bl_{Z}(X)$ be the blowing-up of $X$ along $Z$. Then $\widetilde{X}$ is a smooth projective horospherical $G/H$-embedding, and $\pi^{-1}_{*}D\in\mathfrak{D}(G/H)\backslash (\mathfrak{D}_{\widetilde{X}}\cup\mathfrak{D}_{0}(G/H))$, where $\pi: \widetilde{X}\rightarrow X$ is the natural morphism, and $\pi^{-1}_{*}D$ is the strict transform of $D$.
\end{prop}

\begin{proof}
By \cite[Thm. 4.2, Cor. 6.2]{Kn91}, and Proposition \ref{all faces are colored faces, maps of orbits, C_u maps to a point}$(i)$, there are exactly three $G$-orbits $G/H, Z, Y$ on $X$, the two $G$-orbits $Z$ and $Y$ are complete, and $-\mathfrak{C}_{Y}=\mathfrak{C}_{Z}=\mathbb{Q}^{+}\rho(\nu_{D})$. Hence, $Y$ and $Z$ are isomorphic to rational $G$-homogeneous spaces. In particular, $Z$ is smooth. Hence, $\widetilde{X}$ is smooth. Since $Z$ is a closed $G$-stable subvariety on $X$, there is a natural $G$-action on $\widetilde{X}$ such that $\pi$ is $G$-equivariant. Since $\pi^{-1}(G/H)=G/H$ is open in $\widetilde{X}$, $\widetilde{X}$ is a complete smooth horospherical $G/H$-embedding. Since $X$ is projective and $\pi$ is a projective morphism, $\widetilde{X}$ is also projective.

Since $\pi$ is $G$-equivariant, $\pi^{-1}_{*}D$ is a prime $B$-stable divisor and it is not $G$-stable. Thus, $\pi^{-1}_{*}D\in\mathfrak{D}(G/H)$. Since $\pi$ is birational and $\pi$ maps $(\pi^{-1}_{*}D)\cap (G/H)$ isomorphically to $D\cap(G/H)$, we know that $\mathbb{C}(X)^{(B)}=\mathbb{C}(\widetilde{X})^{(B)}$ and $\rho_{\widetilde{X}}(\nu_{\pi^{-1}_{*}D})=\rho_{X}(\nu_{D})\neq 0$, i.e. $\pi^{-1}_{*}D\notin\mathfrak{D}_{0}(G/H)$.

Moreover, since $\widetilde{X}$ is the blowing-up of $X$ along $Z$, there are exactly three $G$-orbits $G/H, \pi^{-1}(Z)$ and $\pi^{-1}(Y)$ on $\widetilde{X}$, where $\pi^{-1}(Z)$ and $\pi^{-1}(Y)$ are the inverse images of $Z$ and $Y$ respectively. By considering the morphism $\pi$, we know that $\pi^{-1}_{*}D$ contains neither $\pi^{-1}(Z)$ nor $\pi^{-1}(Y)$. Hence, $\pi^{-1}_{*}D\notin\mathfrak{D}_{\widetilde{X}}$.
\end{proof}

\begin{prop} \label{when will a product of horo. var. be horo.}
Assume that for each $1\leq i\leq m$, $X_{i}$ is a horospherical $G/H_{i}$-embedding such that $H_{i}\supseteq R_{u}(B)$ and $N_{G}(H_{i})=P_{I_{i}}$. Denote by $X=\prod\limits_{i=1}^{m}X_{i}$, $H=\bigcap\limits_{i=1}^{m}H_{i}$, and $N_{G}(H)=P_{I}$. Then the following are equivalent:

$(i)$ $X$ is a $G$-horospherical variety;

$(ii)$ $X$ is a horospherical $G/H$-embedding;

$(iii)$ the natural injective morphism $\pi: G/H\rightarrow \prod\limits_{i=1}^{m}G/H_{i}$ is an open immersion;

$(iv)$ $\pi$ is an isomorphism;

$(v)$ $M_{G/H}=\bigoplus\limits_{i=1}^{m} M_{G/H_{i}}$, and if $i\neq j$, then for any root $\alpha\in S\backslash I_{i}$ and any root $\beta\in S\backslash I_{j}$, $\alpha$ and $\beta$ lie in different connected components of the Dynkin diagram of $G$.
\end{prop}

Before proving this proposition, we give a lemma to explain the conditions in it.

\begin{lem} \label{explanations of when will a product of horo. var. be horo.}
Keep notations as in Proposition \ref{when will a product of horo. var. be horo.}. Then the following hold.

$(i)$ $I=\bigcap\limits_{i=1}^{m}I_{i}$ and $P_{I}=\bigcap\limits_{i=1}^{m}G/P_{I_{i}}$.

$(ii)$ Let $\psi: G/P_{I}\rightarrow\prod\limits_{i=1}^{m}G/P_{I_{i}}$ be the natural morphism. Then $\psi$ is an isomorphism if and only if for any $i\neq j$, any root $\alpha\in S\backslash I_{i}$ and any root $\beta\in S\backslash I_{j}$, $\alpha$ and $\beta$ lie in different connected components of the Dynkin diagram of $G$
\end{lem}

\begin{proof}
$(i)$ Denote by $I'=\bigcap\limits_{i=1}^{m}I_{i}$. Then $P_{I'}=\bigcap\limits_{i=1}^{m}P_{I_{i}}$. Let $M'=\sum\limits_{i=1}^{m}M_{G/H_{i}}\subseteq\chi(P_{I'})$, and $H'=\text{Ker}_{P_{I'}}M'$. By Remark \ref{D=S-I, rho(D_a)=a*}$(i)(ii)$ $P_{I'}=\bigcap\limits_{i=1}^{m}TH_{i}\supseteq TH=P_{I}$, and $M_{G/H_{i}}=\chi(P_{I_{i}}/H_{i})\subseteq\chi(P_{I}/H)=M_{G/H}$. Hence, $M'\subseteq M_{G/H}$, $H'=\text{Ker}_{P_{I'}}M'\supseteq \text{Ker}_{P_{I}}M=H$ and $H'=\text{Ker}_{P_{I'}}M'\subseteq \text{Ker}_{P_{I_{i}}}M_{G/H_{i}}=H_{i}$. The later inclusion implies that $H'\subseteq\bigcap\limits_{i=1}^{m}H_{i}=H$. Thus, $H'=H$, and $P_{I'}=TH'=TH=P_{I}$, i.e. $I=I'=\bigcap\limits_{i=1}^{m}I_{i}$ and $P_{I}=\bigcap\limits_{i=1}^{m}G/P_{I_{i}}$.

\medskip

$(ii)$ By $(i)$, $\psi$ is always injective. Since $G/P_{I}$ is complete, $\psi$ is an isomorphism if and only if the natural morphism $\Psi: G\rightarrow\prod\limits_{i=1}^{m}G/P_{I_{i}}$ is dominant. Let $\mathcal{R}^{-}$ be the set of negative roots of $G$, and $\mathcal{R}^{-}_{I}$ be the set of negative roots generated by $-\alpha$ where $\alpha\in I$. Note that the tangent space $\mathcal{T}_{i}$ of $G/P_{I_{i}}$ at the origin $P_{I_{i}}/P_{I_{i}}$ is naturally isomorphic to $\bigoplus\limits_{\alpha\in \mathcal{R}^{-}\backslash \mathcal{R}_{I_{i}}^{-}}\mathfrak{g}_{\alpha}$, and the tangent space $\mathcal{T}_{G}$ of $G$ at the origin is naturally isomorphic to its Lie algebra $\mathfrak{g}$. Hence, $\Psi$ is dominant if and only if the map between the tangent spaces at the origins $\mathcal{T}_{\Psi}: \mathcal{T}_{G}\rightarrow\bigoplus\limits_{i=1}^{m}\mathcal{T}_{i}$ is surjective. The later holds if and only if for any $i\neq j$, $(\mathcal{R}^{-}\backslash \mathcal{R}_{I_{i}}^{-})\cap(\mathcal{R}^{-}\backslash \mathcal{R}_{I_{j}}^{-})=\emptyset$, i.e. $\mathcal{R}^{-}=\mathcal{R}_{I_{i}}^{-}\cup \mathcal{R}_{I_{j}}^{-}$. By checking the Dynkin diagrams of simple groups case by case, we know that this is equivalent to the fact that for any $i\neq j$, any root $\alpha\in S\backslash I_{i}$ and any root $\beta\in S\backslash I_{j}$, $\alpha$ and $\beta$ lie in different connected components of the Dynkin diagram of $G$.
\end{proof}

\begin{proof}[Proof of Proposition \ref{when will a product of horo. var. be horo.}]
$(ii)\Rightarrow (iii)$ By the definition of spherical varieties, $G/H$ is the unique open $G$-orbit on $X$. On the other hand, $\prod\limits_{i=1}^{m}G/H_{i}$ is a $G$-stable open subset on $X$. Hence, $\pi$ is an open immersion.

\medskip

$(iii)\Rightarrow (iv)$ Consider the following commutative diagram $(*)$:
\begin{eqnarray*}
\xymatrix{
G/H\ar[r]^-{\pi}\ar[d]_-{\phi_{1}}&\prod\limits_{i=1}^{m}G/H_{i}\ar[d]^-{\phi_{2}}\\
G/P_{I}\ar[r]^-{\psi}&\prod\limits_{i=1}^{m}G/P_{I_{i}}.
}
\end{eqnarray*}

Since $\pi$ is dominant, $\phi_{2}$ is surjective and $\psi\circ\phi_{1}=\phi_{2}\circ\pi$, we know that $\psi$ is dominant, which implies that $\psi$ is surjective, since $G/P_{I}$ is complete. On the other hand, $\psi$ is injective, since $P_{I}=\bigcap\limits_{i=1}^{m}P_{I_{i}}$. Hence, $\psi$ is an isomorphism. We identify $G/P_{I}$ with $\prod\limits_{i=1}^{m}G/P_{I_{i}}$ in the following.

Now we consider the fibers of $\phi_{1}$ and $\phi_{2}$. Denote by $x=P_{I}/P_{I}\in G/P_{I}=\prod\limits_{i=1}^{m}G/P_{I_{i}}$. Since $\pi$ is an open immersion, the morphism $\phi_{1}^{-1}(x)\xrightarrow{\pi}\phi_{2}^{-1}(x)$ is an open immersion. Thus, the torus $P_{I}/H$ is a subgroup of the torus $\prod\limits_{i=1}^{m}P_{I_{i}}/H_{i}$ with finite index, which implies that $P_{I}/H=\prod\limits_{i=1}^{m}P_{I_{i}}/H_{i}$. Thus, for any point $x'\in G/P_{I}$, $\phi_{1}^{-1}(x')\xrightarrow{\pi}\phi_{2}^{-1}(x')$ is an isomorphism. Hence, $\pi$ is an isomorphism.

\medskip

$(iv)\Rightarrow (ii)$ Since $B^{-}H$ is an open subset of $G$, $\prod\limits_{i=1}^{m}G/H_{i}$ has an open $B^{-}$-orbit. By Proposition \ref{equi. defi.}$(i)(iii)$, $X$ is a horospherical $G/H$-emedding.

\medskip

$(ii)\Rightarrow (i)$ is trivial.

\medskip

$(i)\Rightarrow (iv)$ Let $\hat{x}\in X$ be a point such that $G_{\hat{x}}\supseteq R_{u}(B)$. Denote by $p_{i}: X\rightarrow X_{i}$ the $i$-th projection, $\hat{H}=G_{\hat{x}}$, $\hat{x}_{i}=p_{i}(\hat{x})$, and $\hat{H}_{i}=G_{\hat{x}_{i}}$. Since $G\cdot\hat{x}$ is an open $G$-orbit on $X$, $p_{i}(G\cdot\hat{x})=G/H_{i}$. In particular, $\hat{x}_{i}\in G/H_{i}$. Now we consider the composition $G\cdot\hat{x}\xrightarrow{p_{i}} G/H_{i}\xrightarrow{\varphi_{i}} G/P_{I_{i}}$. By Lemma \ref{parabolic containing R_u(B) is unique}, $\hat{x}_{i}\in\varphi_{i}^{-1}(P_{I_{i}}/P_{I_{i}})$. Thus, $\hat{H}_{i}\subseteq P_{I_{i}}$ and there is an element $g_{i}\in P_{i}$ such that $g_{i}\cdot x_{i}=\hat{x}_{i}$, where $x_{i}=H_{i}/H_{i}\in G/H_{i}$. Hence, $\hat{H}_{i}=g_{i}H_{i}g_{i}^{-1}$ and $N_{G}(\hat{H}_{i})=g_{i}N_{G}(H)g_{i}^{-1}=P_{I_{i}}$.

Denote by $P_{\hat{I}}=N_{G}(\hat{H})$. Note that $\hat{x}=(\hat{x}_{1}, \ldots, \hat{x}_{m})$ in $X=\prod\limits_{i=1}^{m}X_{i}$. Thus, $\hat{H}=\bigcap\limits_{i=1}^{m}\hat{H}_{i}$. Apply the equivalence of the conclusions $(ii)$ and $(iv)$ in this theorem to $\hat{H}$ and $\hat{H}_{i}$ instead of $H$ and $H_{i}$, then we get that $\hat{\pi}: G/\hat{H}=G\cdot\hat{x}\rightarrow \prod\limits_{i=1}^{m}G/\hat{H}_{i}$ is an isomorphism. Note that $G/\hat{H}_{i}=G\cdot\hat{x}_{i}=G/H_{i}$, and $G/H\subseteq\prod\limits_{i=1}^{m}G/H_{i}$. Thus, $G/H\subseteq G\cdot\hat{x}$. Since $G\cdot\hat{x}$ is $G$-homogeneous, $G/H=G\cdot\hat{x}$. Hence, the natural morphism $\pi$ is an isomorphism.

\medskip

$(iv)\Rightarrow(v)$ Recall the first part of the proof of $(iii)\Rightarrow(iv)$, we know that $\psi$ is an isomorphism. By Lemma \ref{explanations of when will a product of horo. var. be horo.}$(ii)$, the second assertion of $(v)$ holds. Since $\pi$ is an isomorphism, the corresponding fibers $P_{I}/H$ and $\prod\limits_{i=1}^{m}P_{I_{i}}/H_{i}$ are isomorphic to each other. By Remark \ref{D=S-I, rho(D_a)=a*}$(ii)$, $M_{G/H}=\chi(P_{I}/H)\cong\bigoplus\limits_{i=1}^{m}\chi(P_{I_{i}}/H_{i})=\bigoplus\limits_{i=1}^{m} M_{G/H_{i}}$.

\medskip

$(v)\Rightarrow (iv)$ Consider the commutative diagram $(*)$ in the proof of $(iii)\Rightarrow (iv)$. By Lemma \ref{explanations of when will a product of horo. var. be horo.}$(ii)$, $\psi: G/P_{I}\rightarrow\prod\limits_{i=1}^{m}G/P_{I_{i}}$ is an isomorphism. We identify them and still denote by $x=P_{I}/P_{I}\in G/P_{I}$. Then $\phi_{1}^{-1}(x)=P_{I}/H$ and $\phi_{2}^{-1}(x)=\prod\limits_{i=1}^{m}P_{I_{i}}/H_{i}$ are two tori. Since $\chi(P_{I}/H)=M_{G/H}=\bigoplus\limits_{i=1}^{m}M_{G/H_{i}}=\bigoplus\limits_{i=1}^{m}\chi(P_{I_{i}}/H_{i})$, we get that $P_{I}/H\cong\prod\limits_{i=1}^{m}P_{I_{i}}/H_{i}$. Thus, for any point $x'\in G/P_{I}$, $\phi_{1}^{-1}(x')\xrightarrow{\pi}\phi_{2}^{-1}(x')$ is an isomorphism. Hence, $\pi$ is an isomorphism.
\end{proof}

\begin{defi} \label{defi. of product of colored fans}
Let $G/H$ be a homogeneous spherical variety, and $\mathbb{F}$ be a colored fan in $(N_{G/H})_{\mathbb{Q}}$. We say that $\mathbb{F}$ is a product of $\mathbb{F}_{1},\ldots,\mathbb{F}_{m}$ and denote by $\mathbb{F}_{X}=\prod\limits_{i=1}^{m}\mathbb{F}_{i}$, if there is a decomposition $N_{G/H}=\bigoplus\limits_{i=1}^{m}N_{i}$ of abelian groups such that

$(i)$ $\mathbb{F}_{i}=\{(\mathfrak{C}, \mathfrak{D})\in\mathbb{F}\mid \mathfrak{C}\subseteq (N_{i})_{\mathbb{Q}}\}$;

$(ii)$ for any $\mathfrak{C}^{c}_{i}=(\mathfrak{C_{i}}, \mathfrak{D}_{i})\in\mathbb{F}_{i}$, $\prod\limits_{i=1}^{m}\mathfrak{C}^{c}_{i}\in\mathbb{F}$, where $\prod\limits_{i=1}^{m}\mathfrak{C}^{c}_{i}=(\prod\limits_{i}\mathfrak{C}_{i}, \bigcup\limits_{i}\mathfrak{D}_{i})$, and every colored cone in $\mathbb{F}$ is of this form.
\end{defi}

\begin{prop} \label{colored fan of product is product of colored fans}
Assume that for each $1\leq i\leq m$, $X_{i}$ is a horospherical $G/H_{i}$-embedding such that $H_{i}\supseteq R_{u}(B)$. If $X=\prod\limits_{i=1}^{m}X_{i}$ is $G$-horospherical, then $\mathbb{F}_{X}=\prod\limits_{i=1}^{m}\mathbb{F}_{X_{i}}$.
\end{prop}

\begin{proof}
Let $H=\bigcap\limits_{i=1}^{m}H_{i}$. Denote by $P_{I}=N_{G}(H)$ and $P_{I_{i}}=N_{G}(H_{i})$. By Proposition \ref{when will a product of horo. var. be horo.}$(i)(iv)(v)$, $M_{G/H}=\bigoplus\limits_{i=1}^{m}M_{G/H_{i}}$, $N_{G/H}=\bigoplus\limits_{i=1}^{m}N_{G/H_{i}}$ and $G/H=\prod\limits_{i=1}^{m}G/H_{i}$. Thus, for any $D\in\mathfrak{D}(G/H_{i})$, $\pi_{i}^{-1}(D)\in\mathfrak{D}(G/H)$ and $\rho_{G/H}(\nu_{\pi_{i}^{-1}(D)})=\rho_{G/H_{i}}(\nu_{D})\in (N_{G/H_{i}})_{\mathbb{Q}}\subseteq(N_{G/H})_{\mathbb{Q}}$, where $\pi_{i}: X\rightarrow X_{i}$ is the $i$-th projection. Note that by \cite[Cor. 6.2]{Kn91}, $\mathcal{V}(G/H)=(N_{G/H})_{\mathbb{Q}}$. So $\prod\limits_{i=1}^{m}\mathbb{F}_{X_{i}}$ is a colored fan in $(N_{G/H})_{\mathbb{Q}}$.

Let $X'$ be a horospherical $G/H$-embedding such that $\mathbb{F}_{X'}=\prod\limits_{i=1}^{m}\mathbb{F}_{X_{i}}$. By Theorem \ref{morphism of fan}, there is a $G$-equivariant morphism $\Phi_{j}: X'\rightarrow X_{j}$ extending the natural morphism $G/H\rightarrow G/H_{i}$ and inducing the morphism of colored fans $\phi_{j}: \mathbb{F}_{X'}=\prod\limits_{i=1}^{m}\mathbb{F}_{X_{i}}\rightarrow\mathbb{F}_{X_{j}}$. Define $\Phi: X'\rightarrow X=\prod\limits_{i=1}^{m}X_{i}, x\mapsto (\Phi_{1}(x),\ldots,\Phi_{m}(x))$. Denote by $\phi=\prod\limits_{i=1}^{m}\phi_{i}: \mathbb{F}_{X'}=\prod\limits_{i=1}^{m}\mathbb{F}_{X_{i}}\rightarrow \mathbb{F}_{X}$ the morphism of colored fans corresponding to $\Phi$.
\medskip

Step 2. For the moment, we assume that all $X_{i}$ are complete. By \cite[Thm. 4.2, Cor. 6.2]{Kn91}, $X'$ is also complete, which implies that the birational morphism $\Phi$ is surjective. By Theorem \ref{morphism of fan}, $\pi_{i}$ corresponds to a morphism of colored fans $\psi_{i}: \mathbb{F}_{X}\rightarrow\mathbb{F}_{X_{i}}$. Consider the morphism of colored fans $\psi=\prod\limits_{i=1}^{m}\psi_{i}: \mathbb{F}_{X}\rightarrow\prod\limits_{i=1}^{m}\mathbb{F}_{X_{i}}=\mathbb{F}_{X'}, \mathfrak{C}^{c}_{Y}\mapsto\prod\limits_{i=1}^{m}\psi_{i}(\mathfrak{C}^{c}_{Y})$ and the identity $G/H=\prod\limits_{i=1}^{m}G/H_{i}$. By Theorem \ref{morphism of fan}, they correspond to a $G$-equivariant morphism $\Psi: X\rightarrow X'$.

Take any $Y'\in S_{X', G}$. Let $Y=\Phi(Y')$ and $Y_{i}=\phi_{i}(Y')=\pi_{i}\circ\Phi(Y')=\pi_{i}(Y)$. By the construction of $\mathbb{F}_{X'}$, $\mathfrak{C}^{c}_{Y'}=\prod\limits_{i=1}^{m}\mathfrak{C}^{c}_{Y_{i}}$. By Theorem \ref{morphism of fan}, $\phi(\mathfrak{C}^{c}_{Y'})\subseteq\mathfrak{C}^{c}_{Y}$ and $\psi_{i}(\mathfrak{C}^{c}_{Y})\subseteq\mathfrak{C}^{c}_{Y_{i}}$. Thus, $\psi\circ\phi(\mathfrak{C}^{c}_{Y'})\subseteq\prod\limits_{i=1}^{m}\mathfrak{C}^{c}_{Y_{i}}=\mathfrak{C}^{c}_{Y'}$. Since $\phi$ and $\psi$ are compatible with the isomorphisms of vector spaces $(N_{X'})_{\mathbb{Q}}\rightarrow(N_{X})_{\mathbb{Q}}$ and $(N_{X})_{\mathbb{Q}}\rightarrow(N_{X'})_{\mathbb{Q}}$ respectively, we get that $\psi\circ\phi(\mathfrak{C}_{Y'})=\mathfrak{C}_{Y'}$. By Proposition \ref{description of D_Y}, $\rho_{X'}^{-1}(\psi\circ\phi)^{-1}(\rho_{X'}(\mathfrak{D}_{Y'}))=D_{X'}\cap\rho_{X'}^{-1}(\psi\circ\phi)^{-1}(\mathfrak{C}_{Y'}) =D_{X'}\cap\rho_{X'}^{-1}(\mathfrak{C}_{Y'})=\mathfrak{D}_{Y'}$. Thus, $\psi\circ\phi(\mathfrak{C}^{c}_{Y'})=\mathfrak{C}^{c}_{Y'}$. Hence, $\psi\circ\phi=id_{\mathbb{F}_{X'}}$. By Theorem \ref{morphism of fan} and Theorem \ref{correspondence fans and varieties}, $\Psi\circ\Phi=id_{X'}$. In particular, $\Phi$ is a injective morphism. Thus, the birational surjective $G$-equivariant morphism $\Phi$ is an isomorphism and for each $Y_{i}\in S_{X_{i}, G}$, $\prod\limits_{i=1}^{m}Y_{i}$ is itself a $G$-oribt and $\mathfrak{C}^{c}_{\prod\limits_{i=1}^{m}Y_{i}}=\prod\limits_{i=1}^{m}\mathfrak{C}^{c}_{Y_{i}}$. Then the general case when $X_{i}$ may be not complete follows from this fact and Theorem \ref{correspondence fans and varieties} by taking a $G$-euqivariant normal completion for each $X_{i}$.
\end{proof}

\begin{rmk}
If $X=\prod\limits_{i=1}^{m}X_{i}$ is only assumed to be $G$-spherical in Proposition \ref{colored fan of product is product of colored fans}, then maybe $\mathbb{F}_{X}\neq\prod\limits_{i=1}^{m}\mathbb{F}_{X_{i}}$. We can consider an example as follows. Let $X=X_{1}\times X_{2}$, where $X_{i}=G/P_{S\backslash\{\alpha_{i}\}}$. Assume for the moment that $X$ is $G$-spherical. Then $\mathbb{F}_{X}=\mathbb{F}_{X_{1}}\times\mathbb{F}_{X_{2}}$ if and only if $\text{rank}(X)=0$ if and only if $X$ is $G$-homogeneous is and only if $X$ is $G$-equivariantly isomorphic to $G/P_{S\backslash\{\alpha_{1}, \alpha_{2}\}}$. By Lemma \ref{explanations of when will a product of horo. var. be horo.}$(ii)$, the last assertion is equivalent to the fact that $\alpha_{1}$ and $\alpha_{2}$ lie in different connected component of the Dynkin diagram of $G$. In particular, if $G$ is a simple group, then $\mathbb{F}_{X}\neq\mathbb{F}_{X_{1}}\times\mathbb{F}_{X_{2}}$. However, by \cite{Li94}, there exist such an example that $G$ is simple and simply connected, and $G/P_{S\backslash\{\alpha_{1}\}}\times G/P_{S\backslash\{\alpha_{2}\}}$ is $G$-spherical, (see Table $I$ there).
\end{rmk}

\begin{prop} \label{converse of Nef1=Psef1 horospherical}
Assume that for each $1\leq i\leq m$, $X_{i}$ is a complete $\mathbb{Q}$-factorial horospherical $G/H_{i}$-embedding. If $X=\prod\limits_{i=1}^{m}X_{i}$ is  $G$-horospherical, then the following hold.

$(i)$ $\rho(X)=\sum\limits_{i=1}^{m}\rho(X_{i})$, where $\rho(X)$ (resp. $\rho(X_{i})$) is the Picard number of $X$ (resp. $X_{i}$).

$(ii)$ $\text{Nef}^{\, 1}(X)=\text{Psef}^{\, 1}(X)$ if and only if $\text{Nef}^{\, 1}(X_{i})=\text{Psef}^{\, 1}(X_{i})$ for all $1\leq i\leq m$.
\end{prop}

\begin{proof}
$(i)$ By Proposition \ref{colored fan of product is product of colored fans}, $\mathbb{F}_{X}=\prod\limits_{i=1}^{m}\mathbb{F}_{X_{i}}$, which implies that $S_{X, G}=\{\prod\limits_{i=1}^{m}Y_{i}\mid Y_{i}\in S_{X_{i}, G}\}$. Hence, there is a bijection between $\mathcal{V}_{X}$ and the disjoint union $\bigcup\limits_{i=1}^{m}\mathcal{V}_{X_{i}}$. By Remark \ref{D=S-I, rho(D_a)=a*}$(i)(ii)$, we can assume that $H_{i}\supseteq R_{u}(B)$ and $N_{G}(H_{i})=P_{I_{i}}$. Denote by $H=\bigcap\limits_{i=1}^{m}H_{i}$ and $P_{I}=N_{G}(H)$. By Lemma \ref{explanations of when will a product of horo. var. be horo.}$(i)$, $S\backslash I=\bigcup\limits_{i=1}^{m}S\backslash I_{i}$. By Proposition \ref{when will a product of horo. var. be horo.}$(i)(v)$, $M_{G/H}=\bigoplus\limits_{i=1}^{m}M_{G/H_{i}}$, and if $i\neq j$, then $(S\backslash I_{i})\cap(S\backslash I_{j})=\emptyset$. Hence, by Proposition \ref{cycles are rat. equiv. to stable ones, picard group on spherical varieties}$(ii)$ and Remark \ref{D=S-I, rho(D_a)=a*}$(iii)$, $\rho(X)=\sum\limits_{i=1}^{m}\rho(X_{i})$.

\medskip

$(ii)$ follows from Remark \ref{nef1(X)=psef1(X) if and only if nef1(X_i)=psef1(X_i) Q-factorial}.
\end{proof}

Now we are able to show the main theorem in this part, which is a part of the characterization of the smooth projective horospherical varieties whose effective divisors are nef.

\begin{thm} \label{Nef1=Psef1 horospherical is more or less product}
Let $X$ be a smooth projective horospherical $G/H$-embedding such that $H\supseteq R_{u}(B)$, $\mathfrak{D}_{0}(G/H)=\emptyset$ and $\text{Nef}^{\, 1}(X)=\text{Psef}^{\, 1}(X)$. Then there is a $G$-equivariant isomorphism $\Phi: X\rightarrow\prod\limits_{i=1}^{\rho(X)}X_{i}$, where $\rho(X)$ is the Picard number of $X$, and each $X_{i}$ is a smooth projective horospherical $G/H_{i}$-embedding of Picard number one. Moreover, the natural morphism $\pi: G/H\rightarrow\prod\limits_{i=1}^{\rho(X)}G/H_{i}$ is an isomorphism, $\Phi$ extends $\pi$, and $H=\bigcap\limits_{i=1}^{\rho(X)}H_{i}$.
\end{thm}

\begin{proof}
Denote by $\mathcal{T}(X)=G\times^{P_{I}}Y$ the corresponding toroidal horospherical variety. Then by Proposition \ref{information on toric Y getting from horospherical X}, $Y$ is a smooth projective toric variety such that $\text{Nef}^{\, 1}(Y)=\text{Psef}^{\, 1}(Y)$. By Theorem \ref{Nef=Psef toric}, $Y\cong\mathbb{P}^{r_{1}}\times\ldots\mathbb{P}^{r_{\rho(Y)}}$, where $\rho(Y)$ is the Picard number of $Y$ and $\sum\limits_{i=1}^{\rho(Y)}r_{i}$ equals to the rank $r$ of $G/H$.

Denote by $Y_{i}=\mathbb{P}^{r_{i}}$, and $N_{i}=\text{supp}(\mathbb{F}_{Y_{i}})\cap N_{X}$. Then $N_{X}=\bigoplus\limits_{i=1}^{m}N_{i}$. For any cone $\mathfrak{C}\in\mathbb{F}_{Y}$, define $\Phi(\mathfrak{C})=\{D\in\mathfrak{D}_{X}\mid \rho(\nu_{D})\in\mathfrak{C}\}$. By Proposition \ref{description of D_Y}, if $\mathfrak{C}\in\mathbb{F}_{Y}$, then $(\mathfrak{C}, \Phi(\mathfrak{C}))\in\mathbb{F}_{X}$. Let $\mathbb{F}_{i}=\{(\mathfrak{C}, \Phi(\mathfrak{C}))\mid \mathfrak{C}\in\mathbb{F}_{Y_{i}}\}$. The fact $\mathbb{F}_{Y}=\prod\limits_{i=1}^{\rho(Y)}\mathbb{F}_{Y_{i}}$ and Proposition \ref{description of D_Y} imply that $\mathbb{F}_{X}=\prod\limits_{i=1}^{\rho(Y)}\mathbb{F}_{i}$.

By Lemma \ref{D=D_X uinon D_0 now}, $\mathfrak{D}(G/H)=\mathfrak{D}_{X}$.  Hence, by Lemma \ref{product of colored fans implies product of varieties} in the following,  there exists a smooth projective horospherical $G/H_{i}$-embedding $X_{i}$ for each $1\leq i\leq\rho(Y)$ such that $\mathbb{F}_{X_{i}}=\mathbb{F}_{i}$, $\mathfrak{D}(G/H_{i})=\mathfrak{D}_{X_{i}}$, $H=\bigcap\limits_{i=1}^{\rho(Y)}H_{i}$, $X$ is $G$-equivariantly isomorphic to $\prod\limits_{i=1}^{\rho(Y)}X_{i}$, and this isomorphism extends the isomorphism $G/H\rightarrow\prod\limits_{i=1}^{\rho(Y)}G/H_{i}$.  By Proposition \ref{Picard number horospherical} and Corollary \ref{correspondence rays and divisors horospherical Q-factorial}, the Picard numbers $\rho(X)=\rho(Y)$ and $\rho(X_{i})=\rho(Y_{i})=1$. The conclusion follows.
\end{proof}

\begin{lem} \label{product of colored fans implies product of varieties}
Let $X$ be a smooth horospherical $G/H$-embedding such that $H\supseteq R_{u}(B)$, $N_{G}(H)=P_{I}$, $\mathfrak{D}(G/H)=\mathfrak{D}_{X}$ and $\mathbb{F}_{X}=\prod\limits_{i=1}^{m}\mathbb{F}_{i}$. Then there exists a smooth horospherical $G/H_{i}$-embedding $X_{i}$ for each $1\leq i\leq m$ such that $H_{i}\supseteq H$, $\mathfrak{D}(G/H_{i})=\mathfrak{D}_{X_{i}}$, $\mathbb{F}_{X_{i}}=\mathbb{F}_{i}$, $H=\bigcap\limits_{i=1}^{m}H_{i}$, the natural morphism $G/H\rightarrow\prod\limits_{i=1}^{m}G/H_{i}$ is an isomorphism, and this isomorphism can be extended to a $G$-equivariant isomorphism $X\rightarrow \prod\limits_{i=1}^{m}X_{i}$.
\end{lem}

\begin{proof}
Step 1. In this step, we will construct a horospherical $G/H_{i}$-embedding $X_{i}$ for each $i$ such that $\mathbb{F}_{X_{i}}=\mathbb{F}_{i}$ .

Let $\mathfrak{D}_{i}=\{D\in\mathfrak{D}_{X}\mid \rho(\nu_{D})\in (N_{i})_{\mathbb{Q}}\}$ and $I_{i}=S\backslash\mathfrak{D}_{i}$. The fact $\mathbb{F}_{X}=\prod\limits_{i=1}^{m}\mathbb{F}_{i}$ implies that $\mathfrak{D}_{X}$ is the disjoint union of those $\mathfrak{D}_{i}$, where $1\leq i\leq m$. Thus, $\bigcap\limits_{i=1}^{m}I_{i}=I$, $\bigcap\limits_{i=1}^{m}P_{I_{i}}=P_{I}$ and if $i_{1}\neq i_{2}$, then $I_{i_{1}}\cup I_{i_{2}}=S$.

Let $M_{i}=\text{Hom}(N_{i}, \mathbb{Z})$. Then $M_{G/H}=\bigoplus\limits_{i=1}^{m}M_{i}$. Recall that $\chi(P_{I})=\{\chi\in\chi(B)\mid \langle \chi, \alpha^{\vee} \rangle=0 \text{ for all } \alpha\in I\}$ and $\chi(P_{I_{i}})=\{\chi\in\chi(B)\mid \langle \chi, \alpha^{\vee} \rangle=0 \text{ for all } \alpha\in I_{i}\}$.  Take any $\chi_{1}\in M_{i}$, then by Remark \ref{D=S-I, rho(D_a)=a*}$(ii)$, $\chi_{1}\in M_{G/H}\subseteq\chi(P_{I})$. Moreover, for any $\alpha_{1}\in I_{i}\backslash I$, there exists some $j\neq i$ such that $\alpha_{1}\in S\backslash I_{j}$. Thus, $\rho(\nu_{D_{\alpha_{1}}})\in (N_{j})_{\mathbb{Q}}\subseteq M_{i}^{\bot}\cap (N_{X})_{\mathbb{Q}}$. In particular, $\langle \chi_{1}, \alpha_{1}^{\vee} \rangle=0$, i.e. $\chi_{1}\in\chi(P_{I_{i}})$. Hence, $M_{i}\subseteq\chi(P_{I_{i}})\cap M_{G/H}$. Let $H_{i}=\text{Ker}_{P_{I_{i}}}M_{i}$. Then $H_{i}\supseteq\text{Ker}_{P_{I}}M_{G/H}=H\supseteq R_{u}(B)$. By Remark \ref{D=S-I, rho(D_a)=a*}$(ii)$ and the definition of $H_{i}$, we get that $N_{G}(H_{i})=P_{I_{i}}$, $M_{G/H_{i}}=\chi(P_{I_{i}}/H_{i})$ and $H_{i}=\text{Ker}_{P_{I_{i}}}M_{G/H_{i}}$. Hence, $M_{i}$ is a subgroup of $M_{G/H_{i}}$ with a finite index. So $(M_{i})_{\mathbb{Q}}=(M_{G/H_{i}})_{\mathbb{Q}}$ as subspaces of $(M_{G/H})_{\mathbb{Q}}$. By the decomposition $M_{G/H}=\bigoplus\limits_{i=1}^{m}M_{i}$, we can know that $M_{G/H_{i}}\subseteq(M_{i})_{\mathbb{Q}}\cap M_{G/H}=M_{i}$. Hence, $M_{G/H_{i}}=M_{i}$ and $N_{G/H_{i}}=N_{i}$.

Consider the following commutative diagram:
\begin{eqnarray*}
\xymatrix{
G/H\ar[r]^-{\phi_{i}}\ar[d]&G/H_{i}\ar[d]\\
G/P_{I}\ar[r]&G/P_{I_{i}}.
}
\end{eqnarray*}
By Remark \ref{D=S-I, rho(D_a)=a*}$(iii)$, $\mathfrak{D}(G/H_{i})$ can be identified with $S\backslash I_{i}$. Denote by $D_{\alpha, i}$ the $B$-stable divisor on $G/H_{i}$ corresponding to the simple root $\alpha$. Thus, $\phi_{i}^{-1}(D_{\alpha, i})=D_{\alpha}\in \mathfrak{D}(G/H)=S\backslash I$ and $\rho_{G/H}(\nu_{D_{\alpha}})=\rho_{G/H_{i}}(\nu_{D_{\alpha, i}})\in (N_{G/H_{i}})_{\mathbb{Q}}$. Thus, $\phi_{i}^{-1}: \mathfrak{D}(G/H_{i})\rightarrow \mathfrak{D}_{i}, D\mapsto\phi_{i}^{-1}(D)$ is a well-defined bijection. Hence, $\mathbb{F}_{i}$ is a colored fan in $(N_{i})_{\mathbb{Q}}=(N_{G/H_{i}})_{\mathbb{Q}}$. By Theorem \ref{correspondence fans and varieties}, there exists a horospherical $G/H_{i}$-embedding $X_{i}$ such that $\mathbb{F}_{X_{i}}=\mathbb{F}_{i}$.

\medskip

Step 2. Claim: $H=\bigcap\limits_{i=1}^{m}H_{i}$ and $\prod\limits_{i=1}^{m}X_{i}$ is a horospherical $G/H$-embedding.

We know from Remark \ref{D=S-I, rho(D_a)=a*}$(ii)$ that $\bigcap\limits_{i=1}^{m}H_{i}=\bigcap\limits_{i=1}^{m}(\text{Ker}_{P_{I_{i}}}M_{G/H_{i}})=\text{Ker}_{\bigcap\limits_{i=1}^{m} P_{I_{i}}}(\bigcup\limits_{i=1}^{m} M_{G/H_{i}})=\text{Ker}_{\bigcap\limits_{i=1}^{m}P_{I_{i}}}(\bigoplus\limits_{i=1}^{m} M_{G/H_{i}})=\text{Ker}_{P_{I}}(M_{G/H})=H$.

In fact, on each connected component $\Gamma$ of the Dynkin diagram of $G$, there is at most one $1\leq i\leq m$ such that $(S\backslash I_{i})\cap\Gamma\neq\emptyset$. Otherwise, assume that $\alpha_{1}, \alpha_{2}\in\Gamma$ are two simple roots such that $\alpha_{1}\in S\backslash I_{i_{1}}, \alpha_{2}\in S\backslash I_{i_{2}}, i_{1}\neq i_{2}$ and there is no simple roots in $S\backslash I=\bigcup\limits_{i}S\backslash I_{i}$ between $\alpha_{1}$ and $\alpha_{2}$ on $\Gamma$. Since $S\backslash I_{i}$ is bijectively corresponding to $\mathfrak{D}_{X_{i}}$ for each $i$, there are maximal dimensional colored cones $(\mathfrak{C}_{1}, \mathfrak{D}_{1})\in\mathbb{F}_{X_{i_{1}}}, (\mathfrak{C}_{2}, \mathfrak{D}_{2})\in\mathbb{F}_{X_{i_{2}}}$ such that $D_{\alpha_{1}}\in\mathfrak{D}_{1}, D_{\alpha_{2}}\in\mathfrak{D}_{2}$. The fact $\mathbb{F}_{X}=\prod\limits_{i}\mathbb{F}_{X_{i}}$ implies that there is a maximal dimensional colored cone $(\mathfrak{C}, \mathfrak{D})\in\mathbb{F}_{X}$ such that $(\mathfrak{C}_{1}, \mathfrak{D}_{1}), (\mathfrak{C}_{2}, \mathfrak{D}_{2})$ are colored faces of $(\mathfrak{C}, \mathfrak{D})\in\mathbb{F}_{X}$ under the natural inclusions $\mathbb{F}_{X_{i}}\subseteq\mathbb{F}_{X}$. This is contradicted with \cite[Thm. 2.6]{Pa06}. Then by Proposition \ref{when will a product of horo. var. be horo.}$(ii)(iv)(v)$, the claim holds, and the natural morphism $\pi: G/H\rightarrow\prod\limits_{i=1}^{m}G/H_{i}$ is an isomorphism.

By Proposition \ref{colored fan of product is product of colored fans} and Theorem \ref{correspondence fans and varieties}, $X$ is $G$-equivariantly isomorphic to $\prod\limits_{i=1}^{m}X_{i}$, and the isomorphism extends $\pi$.
\end{proof}

\begin{cor} \label{nef1=psef1 horospherical description}
Let $X$ be a smooth projective horospherical $G/H$-embedding such that $\text{Nef}^{\, 1}(X)=\text{Psef}^{\, 1}(X)$. Then there exists a $G$-equivariant morphism $\pi: X\rightarrow G/P_{S\backslash \mathfrak{D}_{0}(G/H)}$ such that each fiber is isomorphic to the product of some smooth projective $L_{0}$-horospherical varieties of Picard number one, where $L_{0}=P_{S\backslash \mathfrak{D}_{0}(G/H)}\cap P_{S\backslash \mathfrak{D}_{0}(G/H)}^{-}$.
\end{cor}

\begin{proof}
By Remark \ref{D=S-I, rho(D_a)=a*}$(v)$, we can assume that $H\supseteq R_{u}(B^{-})$. Let $D_{0}=\sum\limits_{D\in\mathfrak{D}_{0}(G/H)}D$. Then by Theorem \ref{morphism to G/P_0}$(i)$ and Theorem \ref{morphism to G/P_0 horospherical}$(i)$, $D_{0}$ induces a $G$-equivariant morphism $\pi_{0}: X\rightarrow G/P_{S\backslash\mathfrak{D}_{0}(G/H)}$. Denote by $\tilde{x}_{0}$ the point in $G/P_{S\backslash\mathfrak{D}_{0}(G/H)}$ such that $G_{\tilde{x}_{0}}=P_{S\backslash\mathfrak{D}_{0}(G/H)}^{-}$. Let $X_{0}=\pi_{0}^{-1}(\tilde{x}_{0})$, $L_{0}=P_{S\backslash\mathfrak{D}_{0}(G/H)}\cap P_{S\backslash\mathfrak{D}_{0}(G/H)}^{-}$, and $H_{0}=H\cap L_{0}$. Then by Lemma \ref{morphism to G/P_0 then F_X <= F_(X_0)}$(f)$ and Theorem \ref{morphism to G/P_0 horospherical}$(i)$, $X_{0}$ is a complete horospherical $L_{0}/H_{0}$-embedding such that $\mathfrak{D}_{0}(L_{0}/H_{0})=\emptyset$ and $H_{0}\supseteq R_{u}(B^{-})$, where $B_{0}=B\cap L_{0}$. By Theorem \ref{morphism to G/P_0}$(ii)(iii)$, $X_{0}$ is a smooth projective variety such that $\text{Nef}^{\, 1}(X_{0})=\text{Psef}^{\, 1}(X_{0})$. By Theorem \ref{Nef1=Psef1 horospherical is more or less product}, $X_{0}$ is isomorphic to the product of some smooth projective $L_{0}$-horospherical varieties of Picard number one. Note that different fibers of $\pi_{0}$ are isomorphic to each other. The conclusion follows.
\end{proof}

\begin{rmk}
The smooth projective $G$-horospherical varieties of Picard number one have been classified in \cite{Pa09}. Thus, by Theorem \ref{morphism to G/P_0}, Lemma \ref{morphism to G/P_0 then F_X <= F_(X_0)}, Theorem \ref{morphism to G/P_0 horospherical}, Theorem \ref{converse of morphism to G/P_0 horospherical}, Proposition \ref{converse of Nef1=Psef1 horospherical}, and Theorem \ref{Nef1=Psef1 horospherical is more or less product}, we can give a complete characterization of the smooth projective $G$-horospherical varieties whose effective divisors are nef.
\end{rmk}

\begin{e.g.} \label{example}
Let $M$ be a complex vector space of dimension $m+2\geq 5$, and $L, E$ be subspaces such that $\text{dim}_{\mathbb{C}}(L)=1$, $\text{dim}_{\mathbb{C}}(E)=m+1$ and $M=L\oplus E$. Let $X$ be the set of pairs $(W, V)\in Gr(k-1, E)\times Gr(k+1, M)$ such that $W\subseteq E\cap V$, where $2\leq k\leq m-1$. Associate $X$ with the reduced closed subvariety structure. Take $e_{0}\in L\backslash\{0\}$ and $\{e_{1}, e_{2},\ldots,e_{m+1}\}$ to be a basis of $E$.  Let $G=SL_{m+1}(\mathbb{C})$, consider the natural $G$-action on $E$ and let $G$ acts trivially on $L$. Thus, $X$ is naturallly a $G$-variety. This variety will help us to understand the characterization of smooth projective horospherical varieties whose effective divisors are nef. More precisely, we have the following Proposition \ref{properties of the example}.
\end{e.g.}

\begin{prop} \label{properties of the example}
Keep notations as in Example \ref{example}. Then the following hold.

$(i)$ $X$ is a smooth projective $G$-horospherical variety of Picard number two such that $\text{Nef}^{\, 1}(X)=\text{Psef}^{\, 1}(X)$.

$(ii)$ $X$ is not homogeneous under the action of $\text{Aut}^{o}(X)$, where $\text{Aut}^{o}(X)$ is the connected component of the automorphism group of $X$ containing the identity.

$(iii)$ $X$ is not isomorphic to a nontrivial product of two varieties.

$(iv)$ Let $D_{0}=\sum\limits_{D\in\mathfrak{D}_{0}(G/H)}D$, where $H=G_{x}\supseteq R_{u}(B^{-})$ for some point $x$ in the open $B$-orbit. Then $D_{0}$ induces a $G$-equivariant morphism $\pi_{0}: X\rightarrow G/P_{0}^{-}$. Let $F$ be a fiber of $\pi_{0}$, then $X$ is not isomorphic to $G/P_{0}^{-}\times F$.
\end{prop}

Keep notations as above. Take $x_{i}=(W_{0}, V_{i})$ for $i=0, 1, 2$, where $w_{0}=e_{1}\wedge\ldots\wedge e_{k-1}$, $v_{0}=e_{1}\wedge\ldots\wedge e_{k}\wedge(e_{0}+e_{k+1})$, $v_{1}=e_{1}\wedge\ldots\wedge e_{k}\wedge e_{0}$, and $v_{2}=e_{1}\wedge\ldots\wedge e_{k}\wedge e_{k+1}$ are the corresponding representatives under the Pl$\ddot{u}$ker coordinates. Let $X_{i}=G\cdot x_{i}$ and $H_{i}=G_{x_{i}}$. Let $S=\{\alpha_{1},\ldots,\alpha_{m}\}$, where $\alpha_{i}$ is the $i$-th simple root of $G$ by the standard notations. Let $\omega_{i}$ be the $i$-th fundamental dominant weight of $G$.

\begin{lem}
Keep the notations as above, then

$(i)$ $X$ is the disjoint union of $X_{0}$, $X_{1}$ and $X_{2}$;

$(ii)$ the isotropy groups $H_{0}=\text{Ker}_{P_{S\backslash\{\alpha_{k-1}, \alpha_{k}, \alpha_{k+1}\}}}(\omega_{k}-\omega_{k+1})$, $H_{1}=P_{S\backslash\{\alpha_{k-1}, \alpha_{k}\}}$, and $H_{2}=P_{S\backslash\{\alpha_{k-1}, \alpha_{k+1}\}}$.

$(iii)$ $X$ is a projective horospherical $G/H_{0}$-embedding.
\end{lem}

\begin{proof}
Take any point $x=(W, V)\in X$, where $w=e'_{1}\wedge\ldots\wedge e'_{k-1}$ and $v=e'_{1}\wedge\ldots\wedge e'_{k-1}\wedge e'_{k}\wedge e'_{k+1}$ are the corresponding  Pl$\ddot{u}$ker coordinates. If $V\subseteq E$, then there is an element $g\in G$ such that $ge_{i}=\lambda_{i}e'_{i}, \lambda_{i}\in\mathbb{C}^{*}, i=1,\ldots, k+1$, hence $x\in X_{2}$. If $L\subseteq V$, then we can assume that $e'_{k}\in E$ and $e'_{k+1}=e_{0}$. There is an element $g\in G$ such that $ge_{i}=\lambda_{i}e'_{i}, \lambda_{i}\in\mathbb{C}^{*}, i=1,\ldots, k$, hence $x\in X_{1}$.

If $L\nsubseteq V$ and $V\nsubseteq E$, then we can assume that $e'_{k}\in E$ and $e'_{k+1}=e_{0}+t\tilde{e}_{k+1}$, where $t\in\mathbb{C}^{*}$ and $\tilde{e}_{k+1}\in E$. There is an element $g'\in G$ such that $g'e_{i}=\lambda_{i}e'_{i}, \lambda_{i}\in\mathbb{C}^{*}, i=1,\ldots,k$ and $g'e_{k+1}=\lambda_{k+1}\tilde{e}_{k+1}$, $\lambda_{k+1}\in\mathbb{C}^{*}$. This implies that there is a $g\in G$ such that $ge_{i}=\lambda_{i}e'_{i}, i=1,\ldots,k-1$, $ge_{k}=\lambda_{k}\lambda_{k+1}t^{-1}e'_{k}$ and $ge_{k+1}=t\tilde{e}_{k+1}$. Thus, $x\in X_{0}$. This shows that $X=X_{0}\cup X_{1}\cup X_{2}$.

Take any element $h\in G$. Then $h\in H_{i}$ if and only if $hw_{0}=\mu w_{0}$ for some $\mu\in\mathbb{C}^{*}$ and $hv_{i}=\mu_{i}v_{i}$ for some $\mu_{i}\in\mathbb{C}^{*}$. And $hw_{0}=\mu w_{0}$ for some $\mu\in\mathbb{C}^{*}$ if and only if $h\in P_{S\backslash\{\alpha_{k-1}\}}$. On the other hand, $v_{0}=v_{1}+v_{2}$. And $hv_{1}=\mu_{1} v_{1}$ for some $\mu_{1}\in\mathbb{C}^{*}$ if and only if $h\in P_{S\backslash\{\alpha_{k}\}}$. What's more, in this case, $\mu_{1}=\omega_{k}(h)$. Similarly, $hv_{2}=\mu_{2} v_{2}$ for some $\mu_{2}\in\mathbb{C}^{*}$ if and only if $h\in P_{S\backslash\{\alpha_{k+1}\}}$. And in this case, $\mu_{2}=\omega_{k+1}(h)$. Thus, $hv_{0}=\mu_{0}v_{0}$ for some $\mu_{0}\in\mathbb{C}^{*}$ if and only if $hv_{1}=\mu_{1} v_{1}$, $hv_{2}=\mu_{2} v_{2}$ for some $\mu_{1}=\mu_{2}\in\mathbb{C}^{*}$, i.e. if and only if $h\in\text{Ker}_{P_{S\backslash\{\alpha_{k}, \alpha_{k+1}\}}}(\omega_{k}-\omega_{k+1})$. This implies $(ii)$ and the fact $X_{i}\cap X_{j}=\emptyset$ if $i\neq j$, i.e. $(i)$ holds too.

Since $H_{0}\supseteq R_{u}(B)$, $X_{0}=G/H_{0}$ is a homogeneous $G$-horospherical variety. The fact $\text{dim}(X_{i})\leq\text{dim}(X)-2$ for $i=1, 2$ implies that $X$ is a normal variety. Thus, $X$ is a horospherical $G/H_{0}$-embedding. Since $X$ is a closed subvariety of $Gr(k-1, E)\times Gr(k+1, M)$, $X$ is projective. Thus, $(iii)$ holds.
\end{proof}

\begin{proof}[Proof of Proposition \ref{properties of the example}]
Let $X\xrightarrow{\pi_{1}}Y_{1}\subseteq Gr(k-1, E)$ and $X\xrightarrow{\pi_{2}}Y_{2}\subseteq Gr(k+1, M)$ be the restrictions of the two projections from $Gr(k-1, E)\times Gr(k+1, M)$ to its factors, where $Y_{1}, Y_{2}$ are the images. Then $Y_{1}=Gr(k-1, E)$ and each fiber $\pi_{1}^{-1}\pi_{1}(x)$ is isomorphic to $Gr(2, M/V)$, where $x=(W, V)$. Thus, $\pi_{1}$ is a smooth morphism, and $X$ is a smooth variety.

The construction of $X$ implies that $Y_{2}=Gr(k+1, M)$. For any $x=(W, V)\in X_{0}\cup X_{1}$, the fiber $\pi_{2}^{-1}\pi_{2}(x)\cong Gr(k-1, E\cap V)\cong\mathbb{P}^{k-1}$. For any $X=(W, V)\in X_{2}$,  the fiber $\pi_{2}^{-1}(\pi_{2}(x))\cong Gr(k-1, E\cap V)=Gr(k-1, V)\cong Gr(2, V)$. By \cite[Prop. 2.1]{Br11}, $X_{2}$ is stable under the action of $\text{Aut}^{0}(X)$, which implies the conclusion $(ii)$.

By our previous discussions, we know that $G/H_{0}$ is the open $G$-orbit on $X$,  $M_{G/H}=\mathbb{Z}(\omega_{k}-\omega_{k+1})$, $\mathcal{V}_{X}=\emptyset$, $\mathfrak{D}(G/H_{0})=\{D_{\alpha_{k-1}},D_{\alpha_{k}},D_{\alpha_{k+1}}\}$ , $\mathfrak{D}_{0}(G/H_{0})=\{D_{\alpha_{k-1}}\}$, $\rho(\nu_{D_{\alpha_{k}}})=-\rho(\nu_{D_{\alpha_{k+1}}})$, $\mathfrak{C}_{X_{1}}^{c}=(\mathbb{Q}^{+}\rho(\nu_{D_{\alpha_{k+1}}}), \{D_{\alpha_{k+1}}\})$ and $\mathfrak{C}_{X_{2}}^{c}=(\mathbb{Q}^{+}\rho(\nu_{D_{\alpha_{k}}}), \{D_{\alpha_{k}}\})$. Moreover, $X$ is a smooth projective horospherical $G/H_{0}$-embedding.

We know from Theorem \ref{morphism to G/P_0 horospherical}$(i)$ that both $G/P_{0}^{-}$ and $F$ have positive dimensions. Hence, the conclusion $(iv)$ follows from the conclusion $(iii)$.

By Proposition \ref{cycles are rat. equiv. to stable ones, picard group on spherical varieties}$(i)(ii)$ and \cite[Cor. 1.3$(iv)$]{Br93}, $D_{\alpha_{k}}$ and $D_{\alpha_{k+1}}$ are numerically equivalent in $N^{1}(X)_{\mathbb{R}}$, the Picard number $\rho(X)=2$, and $\mathbb{R}^{+}D_{\alpha_{k-1}}$, $\mathbb{R}^{+}D_{\alpha_{k}}$ are exact the two extremal rays of $\text{Psef}^{\, 1}(X)$.

By Theorem \ref{morphism to G/P_0}$(i)$ and Theorem \ref{morphism to G/P_0 horospherical}$(i)$, $D_{\alpha_{k-1}}$ induces a $G$-equivariant morphism $\pi_{0}: X\rightarrow G/P_{S\backslash\{\alpha_{k-1}\}}^{-}=Y_{1}$. By Lemma \ref{parabolic containing R_u(B) is unique}, $\pi_{1}|_{X_{0}}=\pi_{0}|_{X_{0}}$. Note that the morphisms of colored fans corresponding to $\pi_{1}$ and $\pi_{0}$ are both $\mathbb{F}_{X}\rightarrow\{(pt, \emptyset)\}$. By Theorem \ref{morphism of fan}, $\pi_{1}=\pi_{0}$. By Theorem \ref{morphism to G/P_0}$(iii)$, $\text{Nef}^{\, 1}(X)=\text{Psef}^{\, 1}(X)$. Thus, $\pi_{1}$ is a Mori contraction of $X$ and the conclusion $(i)$ holds.

On the other hand, $\pi_{2}^{*}\mathcal{O}_{Y_{2}}(1)$ is not ample and $\pi_{2}$ doesn't factor through $\pi_{1}$. This implies that $\pi_{2}$ factors through the Mori contraction different from $\pi_{1}$. The facts $Y_{2}$ is smooth,$\pi_{2}$ has connected fibers and the Picard number $\rho(X)=2$ imply that $\pi_{2}$ is indeed the other Mori contraction of $X$.

Suppose that $X$ is isomorphic to a nontrivial product $X_{1}\times X_{2}$, then the fact $\rho(X)=2$ implies that $\pi'_{1}: X\rightarrow X_{1}$ and $\pi'_{2}: X\rightarrow X_{2}$ are exact the two Mori contractions. Thus, by reordering $X_{1}$ and $X_{2}$ if necessary, $X_{1}=Y_{1}$ and $X_{2}=Y_{2}$. It's contradicted with the fact that $\text{dim}(X)<\text{dim}(Y_{1})+\text{dim}(Y_{2})$. Hence, the conclusions $(iii)$ and $(iv)$ hold.
\end{proof}

\small

INSTITUTE OF MATHEMATICS, AMSS, CHINESE ACADEMY OF SCIENCES,

\smallskip

55 ZHONGGUANCUN EAST ROAD, BEIJING, 100190, P.R.CHINA

\smallskip

E-mail address: qifengli@amss.ac.cn

\end{document}